\documentclass[12pt]{amsart}
\usepackage{amscd}      
\usepackage{amssymb}
\usepackage{rotating} 
 \usepackage{tikz}
\usetikzlibrary{patterns,snakes}

\usepackage{lscape}
\usepackage{latexsym, amsmath, amsthm, graphics, amsxtra, pb-diagram}

\usepackage{xypic}      
\LaTeXdiagrams          
\usepackage[all]{xy}
\xyoption{2cell} \UseAllTwocells \xyoption{frame} \CompileMatrices
\allowdisplaybreaks[3]

\usepackage{latexsym}
\usepackage{epsfig}
\usepackage{amsfonts}
\usepackage{enumerate}
\usepackage{times}

\theoremstyle{theorem}
\newtheorem{thm}{Theorem}[section]
\newtheorem{cor}[thm]{Corollary}
\newtheorem{lem}[thm]{Lemma}

\newtheorem{prop}[thm]{Proposition}

\theoremstyle{definition}
\newtheorem{defn}[thm]{Definition}
\newtheorem{rmk}[thm]{Remark}
\newtheorem{numrmk}[thm]{Remark}

\newtheorem{exm}[thm]{Example}
\newtheorem{conds}[thm]{Condition}

\newtheorem{example}[thm]{Example}

\def\Label{\label}

\newtheorem*{thm*}{Theorem}

\numberwithin{equation}{section}

\newtheorem{theorem}[subsection]{Theorem}

\theoremstyle{definition}

\theoremstyle{remark}

\theoremstyle{remark}

\numberwithin{equation}{section}

\newcommand{\M}{\mathcal{M}}
\newcommand{\Mbar}{\overline{\M}}

\newcommand{\tJ}{\tilde{J}}

\newcommand{\bz}{\mathbf{z}}

\newcommand{\ms}{\mathfrak s}

\def\<{\left\langle}
\def\>{\right\rangle}

\newcommand{\Wkp}{\mathcal{W}^{k,p}}

\newcommand{\E}{\mathcal{E}}
\newcommand{\B}{\mathcal{B}}
\newcommand{\X}{\mathcal{X}}
\newcommand{\Y}{\mathcal{Y}}

\newcommand{\F}{\mathcal{F}}

\newcommand{\gG}{\mathcal{G}}
\newcommand{\gH}{\mathcal{H}}

\newcommand{\twoarrows}{\rightrightarrows}

\newcommand{\Man} {\textsf{Man}}

\newcommand{\St} {\textsf{St}}

\newcommand{\sfC}{{\textsf C}}
\newcommand{\sfD}{{\textsf D}}

\newcommand{\XxS}{\X\times S^{2}}
\newcommand{\com}{\mathbb{C}}
\newcommand{\ration}{\mathbb{Q}}
\newcommand{\real}{\mathbb{R}}
\newcommand{\inte}{\mathbb{Z}}
\newcommand{\Ztwo}{\mathbb{Z}_{2}}

\def\darr#1{\raise1.5ex\hbox{$\leftrightarrow$}
\mkern-16.5mu #1}

\def\roughly#1{\raise.3ex\hbox{$#1$\kern-.75em
\lower1ex\hbox{$\sim$}}}

\def\opname#1{\mathop{\kern0pt{\rm #1}}\nolimits}

\def\Aut{\operatorname{Aut}}

\def\ms{\mathfrak{s}}

\title{Seidel Representation for Symplectic Orbifolds}
\author{Hsian-Hua Tseng}
\address{Department of Mathematics\\ Ohio State University\\ 100 Math Tower\\ 231 West 18th Avenue\\ Columbus, OH 43210-1174\\ USA}
\email{hhtseng@math.ohio-state.edu}

\author{Dongning Wang}
\address{Department of Mathematics \\ University of Wisconsin-Madison\\ Van Vleck Hall\\ 480 Lincoln Dr.\\ Madison, WI 53706\\ USA}
\email{dwang@math.wisc.edu}

\date{\today}

\begin{document}
\begin{abstract}
Let $(\X,\omega)$ be a compact symplectic orbifold. We define $\pi_1(Ham(\X, \omega))$, the fundamental group of the 2-group of Hamiltonian diffeomorphisms of $(\X, \omega)$, and construct a group homomorphism from $\pi_1(Ham(\X, \omega))$  to the group $QH_{orb}^*(\X,\Lambda)^{\times}$ of multiplicatively invertible elements in the orbifold quantum cohomology ring of $(\X, \omega)$. This extends the Seidel representation (\cite{Se}, \cite{M}) to symplectic orbifolds. 
\end{abstract}
\maketitle
\tableofcontents

\section{Introduction}
\subsection*{Results}
The purpose of this paper is to construction Seidel representation for compact symplectic orbifolds. Let $(\X, \omega)$ be a compact symplectic orbifold and let $Ham(\X,\omega)$ be the group\footnote{This is in fact a $2$-group.} of Hamiltonian diffeomorphisms of $(\X, \omega)$. We develop a suitable notion of loops of Hamiltonian diffeomorphisms of $(\X, \omega)$ and consider the set $$\pi_1(Ham(\X,\omega))$$ of homotopy classes of (based) loops of Hamiltonian diffeomorphisms of $(\X,\omega)$. Composition of loops gives $\pi_1(Ham(\X,\omega))$ a natural group structure. 

By the work of Chen-Ruan \cite{CR2}, one can define Gromov-Witten invariants of a compact symplectic orbifold $(\X,\omega)$. Roughly speaking, Gromov-Witten invariants of $(\X,\omega)$ are virtual counts of pseudo-holomorphic maps from nodal orbifold Riemann surfaces to $\X$ satisfying certain incidence conditions. A brief review of Gromov-Witten invariants can be found in Section \ref{secReviewoGW}. In this paper we are mainly concerned with genus $0$ Gromov-Witten invariants of $(\X, \omega)$. These invariants can be used to defined an alternate product structure called the {\em quantum cup product} on the cohomology group $H^*(I\X)$ of the {\em inertia orbifold} $I\X$ of $\X$. The resulting ring is called the {\em quantum orbifold cohomology ring} of $(\X, \omega)$ and is denoted\footnote{The dependence in $\omega$ is usually suppressed in the notation. We also omit the Novikov ring in the notation.} by $QH^*_{orb}(\X)$. 

Let $QH^*_{orb}(\X)^\times$ be the set of multiplicatively invertible elements in $QH^*_{orb}(\X)$. The set $QH^*_{orb}(\X)^\times$ is a group under the quantum cup product. The {\em Seidel representation} of $(\X,\omega)$ is a group homomorphism
\begin{equation}\label{seidel_representation}
\mathcal{S}: \pi_1(Ham(\X,\omega))\to QH^*_{orb}(\X)^\times.
\end{equation}
The definition of $\mathcal{S}$ may be briefly outlined as follows. Given an element $a\in \pi_1(Ham(\X,\omega))$ represented by a Hamiltonian loop $\alpha$, we consider a fiber bundle\footnote{In fact this is an example of an {\em orbifiber bundle}, in the sense of Definition \ref{orbifiberbundle}.} $\mathcal{E}_\alpha\to S^2$ which roughly speaking is obtained by gluing two copies of $D^2\times \X$ along the boundary $\partial D^2\times \X=S^1\times \X$ using $\alpha$. The total space $\mathcal{E}_\alpha$ can be equipped with a symplectic structure which is compatible with the fiber bundle structure. We define an element $$\mathcal{S}(a)\in QH^*_{orb}(\X)^\times$$ using genus $0$ Gromov-Witten invariants of $\mathcal{E}_\alpha$. See Definition \ref{defseidel} for the precise definition. 

The following is the main result of the paper.
\begin{theorem}[= Corollary \ref{seidel_hom}]
The map $$\mathcal{S}: \pi_1(Ham(\X,\omega))\to QH^*_{orb}(\X)^\times$$ is a group homomorphism.
\end{theorem}
To prove that $\mathcal{S}$ is a group homomorphism, it suffices to prove the following
\begin{theorem}[see Theorem \ref{SeidelAxiom}]
\hfill
\begin{enumerate}
\item
Let $e\in \pi_1(Ham(\X, \omega))$ be the identity loop, then 
\begin{equation}\label{S(e)}
\mathcal{S}(e)=1\in QH^*_{orb}(X)^\times.
\end{equation}

\item Let $a,b\in \pi_1(Ham(\X,\omega))$ and let $a\cdot b\in \pi_1(Ham(\X, \omega))$ be their product defined by loop composition. Then
\begin{equation}\label{seidel_composition}
\mathcal{S}(a\cdot b)=\mathcal{S}(a)* \mathcal{S}(b)
\end{equation}
where $*$ denote the quantum cup product of $(\X,\omega)$. 
\end{enumerate}
\end{theorem} 

The proof of the above Theorem involves detailed analysis of Gromov-Witten invariants and are highly technical. In this paper, we work with the definition of Gromov-Witten invariants using Kuranishi structures (as introduced in \cite{FO}). To prove (\ref{S(e)}), we first reduce it to the vanishing of certain Gromov-Witten invariants, see Proposition \ref{ModuliofTrivial}. We then show the needed vanishing by showing that the relevant virtual fundamental cycles have dimensions less than their expected dimensions. The arguments are quite technical since all these must be done in the context of Kuranishi spaces. See Section \ref{trivaxiom} for details. 

The proof of (\ref{seidel_composition}), given in Section \ref{compaxiom}, essentially amounts to comparing certain Gromov-Witten invariants of $\E_{a\cdot b}$ with Gromov-Witten invariants of $\E_a$ and $\E_b$. We do this by a degeneration type argument. The main geometric ingredient here is a bigger symplectic fibration that contains both $\E_{a\cdot b}$ and the connected sum $\E_a\#_\X\E_b$ as its fibers. 

\subsection*{Motivations}
The homomorphism $\mathcal{S}$ was first constructed by P. Seidel \cite{Se} for a class of symplectic manifolds. The construction was later extended to all symplectic manifolds\footnote{We remark that Gromov-Witten invariants considered in \cite{M} are not defined using the framework of Kuranishi structures. Therefore our work specialized to the manifold case gives a new construction of Seidel representations for symplectic manifolds.} by D. McDuff \cite{M}. The structure of the Hamiltonian diffeomorphism group $Ham(X,\omega)$ of symplectic manifolds $(X,\omega)$ has been an important topic in symplectic topology (cf. \cite{MS}). The Seidel representation $\mathcal{S}$ has been used to study the group $Ham(X,\omega)$ for symplectic manifolds $(X,\omega)$, see e.g. \cite{MT}. 

Our interests in Seidel representations are motivated by another application, which we now explain. Hamiltonian $S^1$-actions naturally give rise to loops of Hamiltonian diffeomorphisms. Since the Seidel representation $\mathcal{S}$ is a group homomorphism, if a collection of loops $a_1,a_2,...,a_k$ compose to the identity loop $e$, i.e. $$a_1\cdot a_2\cdot...\cdot a_k=e,$$ then we have $$\mathcal{S}(a_1)*\mathcal{S}(a_2)*...*\mathcal{S}(a_k)=1,$$ which is a relation in the quantum cohomology ring. This gives an approach to study quantum cohomology rings of symplectic manifolds/orbifolds by Hamiltonian loops. This approach is particularly successful in the case of symplectic toric manifolds, which admits many Hamiltonian loops due to the presence of Hamiltonian tori actions. In \cite{MT}, this approach was carried out to give a presentation of the quantum cohomology ring of compact symplectic toric manifolds. This presentation plays an important role in the mirror theorem for toric manifolds recently proved by Fukaya-Oh-Ohta-Ono \cite{FOOOtoric3}. We aim at studying the quantum orbifold cohomology ring of compact symplectic toric orbifolds by this approach. The construction of Seidel representation for symplectic orbifolds, carried out in this paper, is the first step in this project. In \cite{TW} we will use this approach to study the quantum orbifold cohomology ring of compact symplectic toric orbifolds. 

\subsection*{Outline}The rest of this paper is organized as follows. 

In Section \ref{part1}, we set up the foundation of Hamiltonian loops and Hamiltonian orbifiber bundles which naturally generalize Hamiltonian loop and Hamiltonian fiber bundle in the manifold case. We begin in \S \ref{sec:stack_review} with a brief review of orbifolds using the language of differentiable stacks, then recall in \S \ref{sec:vec_field} differential forms, vector fields and flows on orbifolds/stacks. In \S \ref{hamloop}, we define Hamiltonian loops and their homotopy equivalence, which give rise to the fundamental group of Hamiltonian diffeomorphism group. In \S \ref{orbifiberbundlegeneral}, we give the definition of orbifiber bundle and a special class of orbifold morphisms into such bundle which is called sectional orbifold morphism. Sectional orbifold morphisms will be the map we ``count'' when we define Seidel representation. In \S \ref{orbifiberbundleS2}, we explain the construction of a Hamiltonian orbifiber bundle over $S^{2}$ from a Hamiltonian loop.

Section \ref{part2} contains the construction of Seidel representation for symplectic orbifolds. We begin with a review of Kuranishi structures in \S \ref{secReviewoKura}, and then recall the construction of orbifold Gromov-Witten invariants and quantum cohomology in \S \ref{secReviewoGW}. In \S \ref{secCurvHamBundle}, we collect properties of $J$-holomorphic orbicurves in a Hamiltonian orbifiber bundle over sphere. Then we use the moduli space of such kind of orbicurves to define Seidel representation for symplectic orbifolds.

In Section \ref{trivaxiom}, we prove the triviality property (\ref{S(e)}). The idea is explained in \S \ref{trivideal}. In \S \ref{secParaAction} and \S \ref{secEqKura}, we introduce the tools of parametrized group actions and parametrized equivariant Kuranishi structures to carry out the proof. In particular, parametrized equivariant Kuranishi structures are constructed for the relevant moduli spaces in \S \ref{secEqKura}. Finally  in \S \ref{wholetriv}, we complete the proof by carefully analyzing the Kuranishi structure constructed in \S \ref{secEqKura}.

In Section \ref{compaxiom}, we construct a degeneration of Hamiltonian orbifiber bundles which relates two Hamiltonian orbifiber bundles associated to two Hamiltonian loops with the Hamiltonian orbifiber bundle corresponding to their product, and then use it to prove the composition property (\ref{seidel_composition}).

\section*{Acknowledgment} 
We are in debt to A. Givental for suggesting this problem and helpful discussions. We thank E. Getzler for suggesting the framework of stacks and their automorphism 2-groups, E. Lerman for explaining stratification of circle action on orbifolds, and R. Hepworth for explaining glueing of stack maps. We are grateful to Y.-G. Oh and K. Ono who listened and gave helpful comments to the proof of the two properties. We also thank B. Chen, D. McDuff for helpful discussions and Y. Ruan for suggestions. 

\section{Hamiltonian Loops and Hamiltonian Orbifiber Bundles}\label{part1}
\subsection{A Review of Orbifolds via Deligne-Mumford Stacks}\label{sec:stack_review}

In this subsection, we collect the definitions of orbifolds using the language of stacks following \cite{Le}. 

The classical definition of orbifolds can be reformulated using groupoids since every orbifold atlas in the sense of Satake determines an \'{e}tale Lie groupoid. A morphism between orbifolds should carry more information than just a map defined on charts in order to pullback bundles. Such a map carrying additional information is called a good map in \cite{CR2}. These maps correspond to bibundles or Hilsum-Skandalis maps when orbifolds are viewed as Lie groupoids. Since we need to glue morphisms between orbifolds in this paper, any framework using isomorphism classes of bibundles/Hilsum-Skandalis maps as morphisms between orbifolds will encounter the issue pointed out in \cite[Lemma 3.41]{Le}. Namely, there could be two different morphisms which are the same when restricted to an open cover. Furthermore, we need to study Hamiltonian flows on symplectic orbifolds, thus viewing an orbifold as an \'{e}tale groupoid is inconvenient, because a flow can easily flow out of that groupoid (see \cite{Hep}). 
 We refer the readers to \cite{Le} and \cite{ML} for more detailed reasons why stack is a better language to study orbifolds.

A stack is a category fibered in groupoids satisfying certain glueing conditions. We give the definition of category fibered in groupoids below:

\begin{defn}\label{CFG}
A {\em category fibered in groupoids over a category $\sfC$} is a functor $\pi:\sfD\to \sfC$ such that 
\begin{enumerate}
\item\label{pullob} for an arrow $x' \stackrel{r}{\to} x$ in $\sfC$ and an object $\xi\in Ob(\sfD)$ with $\pi(\xi)=x$ there is an arrow $\xi' \stackrel{g}{\to} \xi$ such that $\pi(g)=r$.
\item\label{pullarr} for a diagram $\vcenter{ \xymatrix@=8pt@ur{
& \xi''\ar[d]^{g_{2}}\\ \xi' \ar[r]_{g_{1}} & \xi}}$ in $\sfD$ and a diagram
$\vcenter{ \xymatrix@=8pt@ur{ & \pi(\xi'')\ar[dl]_{r}\ar[d]^{\pi(g_{2})}\\
\pi(\xi') \ar[r]_{\pi(g_{1})} & \pi(\xi)}}$ in $\sfC$ there is a unique
arrow $g:\xi''\to \xi'$ in $\sfD$ making the diagram $\vcenter{
\xymatrix@=8pt@ur{ & \xi''\ar[d]^{r_{2}}\ar[dl]_{g}\\ \xi'
\ar[r]_{r_{1}} & \xi}}$ commutative and satisfying $\pi (g) = r$.
\end{enumerate}
\end{defn}

\begin{defn}[Fiber of CFG] Let $\pi:\sfD\to \sfC$ be a category fibered
in groupoids and $x\in Ob(\sfC)$ an object.  The {\em fiber} of $\sfD$ over $x$ is the category $\sfD (x)$ with
objects given by
\[
Ob(\sfD (x)) := \{\xi \in Ob(\sfD) \mid \pi (\xi) = x\}\]
 and arrows/morphisms given by
\[
Mor(\sfD (x))(\xi',\xi) :=\{ (g:\xi' \to \xi) \in Mor(\sfD )(\xi', \xi)  \mid \pi (g) = id_x\} \quad \text{for }  \xi, \xi' \in Ob(\sfD (x)).
\]
\end{defn}

Note that in Definition \ref{CFG} (\ref{pullob}), $\xi'$ may not be uniquely determined by $r$ and $\xi$. We make a choice to get unique $\xi'$ for $r$ and $\xi$.
\begin{defn}\label{cleavage}
For a category fibered in groupoids $\pi:\sfD\to \sfC$, $\xi\in Ob(\sfD(x))$ and  $x'\stackrel{r}{\to}x\in Mor(\sfC)$  we choose an arrow $g\in Mor(\sfD)$ with target $\xi$ and $\pi(g)=r$.  We denote the source of $g$ by $r^* \xi$ and refer to it as the {\em pullback of $\xi$ by $r$}.  We always choose $id^*\xi = \xi$. Such a choice is called a {\em cleavage} of $\pi:\sfD\to \sfC$. 
\end{defn}

The glueing conditions for stack are defined using descent categories. 

\begin{defn}\label{Descent}
Let $\Man$ be the category of smooth manifolds with functors being smooth maps. Let $\pi:\sfD\to \Man$ be a category fibered in groupoids over $\Man$, $M$ a manifold and $\mathfrak{U}:=\{U_i\}$ an open cover of $M$. The {\em descent category of } $\mathfrak{U}$, denoted by $\sfD(\mathfrak{U})$, is defined as follows:

\begin{enumerate}
\item An object $(\{\xi_i\}, \{\phi_{ij}\})$ of $\sfD(\mathfrak{U})$ is a collection of objects $\xi_i \in \sfD(U_i)$, together with isomorphisms $\phi_{ij}\colon \mathrm{pr}_2^* \xi_j \simeq \mathrm{pr}_1^* \xi_i$ in $\sfD(U_{ij}) :=\sfD(U_i \times_M U_j)$, satisfying the following cocycle condition: 
for any triple of indices $i$, $j$ and $k$, we have the equality \[
\mathrm{pr}_{13}^* \phi_{ik} = \mathrm{pr}_{12}^* \phi_{ij} \circ
\mathrm{pr}_{23}^* \phi_{jk} \colon \mathrm{pr}_3^*\xi_k \to\mathrm{pr}_1^*\xi_i \]
where $\mathrm{pr}_{ab}$ and $\mathrm{pr}_a$ are projections from $U_{ijk}:=U_{i}\times_{M}U_{j}\times_{M} U_{k}$ to the $a^{th}$ and $b^{th}$ factors, and the $a^{th}$ factor respectively (e.g. $\mathrm{pr}_{13}:U_{ijk}\to U_{ik}$, $\mathrm{pr}_{2}:U_{ijk}\to U_{j}$). 
\item An arrow $\{g_i\} \colon (\{\xi_i\}, \{\phi_{ij}\}) \to (\{\xi'_i\}, \{\psi_{ij}\})$ in $\sfD(\mathfrak{U})$
is a collection of arrows $\xi_i \stackrel{g_{i}}{\to} \xi'_i$ in $\sfD(U_i)$ such that for each pair of indices $i$, $j$, the following diagram is commutative:
   \[
   \xymatrix@C+15pt{
   {}\mathrm{pr}_2^* \xi_j \ar[r]^{\mathrm{pr}_2^* g_j} \ar[d]^{\phi_{ij}}
   & {}\mathrm{pr}_2^* \xi'_j\ar[d]^{\psi_{ij}}\\
   {}\mathrm{pr}_1^* \xi_i \ar[r]^{\mathrm{pr}_1^* g_i}&
   {}\mathrm{pr}_1^* \xi'_i.
   }
   \]
\end{enumerate}
\end{defn}

Note that $\sfD(\mathfrak{U})$ does not depend on the choice of cleavage of $\pi:\sfD\to \Man$ in the sense that different choices of cleavage yield canonically isomorphic descent categories.

For each object $\xi$ of $\sfD(M)$ we can construct an object $(\{\xi_i\}, \{\phi_{ij}\})$ in $\sfD(\mathfrak{U})$ on a covering $\mathfrak{U}:=\{\iota_{i}:U_i\to M\}$ as follows: 
\begin{itemize}
\item The objects are the pullbacks $\iota_{i}^{*}\xi$;
\item The isomorphisms $\phi_{ij} : pr_{2}^{*} \iota_{j}^{*}\xi \simeq pr_{1}^{*} \iota_{i}^{*}\xi$ are the isomorphisms that come from the fact that both $pr_{2}^{*} \iota_{j}^{*}\xi$ and $pr_{1}^{*} \iota_{i}^{*}\xi$ are pullbacks of $\xi$ to $U_{ij}$.
\end{itemize}
 If we identify $pr_{2}^{*} \iota_{j}^{*}\xi$ and $pr_{1}^{*} \iota_{i}^{*}\xi$, as is commonly done, then the $\phi_{ij}$ are identities. For each arrow $\alpha:\xi\to \xi'$ in $\sfD(\mathfrak{U})$, we get arrows $\iota_{i}^{*}\alpha: \iota_{i}^{*}\xi \to \iota_{i}^{*}\xi'$, yielding an arrow from the object with descent associated with $\xi$ to the one associated with $\xi'$. This defines a functor $\sfD(M) \to \sfD(\mathfrak{U})$. 

\begin{defn}\label{stackman}
A category fibered in groupoids $\pi: \sfD\to \Man$ is a {\em stack over manifolds} if for any manifold $M$ and any open cover $\mathfrak{U}$ of $M$ the functor $\sfD(M) \to \sfD(\mathfrak{U})$ defined above is an equivalence of categories.
\end{defn}

\begin{defn}\label{Mapstacksman}
Let $ \pi_\sfC: \sfC \to \Man$, $\pi_\sfD: \sfD \to \Man$ be two stacks. A functor $f: \sfC\to \sfD$ is a {\em map of stacks} (more precisely a 1-arrow in the 2-category $\St$ of stacks) if it commutes with the projections to $\Man$:
$\pi_\sfD \circ f = \pi_\sfC.$ A map of stacks is called an {\em isomorphism} if it is an equivalence of categories.
\end{defn}

Every manifold $M$ determines a stack $\underline{M}$ over manifolds:  the objects of $\underline{M}$ are maps $Y\stackrel{f}{\to} M$ of manifolds into $M$. A morphism in
$\underline{M}$ from $f:Y\to M$ to $f':Y' \to M$ is a map of manifolds $h:Y\to Y'$ making the diagram $\vcenter{\xy (0,6)*+{Y}="1";
(0,-6)*+{Y'}="2"; (12,0)*+{M} ="3"; {\ar@{->}_{h} "1";"2"};
{\ar@{->}_{f'} "2";"3"};{\ar@{->}^{f} "1";"3"};
\endxy }$ commutative. The functor $\underline{M}\to \Man$ is the forgetful map which sends $f:Y\to M$ to $Y$. To keep notation simple, we follow the common practice to drop the underline and write $M$ as the stack. 

\begin{defn}\label{atlas}
A stack $\X$ over manifolds is called a {\em differentiable stack} if there is a manifold $X$ and a morphism $p : X\to \X$ such that:
\begin{enumerate}
\item For all $Y \to \X$ the stack $X \times_{\X} Y$  is a manifold (i.e., equivalent to the stack of some manifold).
\item For all $Y\to \X$, the
projection $Y \times_{\X} X\to Y$ is a surjective submersion.
\end{enumerate}
The map $p: X \to \X$ is then called a {\em covering} or an {\em atlas} of $\X$.

Given an atlas $p: X_{0} \to \X$, $X_{1}:=X_{0} \times_{\X} X_{0}$ is a manifold, the projections to the first and second factor of $X_{0}$ define two maps $s:X_{1}\to X_{0}$ and $t:X_{1}\to X_{0}$. With these two maps, $\gG:=(X_{1}\twoarrows X_{0})$ is a Lie groupoid. Such Lie groupoid is called an atlas groupoid of the stack $\X$. The category $B\gG$ of principal $\gG$-bundles in turn defines a category fibered in groupoids over manifolds, which is equivalent to $\X$ as differentiable stack. 

A differentiable stack $\X$ is called a {\em differentiable Deligne-Mumford stack},  if $\X$ admits an atlas groupoid $X_{1}\twoarrows X_{0}$ which is \'{e}tale (i.e., $s$ and $t:X_1\to X_0$ are local diffeomorphisms) and proper (i.e., $s\times t: X_{1}\to X_{0}\times X_{0}$ is proper). From now on we will refer differentiable Deligne-Mumford stack as DM stack for short.
\end{defn}

\begin{rmk}
\begin{enumerate}
\item Differentiable stacks form a 2-category, with DM stacks forming a sub 2-category.
\item Every differentiable stack $\X$ determines an equivalent class of topological spaces $|\X|$ called the {\em coarse moduli space}: take an atlas $\gG_{\X}=(G_{1}\twoarrows G_{0})$ of $\X$, then 
$$|\gG_{\X}|:=G_{0}/\<x\sim y\ \text{ i.f.f.}\ \exists g\in G_{1}\ s.t.\ s(g)=x,\ t(g)=y\>$$
defines a topological space. Different choices of atlas are Morita equivalent and thus define homeomorphic topological spaces. Every stack map $\pmb{f}:\X\to\Y$ induces a continuous map $|\pmb{f}|:|\X|\to |\Y|$.
\end{enumerate}
\end{rmk}

Different atlas groupoids play the role of different coordinates of a given stack. We collect the dictionary lemmas between groupoids and stacks from \cite[Lemma 2.29--2.31]{BX}. The first lemma says every morphism between atlas groupoids defines a map between the corresponding stacks.

\begin{lem}[First Dictionary Lemma]
Let $\X$ and $\Y$ be differentiable stacks, $\gG_{\X}=(X_{1}\twoarrows X_{0})$ and $\gG_{\Y}=(Y_{1}\twoarrows Y_{0})$ atlas groupoids respectively. Let $\phi:\gG_{\X}\to \gG_{\Y}$ be a morphism of Lie groupoids. Then there exists a morphism of stacks $f:\X\to\Y$ and a
2-isomorphism 
\begin{equation}\label{eta-di}
\vcenter{\xymatrix{
X_0\dto\rto^{\phi_0}\drtwocell\omit{^\eta} & Y_0\dto\\
\X\rto^f&\Y}}
\end{equation}
such that the cube
\begin{equation}\label{cube}
\vcenter{\xymatrix@=.5pc{
& X_1\ddlto\drto \ar[rrr]^{\phi_1} &&& Y_1\ddlto\drto &\\
&& X_0\ddlto\ar[rrr]^(.3){\phi_0} &&& Y_0\ddlto \\
X_0\drto\ar[rrr]^(.3){\phi_0} &&& Y_0\drto &&\\
& \X\ar[rrr]^f &&& \Y &}}
\end{equation}
is 2-commutative.  If $(f',\eta')$ is another pair satisfying these properties,
then there is a unique 2-isomorphism $\theta:f\Rightarrow f'$ such that
$\theta\ast\eta'=\eta$. 
\end{lem}

The second and third Dictionary Lemmas below treat the converse. Note that if we fix atlas groupoids of two differentiable stacks, a stack map between them may not be expressible as a morphism between the two groupoids. This is not surprising because this happens for smooth manifolds already: if a chart of the domain manifold is too large compared to the corresponding chart of the target manifold, we cannot express the map on these two charts.

\begin{lem}[Second Dictionary Lemma]\Label{sdl}
Let $f:\X\to\Y$ be a morphism of stacks, $\phi_0:X_0\to Y_0$ a morphism of manifolds and $\eta$ a 2-isomorphism as in~(\ref{eta-di}). 
Then there exists a unique morphism of Lie groupoids
$\phi_1:X_1\to Y_1$ covering $\phi_0$ and making the cube~(\ref{cube})
2-commutative. 
\end{lem}

\begin{lem}[Third Dictionary Lemma]
Let $f:\X\to\Y$ be a morphism of stacks. Let $\phi:\gG_{\X}\to \gG_{\Y}$ and $\psi:\gG_{\X}\to \gG_{\Y}$ be two morphisms of Lie groupoids. Let $\eta$ and
$\eta'$ be 2-isomorphisms, where $(\phi,\eta)$ and $(\psi,\eta')$ both form
2-commutative cubes such as~(\ref{eta-di}).  Then there exists a unique
natural equivalence $\theta:\phi\Rightarrow\psi$ such that the diagram
$$\xymatrix@=.5pc{
&&&& Y_1\ddlto\drto &\\
X_0\ar[urrrr]^\theta\ddrto\ar[drrr]_{\phi_0}\ar[rrrrr]^{\psi_0} &&&&&
Y_0\ddlto\\ 
&&& Y_0\drto &&\\
&\X\ar[rrr]^f &&&\Y &}$$
is 2-commutative. 
\end{lem}

Now given an orbifold defined using orbifold charts and embeddings as in \cite{Sa} or \cite{CR2}, we can construct an \'{e}tale Lie groupoid, then the groupoid defines a DM stack. If we start with another system of local charts defining the same orbifold, the corresponding \'{e}tale Lie groupoid is Morita equivalent to the first one. Then the corresponding DM stack is equivalent to the DM stack determined by the first one. When the DM stack is effective, i.e. has an effective \'{e}tale atlas, we can go backward to get a topological space together with a system of orbifold charts. In this sense, we use DM stack as the definition of orbifold. 

\begin{defn}\label{orbifolddef} The 2-category of orbifolds are defined as the 2-category of DM stacks, namely:
\begin{itemize}
\item An {\em orbifold} is a DM stack. It is called effective if the stack has an effective atlas.
\item An orbifold morphism is a stack map between two DM stacks.
\item A 2-morphism between two orbifold morphisms is a 2-morphism between the two stack maps.
\end{itemize}
\end{defn}
\begin{defn}\label{orbifolddiffeo}
An orbifold morphism $\X\to \Y$  is called a {\em diffeomorphism} if it is an isomorphism between the two stacks.
\end{defn}

\begin{rmk}
When $\pmb{f}:\X \to \Y$  is diffeomorphism between orbifolds, its underlying map $|\pmb{f}|: |\X|\to |\Y|$ is a homeomorphism.
\end{rmk}
\begin{defn}
For an orbifold $\X$, a 0-dimensional substack of $\X$ which covers one point in $|\X|$ is called an orbipoint of $\X$.
\end{defn}
\begin{defn}
Let $\X$ be an orbifold $\X$ and let $\gG_\X:=(\gG_{\X1}\rightrightarrows \gG_{\X0})$ be an atlas groupoid of $\X$, with source and target maps denoted by $s, t: \gG_{\X1}\to \gG_{\X0}$.
\begin{enumerate}
\item 
The {\em inertia groupoid} $\Lambda \gG_{\X}$ of $\gG_{\X}$ is  defined as follows. The objects are $Ob(\Lambda\gG_{\X}):=\{g\in \gG_{\X1}: s(g)=t(g)\}$. For $g, g'\in  Ob(\Lambda\gG_{\X})$, an arrow $\alpha: g\to g'$ is an element $\alpha\in \gG_{\X1}$ such that $g\alpha=\alpha g'$. 

\item 
The {\em inertia orbifold} $I\X$ associated to $\X$ is the orbifold presented by the inertia groupoid $\Lambda \gG_{\X}$.
\end{enumerate}
\end{defn}
The inertia orbifold $I\X$ does not depend on the choice of atlas groupoid $\gG_{\X}$. As a category, the objects of $I\X$ are pairs $(x, g)$ where $x\in Ob(\X)$ and $g$ is an automorphism of $x\in Ob(\X)$. The inertia orbifold $I\X$ can be viewed as a way to stratify the orbifold according to the orbit types. Following \cite{CR1, CR2} we write $I\X$ as a disjoint union of components\footnote{By a component we mean a union of irreducible components.}$$I\X=\sqcup \X_{(g)}.$$
Each component $\X_{(g)}$ is called a {\em twisted sector}\footnote{Because of the description of objects of $I\X$ as pairs $(x,g)$ with $x\in Ob(\X)$ and $g$ an automorphism of $x$, we often think of $g$ which indexes the twisted sector $\X_{(g)}$ as an element in the automorphism group.}. There is a distinguished component $$\X\simeq \X_{(0)}:=\{(x, id)\,|\, x\in Ob(\X)\}\subset I\X,$$ called the {\em untwisted sector}.

There is a natural involution $\mathcal{I}: I\X\to I\X$ defined by $\mathcal{I}((x,g)):=(x, g^{-1})$, which plays an important role in the construction of Chen-Ruan orbifold cohomology. 

\subsection{Differential Forms, Vector Fields and Flows on Orbifolds}\label{sec:vec_field}
In this subsection we review differential forms, vector fields and flows on orbifolds/DM stacks following \cite{BX}, \cite{ML} and \cite{Hep}.

The sheaves $\Omega^{k}(\X)$ of differential $k$-forms on a stack $\X$ is defined as follows (\cite[Section 3.2]{BX}): for an object $v\in\X$ over a manifold $U$, define $\Omega^{k}(v):=\Omega^{k}_{U}(U)$, the differential forms on $U$. A differential form of degree $k$ on $\X$ is a global section of the sheaf $\Omega^{k}(\X)$, i.e., a homomorphism from the trivial sheaf on $\X$ to $\Omega^{k}(\X)$.

For DM stacks, differential forms have the following description in an atlas.
\begin{lem}[Proposition 2.9.i of \cite{ML}]\label{dictionaryForms}
Let $X_1 \rightrightarrows X_0 \xrightarrow{\xi} \mathcal{X}$ be a groupoid presentation of a DM stack $\mathcal{X}$, then

$$
\Omega^{\bullet} (\X) \cong
\{ \tau \in \Omega^{\bullet} (X_0) \ | 
\ s^* \tau = t^* \tau \} \cong 
\bigl\{ ( \sigma_1,\sigma_0)  \in 
\Omega^{\bullet} (X_1) \times \Omega^{\bullet} (X_0)
\bigr)
\mid
 s^* \sigma_0 = \sigma_1 , \  t^* \sigma_0 = \sigma_1 \bigr\} .
$$
\end{lem}

By \cite{Hep}, there is a lax functor from the $2$-category of differentiable stacks to itself which extends the functor $\Man\to \Man$ that sends a manifold to its tangent bundle and a map to its derivative. 

\begin{defn}(Definition 3.2 in \cite{Hep})
A \emph{vector field} on a differentiable stack $\X$ is a pair $(X,a_X)$ consisting of a morphism
\[X\colon\X\to T\X,\]
where $T\X$ is the tangent bundle of $\X$, and a $2$-cell (i.e. 2-morphism)
\begin{equation}\label{IntroVectorFieldEquation}\xymatrix{
\X\ar[r]^X\ar@/_6ex/[rr]_{Id_{\X}}^{}="1"& T\X\ar[r]^{\pi_\X}\ar@{=>}"1"^{a_X}& \X.
}\end{equation}
Here $\pi_\X\colon T\X\to\X$ is the natural projection map.
\end{defn}

We recall the definition of integral morphism and flow of a vector field from \cite{Hep}. Let $I$ be either $\real$ or an interval on $\real$.

\begin{defn}[Definition  4.1 in \cite{Hep}]\label{IntegralDefinition}
Let $X$ be a vector field on $\X$.  Then $\Phi\colon\Y\times I\to\X$ is \emph{an integral morphism of $X$} if there is a $2$-morphism
\begin{equation}\label{IntegralTwoMorphism}
t_\Phi\colon X\circ\Phi\Longrightarrow T\Phi\circ\frac{\partial}{\partial t},
\end{equation}
which we represent as the following diagram
\begin{equation}\label{FlowDiagram}\xymatrix{
T(\Y\times I)\ar[r]^-{T\Phi}_{}="2"& T\X\\
\Y\times I\ar[u]^{\frac{\partial}{\partial t}}\ar[r]_-\Phi &\X.\ar[u]_X^{}="1"\ar@{=>}^{t_\Phi}"1";"2"
}\end{equation}
The $2$-morphism $t_\Phi$ must satisfy the property that the $2$-morphisms in the following diagram 
\begin{equation}\label{CategorifiedFlowDiagram}\xymatrix{
\Y\times I\ar@{<-}@/_12ex/[dd]^{}="1"_{Id_{\Y\times I}}\ar[r]_{}="3"^\Phi & \X\ar@{<-}@/^12ex/[dd]_{}="2"^{Id_{\Y\times I}} \\
T(\Y\times I)\ar[u]\ar[r]^-{T\Phi}_{}="4"\ar@{=>}"1"^-{a_{\frac{\partial}{\partial t}}} & T\X\ar@{=>}"2"_-{a_X}\ar[u]_{}="5"\\
\Y\times I\ar[u]\ar[r]_\Phi & \X\ar[u]_{}="6"\ar@{=>}"6";"4"^{t_\Phi}\ar@{=>}"5";"3"
}\end{equation}
compose to the trivial $2$-morphism from $\Phi\colon\Y\times\real\to\X$ to itself.  The choice of $t_\Phi$ is regarded as part of the data for $\Phi$.  Note that if $\Phi$ integrates $X$ and there is an equivalence $\lambda\colon Y\Rightarrow X$, then $\Phi$ also integrates $Y$ when equipped with the $2$-morphism $t_\Phi\circ (\lambda\ast Id_\Phi)$.
\end{defn}

\begin{defn}[Definition  4.3 in \cite{Hep}]\label{FlowDefinition}
Let $X$ be a vector field on $\X$.  A \emph{flow of $X$} is a morphism
\[\Phi\colon\X\times\real\to\X\]
integrating $X$ and equipped with a $2$-morphism $e_\Phi\colon\Phi|_{\X\times\{0\}}\Rightarrow Id_\X$.
\end{defn}

The following proposition in \cite{ML} describes vector fields and its pairing with differential forms in an atlas. Recall that a Lie algebroid of a Lie groupoid $X_{1}\twoarrows X_{0}$ is a vector bundle map called the anchor $a:A \to TX_{0}$, where $A = \text{ker}\, ds|_{X_{0}}$, and $a = dt$. Moreover, given a surjective submersion  $f: Y \rightarrow X$ of manifolds, a vector field $v_{_Y} \in \mathrm{Vect} (Y)$ and a vector field  $v_{_X} \in \mathrm{Vect} (X)$, 
then $v_{_Y}$ is called \emph{$f$-related} to $v_{_X}$ if $df (v_{_Y}) = v_{_X} \circ f$. 

Given $v_{_X}$, $v_{_Y}$ is determined up to a section of the bundle $\ker df \subset TY$.

\begin{prop}[Proposition 2.9.ii \& iii in \cite{ML}]\label{dictionVector}
Let $X_1 \rightrightarrows X_0 \xrightarrow{\xi} \mathcal{X}$
be a groupoid presentation of a DM stack $\mathcal{X}$ with the associated
algebroid $A \hookrightarrow TX_0$. Then
\begin{enumerate}
\item\label{ArtinGlobalVectors}
The Lie algebra $\mathrm{Vect} (\mathcal{X}) := \mathcal{C}^{\infty}_{T\mathcal{X}} (\mathcal{X})$ of vector fields on $\mathcal{X}$, i.e., of global sections of $T\mathcal{X}$, is isomorphic
to the Lie algebra $\mathcal{C}^{\infty}_{TX_0/A} (X_0)^{X_1}$ of
$X_1$-invariant sections of the bundle $TX_0/A$.  Explicitly, vector
fields on $\mathcal{X}$ are equivalence classes of pairs consisting of
a vector field $v_0$ on $X_0$ and a vector field $v_1$ on $X_1$, which
are both $s$- and $t$-related:
\[
\mathrm{Vect} (\mathcal{X}) \cong
\frac{ 
\bigl\{(v_1, v_0) \in \mathrm{Vect} (X_1) 
 \times \mathrm{Vect} (X_0) \mid  
ds (v_1) = v_0 \circ s , \ dt (v_1) = v_0 \circ t \bigr\}
}{
\bigl\{(v_1, v_0) \mid
ds (v_1) = v_0 \circ s , \ dt (v_1) = v_0 \circ t ,
\ v_1 \in ( \ker ds + \ker dt ) \bigr\} 
}
\]
\item
The contraction of vector fields and forms on $\mathcal{X}$ is induced 
by the contraction of vector fields and forms on $X_0$ and $X_1$:
\begin{equation}\nonumber
\iota_{( v_1 , v_0 )} \, ( \sigma_1 , \sigma_0 ) = ( \iota_{v_1}  \sigma_1 , \iota_{v_0} \sigma_0 )
\end{equation}
\end{enumerate}
\end{prop}

To talk about Hamiltonian dynamics, we need the notion of time dependent vector fields, which can be considered as a special class of vector fields on $\real\times \X$: 

\begin{defn}
A \emph{time dependent vector field} on a differentiable stack $\X$ is a pair $(X,a_X)$ consisting of a morphism 
$$X\colon  \X\times I\to T\X$$
and a $2$-cell
\begin{equation}\label{IntroVectorFieldEquation}\xymatrix{
 \X\times I\ar[r]^X\ar@/_6ex/[rr]_{pr_{\X}}^{}="1"& T\X\ar[r]^{\pi_\X}\ar@{=>}"1"^{a_X}& \X.
}\end{equation}
Here $\pi_\X\colon T\X\to\X$ is the natural projection map and $pr_{\X}$ is the projection to the $\X$ component.
\end{defn}

Like the manifold case, every time dependent vector field $X$ on $\X$ determines a vector field on $\X\times \real$ by $\bar{X}:=X\oplus\frac{\partial}{\partial t}$. The flow $\bar{\Phi}$ of $\bar{X}$ exists locally by the existence and uniqueness of integral morphism (\cite[Theorems 4.4--4.5]{Hep}). Note that the vector field along the non-compact direction is $\frac{\partial}{\partial t}$ whose integral morphism exists globally on $\real$. Repeating the proof of \cite[Theorem 4.8]{Hep}, we get the global existence of the flow $\bar{\Phi}$. Define $\Phi:\X\times \real\to\X$ to be the composition of the following:
$$\xymatrix{
\X\times \real \ar[r]^{Id_{\X}\times \Delta_{t}} &\ \ \ \  \X\times \real\times \real \ar[r]^{\bar{\Phi}} & \X\times \real \ar[r]^{pr_{\X}} & \X}
$$
where $\Delta_{t}$ is the diagonal map $\real\to\real\times\real$.  Then $\Phi$ is called the {\em integral of the time dependent vector field $X$}. It is easy to see that this generalizes the integral of time dependent vector field in the manifold case.

\subsection{Hamiltonian Dynamics of Symplectic Orbifold and $\pi_{1}(Ham(\X,\omega))$}\label{hamloop}
In this section, we give the basic definition of Hamiltonian dynamics such as Hamiltonian vector fields and Hamiltonian diffeomorphisms for symplectic orbifolds. Then we define Hamiltonian loops based at the identity for a symplectic orbifold and the homotopy equivalence between them, and show the homotopy equivalence classes form a group  denoted by $\pi_{1}(Ham(\X,\omega))$. 

\begin{rmk}
Note that we do not discuss the infinite dimensional orbifold structure on the 2-group of Hamiltonian diffeomorphisms $Ham(\X,\omega)$ in this paper. With a proper definition of such structure on $Ham(\X,\omega)$, $\pi_{1}(Ham(\X,\omega))$ defined here is expected to coincide with the orbifold fundamental group of $Ham(\X,\omega)$ with respect to that orbifold structure.
\end{rmk}

\begin{defn}\label{Hamfield}
A vector field $X:\X\to T\X$ is called a {\em Hamiltonian vector field} if there is a 0-form (function) $H\in \Omega^{0}(\X)$ such that  $dH=\omega(X,\cdot)$.
A time dependent vector field $X:\X\times\real\to T\X$ is called {\em time dependent Hamiltonian vector field} if there is a 0-form (function) $H\in \Omega^{0}(\X\times \real)$ such that  $dH_{t}=\omega(X_{t},\cdot)$ for all $t\in \real$, where $H_{t}$ is the pullback of $H$ by the inclusion $\X\times\{t\}\to\X\times\real$ and $X_{t}:=X|_{\X\times\{t\}}$.
\end{defn}
\begin{defn}\label{Hamdiffeo}
A diffeomorphism $\phi:\X\to \X$ is called a {\em Hamiltonian diffeomorphism} if $\phi=\Phi|_{\X\times\{1\}}$ where $\Phi$ is the integral of some time dependent Hamiltonian vector field $X$. 
\end{defn}

All Hamiltonian diffeomorphisms together with 2-morphisms between them form a 2-group, which is denoted by $Ham(\X,\omega)$.

Throughout this section we denote $I=[0,1]$. One may want to define a path of Hamiltonian diffeomorphisms as a functor $I\times \X \to \X$ integrating a Hamiltonian vector field. However in the 2-category of orbifolds defined above, we cannot talk about non-smooth morphisms since they are fibered over the category of smooth manifolds. Namely, every path defined as a functor from $I\times \X$ to $\X$ is smooth. On the other hand, simply connecting two smooth paths may not give a smooth path. This motivates our definition of Hamiltonian paths given below.

Any finite set of numbers $0=t_{0}<t_{1}<...<t_{L+1}=1$ determine a partition of $I$, $$I=\cup_{l=0,...,L}[t_{l},t_{l+1}].$$ We denote $I_{l}:=[t_{l},t_{l+1}]$.
\begin{defn}\label{Hamiltonianpath}
A {\em Hamiltonian path} $\pmb{\gamma}$ is an $L$-tuple of functors: $\{\gamma_{l}:I_{l}\times \X \to \X |l=1,...,L\}$ such that:
\begin{enumerate}
\item $\gamma_{l}$ integrates a time dependent Hamiltonian vector field, for $l=1,...,L$;
\item For all $l\in\{1,...,L-1\}$, there exists a 2-morphism $$a_{l}:\gamma_{l}(t_{l})\Rightarrow \gamma_{l+1}(t_{l}),$$ where $\gamma_{l}(t_{l}):\{t_{l}\}\times \X\to \X$ is the restriction of $\gamma_{l}$ to the substack $\{t_{l}\}\times \X$.
\end{enumerate}
\end{defn}

\begin{rmk}\label{jumprmk}
The 2-morphisms $a_{l}$ measure the changes when moving from one piece to the next. In the manifold case, they are always trivial. But they may not match smoothly at the joining points. Thus the above definition gives piecewisely smooth path in the manifold case. If a Hamiltonian path is  given by a single functor $\gamma:I \times \X\to \X$, i.e. $L=1$, then we call it a smooth Hamiltonian path.
\end{rmk}

Next we define homotopy equivalence relation between Hamiltonian paths. 

It is natural to consider two paths as equivalent if one path is defined by splitting an interval of the other path. Namely, for $\pmb{\gamma}$ an $L$-tuple of functors $\{ \gamma_{l}:I_{l}\times \X \to \X \}$ and a refined partition $0=t_{0}<t_{1}<...<t_{i}<\hat{t}<t_{i+1}<...<t_{L+1}=1$, there is a Hamiltonian path $\pmb{\gamma}'$ with an $(L\!+\!1)$-tuple of functors:  $\{\gamma_{l}:I_{l}\times \X \to \X|l\neq i \}\cup \{\gamma_{i}|_{[t_{i},\hat{t}]\times \X},\ \gamma_{i}|_{[\hat{t},t_{i+1}]\times \X}\}$. We say that $\pmb{\gamma}'$ and $\pmb{\gamma}$ are {\em splitting related}.

Another natural equivalence is that $\pmb{\gamma}$ and $\pmb{\gamma}'$ are both $L$-tuple of functors $\{ \gamma_{l}:I_{l}\times \X \to \X \}$ and $\{ \gamma'_{l}:I_{l}\times \X \to \X \}$ and there exists a 2-morphism $\gamma_{l}\Rightarrow \gamma'_{l}$ for every $l$. Such two paths are called {\em 2-related}.

Two paths $\pmb{\gamma}$ and $\pmb{\gamma}'$ are said to be {\em naturally related} if there exists a finite collection of Hamiltonian paths $\pmb{\gamma}=\pmb{\gamma}_1,\pmb{\gamma}_2,...,\pmb{\gamma}_k=\pmb{\gamma}'$ such that for each $i=1,2,...,k-1$, $\pmb{\gamma}_i$ and $\pmb{\gamma}_{i+1}$ are either splitting related or 2-related.

A partition of $I\times I=[0,1]\times [0,1]$ will be called a {\em brick partition} if it is given by first dividing $[0,1]\times [0,1]$ into rectangles with horizontal lines, then dividing each rectangle into smaller rectangles (bricks) with vertical lines. Namely, $$[0,1]\times [0,1]=\cup_{k=0}^{K}\cup_{l=0}^{L_{k}} [s_{k},s_{k+1}]\times [t_{k,l},t_{k,l+1}],$$ where $0=s_{0}<s_{1}<...<s_{K+1}=1$, and $0=t_{k,0}<t_{k,1}<...<t_{k,L_{k}+1}=1$. The bricks on each layer have the same height but possibly different width. See Figure 1 for an example of brick partitions \footnote{The first coordinate corresponds to the vertical direction.}.

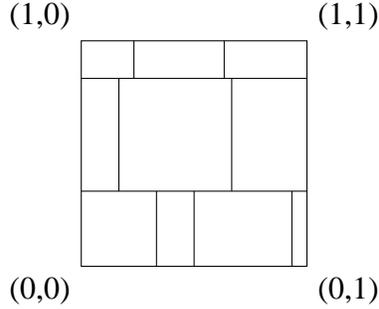
\begin{figure}[htb!]\label{fig_brickpart}
\centering
\begin{tikzpicture}
 \draw (0,0) node[below left] {(0,0)} --
        (3,0) node[below right] {(0,1)} --
        (3,3) node[above right] {(1,1)} --
        (0,3) node[above left]  {(1,0)} -- cycle;
 \draw (0,1) -- (3,1);
 \draw (0,2.5) -- (3,2.5);
 \draw (1,0) -- (1,1);
 \draw (1.5,0) -- (1.5,1);
 \draw (2.8,0) -- (2.8,1);
 \draw (0.5,1) -- (0.5,2.5);
 \draw (2,1) -- (2,2.5);
 \draw (0.7,2.5) -- (0.7,3);
 \draw (1.9,2.5) -- (1.9,3);
\end{tikzpicture} 
\caption{An example of brick partition.}
\end{figure}

Now we are ready to define homotopy equivalence between two Hamiltonian paths.
\begin{defn}\label{Hamhomotopy}
Two Hamiltonian paths $\pmb{\gamma}=\{\gamma_{l}:I_{l}\times \X \to \X \}$ and $\pmb{\gamma}'=\{\gamma'_{l'}:I'_{l'}\times \X \to \X \}$ are homotopy equivalent if there is a brick partition $I\times I=\cup_{k=0}^{K}\cup_{l=0}^{L_{k}} [s_{k},s_{k+1}]\times [t_{k,l},t_{k,l+1}]$ as above and a collection of functors $\pmb{\Gamma}:=\{ \Gamma_{k,l}: [s_{k},s_{k+1}]\times [t_{k,l},t_{k,l+1}]\times \X \to\X|0\le k\le K, \ 0\le l\le L_{k}\}$ such that
\begin{enumerate}
\item For all $k\in \{0,...,K\}$ and $\forall s\in [s_{k},s_{k+1}]$, $\{\Gamma_{k,l}|_{\{s\}\times  [t_{k,l},t_{k,l+1}] }|0\le l\le L_{k}\}$ define a Hamiltonian path;
\item For all $k\in\{1,...,K\}$, $\{\Gamma_{k-1,l}|_{\{s_{k}\}\times  [t_{k-1,l},t_{k-1,l+1}] }|0\le l\le L_{k-1}\}$ and $\{\Gamma_{k,l}|_{\{s_{k}\}\times  [t_{k,l},t_{k,l+1}] }|0\le l\le L_{k}\}$ are naturally related;
\item $\pmb{\Gamma}_{0}:=\{\Gamma_{0,l}|_{\{0\}\times  [t_{0,l},t_{0,l+1}] }|0\le l \le L_{0}\}$ and $\pmb{\gamma}$ are naturally related, \\
$\pmb{\Gamma}_{1}:=\{\Gamma_{K,l}|_{\{1\}\times  [t_{K,l},t_{K,l+1}] }|0\le l \le L_{0}\}$ and $\pmb{\gamma}'$ are naturally related.
\end{enumerate}
Moreover, if  $\Gamma_{k,0}|_{\{(s,0)\times \X\} }=\pmb{\gamma}|_{\{0\}\times \X}=\pmb{\gamma}'|_{\{0\}\times \X}$, and
$\Gamma_{k,L_{k}}|_{\{(s,0)\times \X\} }=\pmb{\gamma}|_{\{1\}\times \X}=\pmb{\gamma}'|_{\{1\}\times \X}$, for all $k$ and $s\in [s_{k},s_{k+1}]$,
 then the two Hamiltonian paths are called {\em homotopy equivalent relative to the ends}.
\end{defn}

It is easy to see the above definition of homotopy equivalence is indeed an equivalence relation.

\begin{lem}\label{removejump}
Let $\gamma:[a,b]\times \X \to \X$ be a stack map integrating a time dependent Hamiltonian vector field $X:[a,b]\times \X \to T\X$, then there exists a stack map $\Gamma:[0,1]\times [a,b]\times \X \to \X$ such that:
\begin{enumerate}
\item $\Gamma|_{\{0\}\times [a,b]\times \X}=\gamma$;
\item for $t\in [0,1]$, $\Gamma|_{ \{t\}\times [a,b]\times \X}$ is Hamiltonian, i.e. it integrates a time dependent Hamiltonian vector field;
\item for $t\in [0,1]$, $\Gamma|_{\{t\}\times \{b\}\times \X}=\gamma|_{\{b\}\times \X}$;
\item there is a small $\epsilon>0$, such that $\Gamma|_{\{1\}\times(b-\epsilon,b]\times \X}=\gamma|_{\{b\}\times \X}\circ \pi_{\X}$, where $\pi_{\X}:(b-\epsilon,b]\times \X\to \X$ is the obvious projection.
\end{enumerate}
\end{lem}
\begin{proof}
Let $\rho:[0,1]\times [a,b]\to [0,1]\times [a,b]$ be a smooth map such that:
\begin{enumerate}
\item $\rho_{s}:=\rho|_{\{s\}\times [a,b]}:\{s\}\times [a,b] \to \{s\}\times [a,b]$;
\item $\rho_{0}(t)=t$, for all $t\in [a,b]$;
\item $\rho_{s}(b)=b$, for all $s\in [0,1]$;
\item $\rho'_{1}(t)=0$, for all $t\in (b-\epsilon,b]$.
\end{enumerate}
Denote by $pr$ the projection from $[0,1]\times [a,b]\times \X$ to $[a,b]\times \X$. Let $\tilde{\Gamma}=\gamma\circ pr$, then it is straightforward to check $\Gamma:=\tilde{\Gamma}\circ(\rho\times Id_{\X})$ satisfies the required properties.
\begin{figure}[htb!]\label{fig_repara}
\centering
\begin{tikzpicture}
 \draw (0,0) node[below left] {(a,a)} --
        (2,0) node[below right] {(a,b)} --
        (2,2) node[above right] {(b,b)} --
        (0,2) node[above left]  {(b,a)} -- cycle;
 \draw (0,0) -- (2,2);
 \draw[smooth,domain=0:1] plot ({\x*\x*\x-2*\x*\x+2*\x},\x);
 \node at (0.8,1.1) {s=0};
\node at (1.1,0.6) {s=1};
\end{tikzpicture} 
\caption{Graph of the function $\rho_{s}$.}
\end{figure}
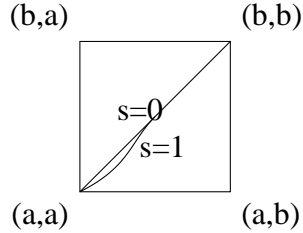
\end{proof}
The above lemma shows that every smooth Hamiltonian path is homotopy equivalent to a Hamiltonian path with stationary end, namely does not depend on time near $b$. Similarly we have the same result at the other end point $a$.

\begin{defn}\label{Hamiltonianloop}
A Hamiltonian path $\pmb{\gamma}:=\{\gamma_{l}:I_{l}\times \X \to \X|l=1,...,L \}$ is called a {\em Hamiltonian loop} if there exists a 2-morphism 
$$a_{0}:\gamma_{L}|_{\{1\}\times \X}\Rightarrow \gamma_{1}|_{\{0\}\times \X}.$$

A Hamiltonian loop is called {\em based at the identity} if $\gamma_{1}|_{\{0\}\times \X}= Id_{\X}$ (and thus there is 2-morphism $\gamma_{1}|_{\{1\}\times \X}\Rightarrow Id_{\X}$).
 
For a symplectic orbifold $(\X,\omega)$, we define $\pi_{1}(Ham(\X,\omega))$ to be the set of all homotopy equivalence classes of Hamiltonian loops of $(\X,\omega)$. 
\end{defn}

As in Remark \ref{jumprmk}, Hamiltonian loops defined as above correspond to piecewise smooth Hamiltonian loops in the manifold case. We introduce the following definition of smooth Hamiltonian loop which generalizes smooth Hamiltonian loop in the manifold case. We denote by $exp:I\to S^{1}$ the exponential map $exp(t):=e^{2\pi i t}$.
\begin{rmk}
In what follows we will need to ``combine'' maps. Here is the convention we use for that. For two maps $f: A\to B$ and $g:A\to C$ with the same domain, we write $f\times g: A\to B\times C$ for the map such that for $a\in A$, $(f\times g)(a)=(f(a),g(a))\in B\times C$. For any two maps $p:A\to B$ and $q:C\to D$ we write $(f,g): A\times C\to B\times D$ for the map such that for $(a,c)\in A\times C$, $(f,g)(a,c)=(f(a),g(c))$. 
\end{rmk}

\begin{defn}\label{smthHamiltonianloop} 
A Hamiltonian loop $\pmb{\gamma}:=\{\gamma:I\times \X \to \X\}$ is called a {\em smooth Hamiltonian loop} if there exists a 2-morphism $\gamma\Rightarrow\tilde{\gamma}\circ (exp, Id_{\X})$ for some stack map $\tilde{\gamma}:S^{1}\times \X\to \X$. 

Similarly a smooth Hamiltonian loop is called based at the identity if there exists a 2-morphism $\gamma|_{\{e^{0i}\}\times \X}\Rightarrow Id_{\X}$.
\end{defn}
From now on, we call (smooth) Hamiltonian loop based at the identity {\em (smooth) based Hamiltonian loop}, and homotopy relative to the based point  \emph{relative homotopy} for short.

\begin{lem}\label{smthyloop}
Every based Hamiltonian loop is relative homotopy equivalent to a smooth based Hamiltonian loop. 
\end{lem}
\begin{proof}
At a jumping point $t_{k}$, apply Lemma \ref{removejump}, we get a based Hamiltonian path $\pmb{\gamma}'=\{\gamma_{l}:I_{l}\times \X\to \X\}$ which
\begin{enumerate}
\item is relative homotopy equivalent to the original Hamiltonian loop;
\item  $\gamma_{k-1}|_{(t_{k}-\epsilon,t_{k}]\times \X}=\gamma_{k-1}|_{\{t_{k}\}\times \X}\circ \pi^{-}_{\X}$, where $\pi^{-}_{\X}$ is the projection $(t_{k}-\epsilon,t_{k}]\times \X\to
\X$;
\item  $\gamma_{k}|_{(t_{k},t_{k}+\epsilon]\times \X}=\gamma_{k}|_{\{t_{k}\}\times \X}\circ \pi^{+}_{\X}$, where $\pi^{+}_{\X}$ denote the projection $(t_{k},t_{k}+\epsilon]\times
\X\to \X$.
\end{enumerate}
Note that there exists a 2-morphism $\gamma_{k-1}|_{\{t_{k}\}\times \X}\Rightarrow \gamma_{k}|_{\{t_{k}\}\times \X}$, thus the two stack maps $\gamma_{k-1}$ and $\gamma_{k}$ can be glued into one stack map from $[t_{k-1},t_{k+1}]\times \X$ to $\X$. 

Apply the above procedure at $t_{k}$ for $k=1,...,L-1$, then we get a smooth Hamiltonian path $\gamma$ relative homotopy equivalent to the original one. Moreover, if the original path is a Hamiltonian loop, then there exists a 2-morphism $\gamma_{L}|_{\{1\}\times \X}\Rightarrow \gamma_{1}|_{\{0\}\times \X}$, then we can apply the above glueing at the two end points to define $\tilde{\gamma}:S^{1}\times \X\to \X$. Thus the original based Hamiltonian loop is relative homotopy equivalent to a smooth based Hamiltonian loop.
\end{proof}
Moreover there is a notion of smooth relative homotopy between smooth based Hamiltonian loops, which will be useful in the next section.
\begin{defn}\label{smthhomotopy}
A  relative homotopy $\pmb{\Gamma}=\{\Gamma: I\times I\times \X\to \X\}$ between two smooth based Hamiltonian loops $\pmb{\gamma_{0}}=\{\gamma_{0}\}$ and $\pmb{\gamma_{1}}=\{\gamma_{1}\}$ is called smooth if  it factors through $(Id_{I}, exp, Id_{\X}):I\times I \times \X\to I\times S^{1}\times \X$, namely $$\Gamma=  \tilde{\Gamma}\circ (Id_{I}, exp, Id_{\X})$$ for some stack map $\tilde{\Gamma}:I\times S^{1}\times \X\to \X$.
\end{defn}

\begin{lem}\label{smthhomotopylem}
If two smooth based Hamiltonian loops are relative homotopy equivalent, then there are smooth relative homotopy between them.
\end{lem}
\begin{proof}
For two smooth Hamiltonian loops $\pmb{\gamma}=\{\gamma I\times\X \to \X \}$ and $\pmb{\gamma}'=\{\gamma':I\times \X \to \X \}$, let $\pmb{\Gamma}=\{ \Gamma_{k,l}: [s_{k},s_{k+1}]\times [t_{k,l},t_{k,l+1}]\times \X \to\X|0\le k\le K, \ 0\le l\le L_{k}\}$ be a relative homotopy between them. We will apply the method in the proof of Lemma \ref{removejump} to reparametrize $\Gamma_{k,l}$, and then similar to the proof of Lemma \ref{smthyloop} we can get them into one stack map. 

Let $\xi_{k}:[s_{k},s_{k+1}]\to[s_{k},s_{k+1}]$ be a smooth function such that for a small positive number $\delta$,
\begin{enumerate}
\item $\xi_{k}(s)=s_{k}$ for $s\in [s_{k},s_{k}+\delta)$;
\item $\xi_{k}(s)=s_{k+1}$ for $s\in [s_{k+1}-\delta,s_{k+1})$;
\item $\xi_{k}(s)=s$ for $s\in [s_{k}+2\delta,s_{k+1}-2\delta)$.
\end{enumerate}
Similarly define $\eta_{k,l}:[t_{k,l},t_{k,l+1}]\to[t_{k},t_{k,l+1}]$ be a smooth function such that for a small positive number $\delta$,
\begin{enumerate}
\item $\eta_{k,l}(t)=t_{k,l}$ for $t\in [t_{k,l},t_{k,l}+\delta)$;
\item $\eta_{k,l}(t)=t_{k,l+1}$ for $t\in [t_{k,l+1}-\delta,t_{k,l+1})$;
\item $\eta_{k,l}(t)=t$ for $t\in [t_{k,l}+2\delta,t_{k,l+1}-2\delta)$.
\end{enumerate}
Then we can glue all $\Gamma_{k,l}\circ(\xi_{k}, \eta_{k,l}, Id_{\X})$ along proper boundaries to get a smooth relative homotopy between the original loops.
\end{proof}

\begin{example}[Hamiltonian loop of $\lbrack\com/\Ztwo\rbrack$\footnote{We use the brackets [ ] to indicate stacky quotient.}]  Consider the $\Ztwo$-action on $\com$ by $\pm1\cdot z:=\pm z$.
Let $[\com/\Ztwo]$ be the stack represented by the translation groupoid of this action $\Ztwo \ltimes \com:=\Ztwo\times \com\twoarrows \com$. By the dictionary lemma, the groupoid morphism $I\times \Ztwo \ltimes \com\to \Ztwo \ltimes \com$ defined by 
$$
\gamma_{0}(t;z):=e^{\pi it}z,\ \ \ \gamma_{1}(t;g,z):=(g,e^{\pi it}z)
$$
determines a stack map $\gamma: I\times[\com/\Ztwo]\to[\com/\Ztwo]$. It is easy to see that $\gamma$ defines a based Hamiltonian loop. Note that on $Ob(\Ztwo \ltimes \com)=\com$, $\gamma_{0}$ starts with the identity map, and ends with a $180^{\circ}$ rotation which induces an automorphism on $[\com/\Ztwo]$ differing from the identity by a 2-morphism, thus it is a path based at the identity.

 We now show that it is also a smooth based Hamiltonian loop by giving a concrete construction of $\tilde{\gamma}$ such that $\gamma=\tilde{\gamma}\circ (exp, Id_{\X})$. Consider  
$$U_{L}:= \{e^{2\pi it}|t\in (0,1)\}=
\begin{aligned}
\begin{tikzpicture}
 \def\R{1} 
 \draw[thick] (0.5*\R,0) arc (7:355:0.5*\R);
\filldraw[fill=white, draw=black, fill opacity=0] (0.5*\R, 0) circle (0.06);
\end{tikzpicture} 
\end{aligned}
, \ \ \ \ \ \ U_{R}:=\{e^{2\pi it}|t\in (-\frac{1}{2},\frac{1}{2})\}=
\begin{aligned}\begin{tikzpicture}
 \def\R{1} 
  \draw[thick] (-0.5*\R,0) arc (-175:175:0.5*\R);
\filldraw[fill=white, draw=black, fill opacity=0] (-0.5*\R, 0) circle (0.06);
\end{tikzpicture}\end{aligned}.$$

Now $U_{S^{1}}$ is defined by:
 \begin{eqnarray*}
 Ob(U_{S^{1}}) & = & U_{L}\sqcup U_{R},\\
 Mor(U_{S^{1}}) & = & U_{L}\times_{S^{1}} U_{L}\ \sqcup\ U_{R}\times_{S^{1}} U_{R}\ \sqcup\ U_{L}\times_{S^{1}} U_{R}\ \sqcup\ U_{R}\times_{S^{1}} U_{L}.
\end{eqnarray*}
Note that:
$$
\begin{aligned} U_{L}\times_{S^{1}} U_{L} \end{aligned} =  \begin{aligned}\begin{tikzpicture}
 \def\R{1} 
 \draw[thick] (0.5*\R,0) arc (7:355:0.5*\R);
\filldraw[fill=white, draw=black, fill opacity=0] (0.5*\R, 0) circle (0.06);
\end{tikzpicture}\end{aligned}, \ \  \
 U_{R}\times_{S^{1}} U_{R}  = \begin{aligned} \begin{tikzpicture}
 \def\R{1} 
   \draw[thick] (-0.5*\R,0) arc (-175:175:0.5*\R);
\filldraw[fill=white, draw=black, fill opacity=0] (-0.5*\R, 0) circle (0.06);
\end{tikzpicture}\end{aligned},
$$

$$
U_{L}\times_{S^{1}} U_{R}  =   \begin{aligned}\begin{tikzpicture}
 \def\R{1} 
 \draw[thick] (0.5*\R,0) arc (0:180:0.5*\R);
\end{tikzpicture} \end{aligned} \  \ \ \sqcup\ \ \ \  \begin{aligned}\begin{tikzpicture}
 \def\R{1} 
 \draw[thick] (-0.5*\R,0) arc (180:360:0.5*\R);
\end{tikzpicture}\end{aligned} , 
$$
$$
U_{R}\times_{S^{1}} U_{L}  = \begin{aligned} \begin{tikzpicture}
 \def\R{1} 
 \draw[thick] (0.5*\R,0) arc (0:180:0.5*\R);
\end{tikzpicture} \end{aligned} \ \ \ \sqcup \ \ \  \begin{aligned} \begin{tikzpicture}
 \def\R{1} 
 \draw[thick] (-0.5*\R,0) arc (180:360:0.5*\R);
\end{tikzpicture}\end{aligned} .
$$

We denote $(e^{2\pi i t})_{*}\in U_{*}$ and $g=(e^{2\pi i t})_{*}\to (e^{2\pi i t'})_{\bullet}\in U_{*}\times_{S^{1}} U_{\bullet} $,  for $*, \bullet=R, L$, where
\begin{itemize}
\item  $t'=t-1$, if $g\in \begin{aligned} \begin{tikzpicture}
 \def\R{1} 
 \draw[thick] (-0.5*\R,0) arc (180:360:0.5*\R);
\end{tikzpicture}\end{aligned} \subset U_{L}\times_{S^{1}} U_{R} $;
\item $t'=t+1$, if $g\in \begin{aligned} \begin{tikzpicture}
 \def\R{1} 
 \draw[thick] (-0.5*\R,0) arc (180:360:0.5*\R);
\end{tikzpicture}\end{aligned} \subset U_{R}\times_{S^{1}} U_{L} $;
\item $t'=t$, otherwise.
\end{itemize}

Then the groupoid morphism $\tilde{\gamma}$ is given by:
$$\tilde{\gamma}_{0}((e^{2\pi i t})_{*},z):=e^{\pi i t}z , \ \ \ \ \ \ \ \ \ \   \text{for } \ (e^{ 2\pi i t})_{*}\in U_{*},\ *=R,L;$$
$$\tilde{\gamma}_{1}((e^{2\pi i t})_{*}\to (e^{ 2\pi i t'})_{\bullet}, \xymatrix{z\ar[r]^{h} & h\cdot z}):=\xymatrix{e^{\pi i t}z \ar[r]^{h'} & h\cdot e^{\pi i t}z},\ \ \ \ \ *,\bullet =R,L.$$

There is an obvious groupoid morphism $Exp$ from $I\twoarrows I$ to $U_{S^{1}}$ representing the exponential map $exp$. Let $Id_{\inte_{2}\ltimes\com}$ be the identity groupoid morphism from $\inte_{2}\ltimes\com$ to itself. Then  $\tilde{\gamma}$ composed with $(Exp, Id_{\inte_{2}\ltimes\com})$ defines the stack map $\gamma$. Thus $\gamma$ is a Hamiltonian loop based at the identity. 
\end{example}

\begin{rmk}
In general, all Hamiltonian circle actions in \cite{LT} and \cite{ML} can be viewed as Hamiltonian loops.
\end{rmk}

Given two based Hamiltonian loops:
$$\pmb{\gamma}:=\{\gamma_{l}:I_{l}\times \X \to \X|l=1,...,L \}, \ \ \ \ \pmb{\gamma}':=\{\gamma'_{l}:I_{l}\times \X \to \X|l=1,...,L \},$$ there are two ways to define a product operation. One way is by connecting two loops:
$$ 
\pmb{\gamma}\cdot_{cn}\pmb{\gamma}':=\{\gamma'_{l}\circ db_{l}^{L}:I^{L}_{l}\times \X \to \X,\gamma'_{l'}\circ db^{R}_{l}:I'^{R}_{l'}\times \X \to \X|l=1,...,L,l'=1,...,L' \}
$$
where $db^{L}_{l}(t):=\frac{t}{2}$, $I^{L}_{l}$ is the pre-image of $I_{l}$, and $db^{R}_{l'}(t):=\frac{t+1}{2}$, $I'^{R}_{l'}$ is the pre-image of  $I'_{l'}$.

The other way is by composition. Without loss of generality, we may assume that the loop is defined over the same partition of $I$ since there is a common refined partition and split the two loops according to the refined partition. Thus 
$$\pmb{\gamma}=\{\gamma_{l}:I_{l}\times \X \to \X|l=1,...,L \},\ \ \ \ \ \pmb{\gamma}'=\{\gamma'_{l}:I_{l}\times \X \to \X|l=1,...,L \}.$$
 Then define
$$\pmb{\gamma}\cdot_{cp}\pmb{\gamma}':=\{\gamma_{l}\circ(pr_{I_{l}}\times\gamma'_{l}):I_{l}\times \X \to \X|l=1,...,L \},$$where $pr_{I_{l}}:I_l\times \X\to I_l$ is the natural projection to $I_{l}$.

It is straightforward to check from the definitions that if $\pmb{\gamma}_{1},\pmb{\gamma}_{2}$ and $\pmb{\gamma}'_{1},\pmb{\gamma}'_{2}$ are homotopy equivalent, then so do their products. Thus the two products descend to products in $\pi_{1}(Ham(\X,\omega))$.

Note that ``$\cdot_{cp}$'' and ``$\cdot_{cn}$''  are nothing but two ways to define products between loops inside a Lie group. In that case the two definitions coincide when passing to homotopy. The next lemma shows the same is true for Hamiltonian loops considered in this paper. 
\begin{prop}
 The products $\pmb{\gamma}\cdot_{cp}\pmb{\gamma}'$ and $\pmb{\gamma}\cdot_{cn}\pmb{\gamma}'$ induce the same product on $\pi_{1}(Ham(\X,\omega))$.
\end{prop}
\begin{proof}
By Lemma \ref{smthyloop}, it is sufficient to prove that for any two smooth based Hamiltonian loops $\pmb{\gamma}=\{\gamma:I\times \X\to \X\}$ and $\pmb{\gamma}'=\{\gamma':I\times \X\to \X\}$, there is a homotopy $\pmb{\Gamma}$ between

$$\pmb{\gamma}\cdot_{cn}\pmb{\gamma}'=\{\gamma\circ db^{L}:[0,\frac{1}{2}]\times \X \to \X,\gamma'\circ db^{R}:[\frac{1}{2},1]\times \X \to \X\},$$
and 
$$\pmb{\gamma}\cdot_{cp}\pmb{\gamma}'=\{\gamma\circ(pr_I\times \gamma'): [0,1]\times \X \to \X\}.$$ 

Let $h_{1}, h_{2}: [0,1]\times [0,\frac{1}{2}]\to [0,1], h_{3}, h_{4}: [0,1]\times[\frac{1}{2},1]\to [0,1]$ be given by:
$$h_{1}(s,t):=(1+s)t,\ \ h_{2}(s,t):=(1-s)t,$$ 
$$h_{3}(s,t):=(1-s)t+s, \ \ h_{4}(s,t):=(1+s)t-s.$$

Define 
$\Gamma_{1}: [0,1]\times[0,\frac{1}{2}]\times\X\to\X$ to be the composition of
$$
(pr_{[0,1]}\times pr_{[0,\frac{1}{2}]}\times (\gamma'\circ pr_{[0,1]\times \X}))\circ((h_{1}\times pr_{[0,\frac{1}{2}]}), Id_{\X}): [0,1]\times[0,\frac{1}{2}]\times\X \to [0,1]\times[0,\frac{1}{2}]\times \X 
$$
and
$$
(\gamma\circ pr_{[0,1]\times \X})\circ ((h_{2}\times pr_{[0,\frac{1}{2}]}),Id_{\X}) : [0,1]\times[0,\frac{1}{2}]\times \X \to \X.
$$

Define 
$\Gamma_{2}: [0,1]\times[\frac{1}{2},1]\times\X\to\X$ to be the composition of
$$
(pr_{[0,1]}\times pr_{[\frac{1}{2},1]}\times (\gamma'\circ pr_{[0,1]\times \X}))\circ((h_{3}\times pr_{[\frac{1}{2},1]}), Id_{\X}): [0,1]\times[\frac{1}{2},1]\times\X \to  [0,1]\times[\frac{1}{2},1]\times \X 
$$
and
$$
(\gamma\circ pr_{[0,1]\times \X})\circ ((h_{4}\times pr_{[\frac{1}{2},1]}), Id_{\X}) :[0,1]\times[\frac{1}{2},1]\times \X \to \X.
$$

Then $\pmb{\Gamma}:=\{\Gamma_{1},\Gamma_{2}\}$ defines a homotopy between $\pmb{\gamma}\cdot_{cn}\pmb{\gamma}'$ and $\pmb{\gamma}\cdot_{cp}\pmb{\gamma}'$.
\end{proof}

We denote the induced product on $\pi_{1}(Ham(\X,\omega))$ by ``$\cdot$". From the ``$\cdot_{cp}$" description, it is easy to see that the product ``$\cdot$'' is associative, and  that every Hamiltonian loop is invertible.  Thus we have:

\begin{prop}
The set $\pi_{1}(Ham(\X,\omega))$ forms a group with the product ``$\cdot$''.
\end{prop}

\begin{defn}\label{defn_fund_gp_Ham}
The group of homotopy equivalence classes of Hamiltonian loops  $(\pi_{1}(Ham(\X,\omega)),\cdot)$ is called the fundamental group of Hamiltonian diffeomorphisms of the symplectic orbifold $(\X,\omega)$.
\end{defn}

We remark that the results in this section can be easily extended to any Lie 2-group using the same idea here.

\subsection{Orbifiber Bundle and Sectional Orbifold morphism}\label{orbifiberbundlegeneral}
In this subsection, we discuss a notion of fiber bundles with fibers being orbifolds.

\begin{defn}\label{orbifiberbundle}
An {\em orbifiber bundle} is a quadruple $(\E, \B, \pi, \F)$, where $\E$, $\B$ and $\F$ are orbifolds, and $\pmb{\pi}: \E \to \B$ is an orbifold morphism, 
satisfying the following local triviality condition: For any point $x\in |\B|$, there is a substack $U_{x}$ of $\B$ s.t. $|U_{x}|$ is an open neighborhood of $x$ and an orbifold diffeomorphism $\pmb{\phi} :\F\times U_{x}  \to \pmb{\pi}^{-1}(U_{x})$,
such that the following diagram is commutative:
$$
\xymatrix{
\gH_{\F,x}\times U_{x} \ar[r]^{\phi} \ar[rd]_{p_{2}} & \pmb{\pi}^{-1}(U_{x}) \ar[d]^{\pmb{\pi}|_{U_{x}}}\\
& U_{x},
}
$$
where $p_{2}$ is the projection to the second component.

The orbifold $\B$ is called the base of the bundle, $\E$ the total orbifold, and $\F$ the fiber. The orbifold morphism $\pmb{\pi}$  is called the  bundle projection.
\end{defn}

\begin{numrmk}
Orbifiber bundles generalize the notion of gerbes over manifold/orbifold by allowing fibers to be orbifolds other than stacky points.
\end{numrmk}

\begin{numrmk}
Obviously every orbifiber bundle $(\E, \B, \pmb{\pi}, \F)$ determines a fiber bundle $(|\E|, |\B|, |\pmb{\pi}|, |\F|)$ where $|\E|$ and $|\F|$ are the underlying topological space of $\E$ and $\F$ respectively, $|\pmb{\pi}|$ is the continuous map induced by $\pmb{\pi}$.
\end{numrmk}

\begin{defn}\label{orbifibrmorphism}
An {\em orbifiber bundle morphism} between $\pmb{\pi}_{1}:\E_{1}\to \B_{1}$ and $\pmb{\pi}_{2}:\E_{2}\to \B_{2}$ is a pair of orbifold morphisms $\pmb{\xi}:\E_{1}\to\E_{2}$ and $\pmb{f}:\B_{1}\to\B_{2}$ which commute with $\pmb{\pi}_{1}$ and $\pmb{\pi}_{2}$.
An orbifiber bundle morphism is called an orbifiber bundle isomorphism if the pair of orbifold morphisms are diffeomorphisms.
\end{defn}

To define Seidel representation, we only need orbifiber bundles over manifolds. Thus in what follows, we assume that the base $B$ of orbifiber bundles are manifolds.  

To construct Seidel representation for symplectic manifolds, one counts $J$-holomorphic sections of the Hamiltonian fibration. In the orbifold case, we need an analogue of sections for orbifiber bundles. A natural candidate is a morphism from the base to the total orbifold such that its precomposition with the projection is the identity of the base up to a natural transformation. We call such morphisms {\em orbisections} of the orbifiber bundle. However in order to define orbifold Gromov-Witten invariants, one needs to count orbicurves with all possible orbifold structures. For this reason we need a more general notion of sections in the orbifold case, defined as follows:

\begin{defn}\label{orbisection}
A \emph{sectional orbifold morphism} of  $\pmb{\pi}:\E\to B$ is an orbifold structure $\B$ on $B$ together with an orbifold morphism $\pmb{s}:\B\to \E$ which lifts a section $s:B\to E$ of the bundle $(E, B, \pi, F)$. 
\end{defn}

The following example illustrates an important point in the notion of sectional orbifold morphisms. 

\begin{example}\label{exmCtwo}
Let $\X$ be the orbifold $[\com/\Ztwo]$ with $\Ztwo$ acting on $\com$ by $\pm1\cdot z=\pm z$. Consider the following groupoid chart $U_{S^{2}}$ of $S^{2}$:
 
 \begin{eqnarray*}
 Ob(U_{S^{2}}) & & :=\begin{aligned}
 \begin{tikzpicture}
 \def\R{1} 
\filldraw[gray, ultra nearly transparent] (0.4*\R,-0.3*\R) arc (-asin(0.6):180+asin(0.6):0.5*\R) arc (180:0:0.4 and 0.06)-- cycle;
\fill[gray, semitransparent] (0,-0.3*\R) ellipse (0.4 and 0.05); 
\shadedraw[shading=ball,ball color=red](0.4*\R,-0.3*\R) arc (-asin(0.6):-180+asin(0.6):0.5*\R) arc (180:360:0.4 and 0.06) -- cycle;
\end{tikzpicture} 
\end{aligned}
\sqcup
\begin{aligned}
\begin{tikzpicture}
 \def\R{1} 
 \shadedraw[shading=ball,ball color=red, white] (0,0) circle (0.5*\R);
 \filldraw[fill=white, draw=black, fill opacity=0] (0,0.5*\R) circle (0.06);
 \filldraw[fill=white, draw=black, fill opacity=0] (0,-0.5*\R) circle (0.06);
\draw[thick, double] (0,0.5*\R) arc (90:-90:0.5*\R);
\end{tikzpicture} 
\end{aligned}
\sqcup
\begin{aligned}
\begin{tikzpicture}
 \def\R{1} 
 \shadedraw[shading=ball,ball color=red, white] (0,0) circle (0.5*\R);
  \filldraw[fill=white, draw=black, fill opacity=0] (0,0.5*\R) circle (0.06);
 \filldraw[fill=white, draw=black, fill opacity=0] (0,-0.5*\R) circle (0.06);
 \draw[thick, double] (0,-0.5*\R) arc (-90:-270:0.5*\R);
\end{tikzpicture} 
\end{aligned}
\sqcup
\begin{aligned}
\begin{tikzpicture}
 \def\R{1} 
\shadedraw[shading=ball,ball color=red](0.4*\R,0.3*\R) arc (asin(0.6):180-asin(0.6):0.5*\R) arc (180:360:0.4 and 0.06) -- cycle;
 \fill[gray, ultra nearly transparent] (0.4*\R,0.3*\R) arc (asin(0.6):-180-asin(0.6):0.5*\R) arc (-180:0:0.4 and 0.06) -- cycle;
\fill[gray, semitransparent] (0,0.3*\R) ellipse (0.4 and 0.05); 
\end{tikzpicture}
\end{aligned}\\
Mor(U_{S^{2}}) & & :=
\begin{aligned}
\begin{tikzpicture}
 \def\R{1} 
\shadedraw[shading=ball,ball color=red](0.4*\R,-0.3*\R) arc (-asin(0.6):-180+asin(0.6):0.5*\R) arc (180:360:0.4 and 0.06) -- cycle;
\filldraw[gray, ultra nearly transparent] (0.4*\R,-0.3*\R) arc (-asin(0.6):180+asin(0.6):0.5*\R) arc (180:0:0.4 and 0.06)-- cycle;

\fill[gray, semitransparent] (0,-0.3*\R) ellipse (0.4 and 0.05); 
\filldraw[fill=white, draw=black, fill opacity=0] (0.4*\R,-0.3*\R) circle (0.06);
\filldraw[fill=white, draw=black, fill opacity=0] (0,-0.5*\R) circle (0.06);
\draw[thick, double] (0.4*\R,-0.3*\R) arc (-asin(0.6):-90:0.5*\R);
\end{tikzpicture} 
\end{aligned}
\sqcup
\begin{aligned}
\begin{tikzpicture}
\def\R{1} 
\shadedraw[shading=ball,ball color=red](0.4*\R,-0.3*\R) arc (-asin(0.6):-180+asin(0.6):0.5*\R) arc (180:360:0.4 and 0.06) -- cycle;
\filldraw[gray, ultra nearly transparent] (0.4*\R,-0.3*\R) arc (-asin(0.6):180+asin(0.6):0.5*\R) arc (180:0:0.4 and 0.06)-- cycle;

\fill[gray, semitransparent] (0,-0.3*\R) ellipse (0.4 and 0.05); 
\filldraw[fill=white, draw=black, fill opacity=0] (-0.4*\R,-0.3*\R) circle (0.06);
\filldraw[fill=white, draw=black, fill opacity=0] (0,-0.5*\R) circle (0.06);
\draw[thick, double] (0,-0.5*\R) arc (-90:-180+asin(0.6):0.5*\R);
\end{tikzpicture} 
\end{aligned}
\sqcup
\begin{aligned}   
\begin{tikzpicture}
\def\R{1} 
\shadedraw[shading=ball,ball color=red, white] (0,0) circle (0.5*\R);
\draw[thick, double] (0,0) circle (0.5*\R);
\end{tikzpicture} 
\end{aligned}
\sqcup
\begin{aligned}  
\begin{tikzpicture}
 \def\R{1} 
\shadedraw[shading=ball,ball color=red](0.4*\R,0.3*\R) arc (asin(0.6):180-asin(0.6):0.5*\R) arc (180:360:0.4 and 0.06) -- cycle;
 \fill[gray, ultra nearly transparent] (0.4*\R,0.3*\R) arc (asin(0.6):-180-asin(0.6):0.5*\R) arc (-180:0:0.4 and 0.06) -- cycle;
\fill[gray, semitransparent] (0,0.3*\R) ellipse (0.4 and 0.05); 
\filldraw[fill=white, draw=black, fill opacity=0] (0.4*\R,0.3*\R) circle (0.06);
\filldraw[fill=white, draw=black, fill opacity=0] (0,0.5*\R) circle (0.06);
\draw[thick, double] (0,0.5*\R) arc (90:asin(0.6):0.5*\R);
\end{tikzpicture}
\end{aligned}
\sqcup
\begin{aligned}
\begin{tikzpicture}
 \def\R{1} 
\shadedraw[shading=ball,ball color=red](0.4*\R,0.3*\R) arc (asin(0.6):180-asin(0.6):0.5*\R) arc (180:360:0.4 and 0.06) -- cycle;
 \fill[gray, ultra nearly transparent] (0.4*\R,0.3*\R) arc (asin(0.6):-180-asin(0.6):0.5*\R) arc (-180:0:0.4 and 0.06) -- cycle;
\fill[gray, semitransparent] (0,0.3*\R) ellipse (0.4 and 0.05); 
\filldraw[fill=white, draw=black, fill opacity=0] (-0.4*\R,0.3*\R) circle (0.06);
\filldraw[fill=white, draw=black, fill opacity=0] (0,0.5*\R) circle (0.06);
\draw[thick, double] (-0.4*\R,0.3*\R) arc (180-asin(0.6):90:0.5*\R);
\end{tikzpicture}
\end{aligned}
\sqcup\\
&& \hspace{1cm} 
\begin{aligned}
\begin{tikzpicture}
 \def\R{1} 
\shadedraw[shading=ball,ball color=red](0.4*\R,-0.3*\R) arc (-asin(0.6):-180+asin(0.6):0.5*\R) arc (180:360:0.4 and 0.06) -- cycle;
 \filldraw[gray, ultra nearly transparent] (0.4*\R,-0.3*\R) arc (-asin(0.6):180+asin(0.6):0.5*\R) arc (180:0:0.4 and 0.06)-- cycle;

\fill[gray, semitransparent] (0,-0.3*\R) ellipse (0.4 and 0.05); 
 \filldraw[fill=white, draw=black, fill opacity=0] (0.4*\R,-0.3*\R) circle (0.06);
 \filldraw[fill=white, draw=black, fill opacity=0] (0,-0.5*\R) circle (0.06);
 \draw[thick, double] (0.4*\R,-0.3*\R) arc (-asin(0.6):-90:0.5*\R);
\end{tikzpicture} 
\end{aligned}
\sqcup
\begin{aligned}
\begin{tikzpicture}
\def\R{1} 
\shadedraw[shading=ball,ball color=red](0.4*\R,-0.3*\R) arc (-asin(0.6):-180+asin(0.6):0.5*\R) arc (180:360:0.4 and 0.06) -- cycle;
\filldraw[gray, ultra nearly transparent] (0.4*\R,-0.3*\R) arc (-asin(0.6):180+asin(0.6):0.5*\R) arc (180:0:0.4 and 0.06)-- cycle;

\fill[gray, semitransparent] (0,-0.3*\R) ellipse (0.4 and 0.05); 
\filldraw[fill=white, draw=black, fill opacity=0] (-0.4*\R,-0.3*\R) circle (0.06);
\filldraw[fill=white, draw=black, fill opacity=0] (0,-0.5*\R) circle (0.06);
\draw[thick, double] (0,-0.5*\R) arc (-90:-180+asin(0.6):0.5*\R);
\end{tikzpicture} 
\end{aligned}
\sqcup
\begin{aligned}
\begin{tikzpicture}
\def\R{1} 
\shadedraw[shading=ball,ball color=red, white] (0,0) circle (0.5*\R);
\draw[thick, double] (0,0) circle (0.5*\R);
\end{tikzpicture} 
\end{aligned}
\sqcup
\begin{aligned}  
\begin{tikzpicture}
 \def\R{1} 
\shadedraw[shading=ball,ball color=red](0.4*\R,0.3*\R) arc (asin(0.6):180-asin(0.6):0.5*\R) arc (180:360:0.4 and 0.06) -- cycle;
 \fill[gray, ultra nearly transparent] (0.4*\R,0.3*\R) arc (asin(0.6):-180-asin(0.6):0.5*\R) arc (-180:0:0.4 and 0.06) -- cycle;
\fill[gray, semitransparent] (0,0.3*\R) ellipse (0.4 and 0.05); 
\filldraw[fill=white, draw=black, fill opacity=0] (0.4*\R,0.3*\R) circle (0.06);
\filldraw[fill=white, draw=black, fill opacity=0] (0,0.5*\R) circle (0.06);
\draw[thick, double] (0,0.5*\R) arc (90:asin(0.6):0.5*\R);
\end{tikzpicture}
\end{aligned}
\sqcup
\begin{aligned}
\begin{tikzpicture}
 \def\R{1} 
\shadedraw[shading=ball,ball color=red](0.4*\R,0.3*\R) arc (asin(0.6):180-asin(0.6):0.5*\R) arc (180:360:0.4 and 0.06) -- cycle;
 \fill[gray, ultra nearly transparent] (0.4*\R,0.3*\R) arc (asin(0.6):-180-asin(0.6):0.5*\R) arc (-180:0:0.4 and 0.06) -- cycle;
\fill[gray, semitransparent] (0,0.3*\R) ellipse (0.4 and 0.05); 
\filldraw[fill=white, draw=black, fill opacity=0] (-0.4*\R,0.3*\R) circle (0.06);
\filldraw[fill=white, draw=black, fill opacity=0] (0,0.5*\R) circle (0.06);
\draw[thick, double] (-0.4*\R,0.3*\R) arc (180-asin(0.6):90:0.5*\R);
\end{tikzpicture}
\end{aligned}\\
&&\hspace{1cm} \sqcup
\begin{aligned}
\begin{tikzpicture}
 \def\R{1} 
\shadedraw[shading=ball,ball color=red](0.4*\R,-0.3*\R) arc (-asin(0.6):-180+asin(0.6):0.5*\R) arc (180:360:0.4 and 0.06) -- cycle;
 \filldraw[gray, ultra nearly transparent] (0.4*\R,-0.3*\R) arc (-asin(0.6):180+asin(0.6):0.5*\R) arc (180:0:0.4 and 0.06)-- cycle;

\fill[gray, semitransparent] (0,-0.3*\R) ellipse (0.4 and 0.05); 
\end{tikzpicture} 
\end{aligned}
 \sqcup
 \begin{aligned}
\begin{tikzpicture}
 \def\R{1} 
 \shadedraw[shading=ball,ball color=red, white] (0,0) circle (0.5*\R);
 \filldraw[fill=white, draw=black, fill opacity=0] (0,0.5*\R) circle (0.06);
 \filldraw[fill=white, draw=black, fill opacity=0] (0,-0.5*\R) circle (0.06);
\draw[thick, double] (0,0.5*\R) arc (90:-90:0.5*\R);
\end{tikzpicture} 
\end{aligned}
 \sqcup
 \begin{aligned}   
\begin{tikzpicture}
 \def\R{1} 
 \shadedraw[shading=ball,ball color=red, white] (0,0) circle (0.5*\R);
  \filldraw[fill=white, draw=black, fill opacity=0] (0,0.5*\R) circle (0.06);
 \filldraw[fill=white, draw=black, fill opacity=0] (0,-0.5*\R) circle (0.06);
 \draw[thick, double] (0,-0.5*\R) arc (-90:-270:0.5*\R);
\end{tikzpicture} 
\end{aligned}
 \sqcup
 \begin{aligned}
\begin{tikzpicture}
 \def\R{1} 
\shadedraw[shading=ball,ball color=red](0.4*\R,0.3*\R) arc (asin(0.6):180-asin(0.6):0.5*\R) arc (180:360:0.4 and 0.06) -- cycle;
 \fill[gray, ultra nearly transparent] (0.4*\R,0.3*\R) arc (asin(0.6):-180-asin(0.6):0.5*\R) arc (-180:0:0.4 and 0.06) -- cycle;
\fill[gray, semitransparent] (0,0.3*\R) ellipse (0.4 and 0.05); 
\end{tikzpicture}
\end{aligned}
\end{eqnarray*}

Define and orbifiber bundle $\pi:\E\to S^2$ with fibers $\X$ as follows. Define the total orbifold $\E$ by the following groupoid chart $\gG_{\E}$:

\begin{eqnarray*}
Ob(\gG_{\E}):=
\begin{aligned}
\begin{tikzpicture}
 \def\R{1} 
\shadedraw[shading=ball,ball color=red](0.4*\R,-0.3*\R) arc (-asin(0.6):-180+asin(0.6):0.5*\R) arc (180:360:0.4 and 0.06) -- cycle;
 \filldraw[gray, ultra nearly transparent] (0.4*\R,-0.3*\R) arc (-asin(0.6):180+asin(0.6):0.5*\R) arc (180:0:0.4 and 0.06)-- cycle;
\fill[gray, semitransparent] (0,-0.3*\R) ellipse (0.4 and 0.05); 
\end{tikzpicture} 
\end{aligned}
\times \com\ && \sqcup (
\begin{aligned}
\begin{tikzpicture}
 \def\R{1} 
\shadedraw[shading=ball,ball color=red](0.5*\R,0) arc (0:-180:0.5*\R) arc (180:360:0.5 and 0.1) -- cycle;
 \fill[gray, ultra nearly transparent] (0.5*\R,0) arc (0:180:0.5*\R) arc (180:0:0.5 and 0.1) -- cycle;
\filldraw[fill=gray, draw=black, semitransparent] (0,0) ellipse (0.5 and 0.1); 
 \filldraw[fill=white, draw=black, fill opacity=0] (0.5*\R,0) circle (0.06);
 \filldraw[fill=white, draw=black, fill opacity=0] (0,-0.5*\R) circle (0.06);
\draw[thick, double] (0.5*\R,0) arc (0:-90:0.5*\R);
\end{tikzpicture} 
\end{aligned}
\times\com\ \sqcup
\begin{aligned}
\begin{tikzpicture}
 \def\R{1} 
\shadedraw[shading=ball,ball color=red](0.5*\R,0) arc (0:180:0.5*\R) arc (180:360:0.5 and 0.1) -- cycle;
 \fill[gray, ultra   nearly transparent] (0.5*\R,0) arc (0:-180:0.5*\R) arc (-180:0:0.5 and 0.1)-- cycle;
\filldraw[fill=gray, draw=gray, semitransparent] (0,0) ellipse (0.5 and 0.1); 
\draw[densely dashed] (0.5*\R,0) arc (0:180:0.5 and 0.1);
\draw (-0.5*\R,0) arc (180:360:0.5 and 0.1);
  \filldraw[fill=white, draw=black, fill opacity=0] (0,0.5*\R) circle (0.06);
 \filldraw[fill=white, draw=black, fill opacity=0] (0.5*\R,0) circle (0.06);
 \draw[thick, double] (0.5*\R,0) arc (0:90:0.5*\R);
\end{tikzpicture} 
\end{aligned}
\times\com)/rel_{Ob,1}\ \sqcup \\
\begin{aligned}
\begin{tikzpicture}
 \def\R{1} 
\shadedraw[shading=ball,ball color=red](0.4*\R,0.3*\R) arc (asin(0.6):180-asin(0.6):0.5*\R) arc (180:360:0.4 and 0.06) -- cycle;
 \fill[gray, ultra nearly transparent] (0.4*\R,0.3*\R) arc (asin(0.6):-180-asin(0.6):0.5*\R) arc (-180:0:0.4 and 0.06) -- cycle;
\fill[gray, semitransparent] (0,0.3*\R) ellipse (0.4 and 0.05); 
\end{tikzpicture}
\end{aligned}
\times\com \ && \sqcup  
(   
\begin{aligned}
\begin{tikzpicture}
\def\R{1} 
\shadedraw[shading=ball,ball color=red](0.5*\R,0) arc (0:-180:0.5*\R) arc (180:360:0.5 and 0.1) -- cycle;
 \fill[gray, ultra nearly transparent] (0.5*\R,0) arc (0:180:0.5*\R) arc (180:0:0.5 and 0.1) -- cycle;
\filldraw[fill=gray, draw=black, semitransparent] (0,0) ellipse (0.5 and 0.1); 
\filldraw[fill=white, draw=black, fill opacity=0] (-0.5*\R,0) circle (0.06);
\filldraw[fill=white, draw=black, fill opacity=0] (0,-0.5*\R) circle (0.06);
\draw[thick, double] (0,-0.5*\R) arc (-90:-180:0.5*\R);
\end{tikzpicture} 
\end{aligned}
\times  \com\ \sqcup 
\begin{aligned}
\begin{tikzpicture}
 \def\R{1} 
\shadedraw[shading=ball,ball color=red](0.5*\R,0) arc (0:180:0.5*\R) arc (180:360:0.5 and 0.1) -- cycle;
 \fill[gray, ultra   nearly transparent] (0.5*\R,0) arc (0:-180:0.5*\R) arc (-180:0:0.5 and 0.1)-- cycle;
\filldraw[fill=gray, draw=gray, semitransparent] (0,0) ellipse (0.5 and 0.1); 
\draw[densely dashed] (0.5*\R,0) arc (0:180:0.5 and 0.1);
\draw (-0.5*\R,0) arc (180:360:0.5 and 0.1);
\filldraw[fill=white, draw=black, fill opacity=0] (0,0.5*\R) circle (0.06);
\filldraw[fill=white, draw=black, fill opacity=0] (-0.5*\R,0) circle (0.06);
\draw[thick, double] (-0.5*\R,0) arc (-180:-270:0.5*\R);
\end{tikzpicture} 
\end{aligned}
\times\com )/ rel_{Ob,2}
\end{eqnarray*}

\begin{eqnarray*}
Mor(\gG_{\E}):=
\begin{aligned}
\begin{tikzpicture}
 \def\R{1} 
\shadedraw[shading=ball,ball color=red](0.4*\R,-0.3*\R) arc (-asin(0.6):-180+asin(0.6):0.5*\R) arc (180:360:0.4 and 0.06) -- cycle;
 \filldraw[gray, ultra nearly transparent] (0.4*\R,-0.3*\R) arc (-asin(0.6):180+asin(0.6):0.5*\R) arc (180:0:0.4 and 0.06)-- cycle;

\fill[gray, semitransparent] (0,-0.3*\R) ellipse (0.4 and 0.05); 
 \filldraw[fill=white, draw=black, fill opacity=0] (0.4*\R,-0.3*\R) circle (0.06);
 \filldraw[fill=white, draw=black, fill opacity=0] (0,-0.5*\R) circle (0.06);
 \draw[thick, double] (0.4*\R,-0.3*\R) arc (-asin(0.6):-90:0.5*\R);
\end{tikzpicture} 
\end{aligned}
\times \com\times\Ztwo\ \sqcup
\begin{aligned}
\begin{tikzpicture}
\def\R{1} 
\shadedraw[shading=ball,ball color=red](0.4*\R,-0.3*\R) arc (-asin(0.6):-180+asin(0.6):0.5*\R) arc (180:360:0.4 and 0.06) -- cycle;
\filldraw[gray, ultra nearly transparent] (0.4*\R,-0.3*\R) arc (-asin(0.6):180+asin(0.6):0.5*\R) arc (180:0:0.4 and 0.06)-- cycle;

\fill[gray, semitransparent] (0,-0.3*\R) ellipse (0.4 and 0.05); 
\filldraw[fill=white, draw=black, fill opacity=0] (-0.4*\R,-0.3*\R) circle (0.06);
\filldraw[fill=white, draw=black, fill opacity=0] (0,-0.5*\R) circle (0.06);
\draw[thick, double] (0,-0.5*\R) arc (-90:-180+asin(0.6):0.5*\R);
\end{tikzpicture} 
\end{aligned}
\times\com\times\Ztwo\ && \sqcup \\
(
\begin{aligned}
\begin{tikzpicture}
 \def\R{1} 
\shadedraw[shading=ball,ball color=red](0.5*\R,0) arc (0:-180:0.5*\R) arc (180:360:0.5 and 0.1) -- cycle;
  \fill[gray, ultra nearly transparent] (0.5*\R,0) arc (0:180:0.5*\R) arc (180:0:0.5 and 0.1) -- cycle;
\filldraw[fill=gray, draw=black,  semitransparent] (0,0) ellipse (0.5 and 0.1); 
 \filldraw[fill=white, draw=black, fill opacity=0] (0.5*\R,0) circle (0.06);
 \filldraw[fill=white, draw=black, fill opacity=0] (-0.5*\R,0) circle (0.06);
\draw[thick, double] (0.5*\R,0) arc (0:-180:0.5*\R);
\end{tikzpicture} 
\end{aligned}
\times\com\times \Ztwo\ && \sqcup
\begin{aligned}
\begin{tikzpicture}
 \def\R{1} 
\shadedraw[shading=ball,ball color=red](0.5*\R,0) arc (0:180:0.5*\R) arc (180:360:0.5 and 0.1) -- cycle;
 \fill[gray, ultra   nearly transparent] (0.5*\R,0) arc (0:-180:0.5*\R) arc (-180:0:0.5 and 0.1)-- cycle;
\filldraw[fill=gray, draw=gray, semitransparent] (0,0) ellipse (0.5 and 0.1); 
\draw[densely dashed] (0.5*\R,0) arc (0:180:0.5 and 0.1);
\draw (-0.5*\R,0) arc (180:360:0.5 and 0.1);  
\filldraw[fill=white, draw=black, fill opacity=0] (-0.5*\R,0) circle (0.06);
 \filldraw[fill=white, draw=black, fill opacity=0] (0.5*\R,0) circle (0.06);
 \draw[thick, double] (0.5*\R,0) arc (0:180:0.5*\R);
\end{tikzpicture} 
\end{aligned}
\times\com\times \Ztwo)/rel_{Mor,1}\\
&& \sqcup
\begin{aligned}
\begin{tikzpicture}
 \def\R{1} 
\shadedraw[shading=ball,ball color=red](0.4*\R,0.3*\R) arc (asin(0.6):180-asin(0.6):0.5*\R) arc (180:360:0.4 and 0.06) -- cycle;
 \fill[gray, ultra nearly transparent] (0.4*\R,0.3*\R) arc (asin(0.6):-180-asin(0.6):0.5*\R) arc (-180:0:0.4 and 0.06) -- cycle;
\fill[gray, semitransparent] (0,0.3*\R) ellipse (0.4 and 0.05); 
\filldraw[fill=white, draw=black, fill opacity=0] (0.4*\R,0.3*\R) circle (0.06);
\filldraw[fill=white, draw=black, fill opacity=0] (0,0.5*\R) circle (0.06);
\draw[thick, double] (0,0.5*\R) arc (90:asin(0.6):0.5*\R);
\end{tikzpicture}
\end{aligned}
\times \com\times\Ztwo\ \sqcup
\begin{aligned}
\begin{tikzpicture}
 \def\R{1} 
\shadedraw[shading=ball,ball color=red](0.4*\R,0.3*\R) arc (asin(0.6):180-asin(0.6):0.5*\R) arc (180:360:0.4 and 0.06) -- cycle;
 \fill[gray, ultra nearly transparent] (0.4*\R,0.3*\R) arc (asin(0.6):-180-asin(0.6):0.5*\R) arc (-180:0:0.4 and 0.06) -- cycle;
\fill[gray, semitransparent] (0,0.3*\R) ellipse (0.4 and 0.05); 
\filldraw[fill=white, draw=black, fill opacity=0] (-0.4*\R,0.3*\R) circle (0.06);
\filldraw[fill=white, draw=black, fill opacity=0] (0,0.5*\R) circle (0.06);
\draw[thick, double] (-0.4*\R,0.3*\R) arc (180-asin(0.6):90:0.5*\R);
\end{tikzpicture}
\end{aligned}
\times\com\times\Ztwo\\
\sqcup
\begin{aligned}
\begin{tikzpicture}
 \def\R{1} 
\shadedraw[shading=ball,ball color=red](0.4*\R,-0.3*\R) arc (-asin(0.6):-180+asin(0.6):0.5*\R) arc (180:360:0.4 and 0.06) -- cycle;
 \filldraw[gray, ultra nearly transparent] (0.4*\R,-0.3*\R) arc (-asin(0.6):180+asin(0.6):0.5*\R) arc (180:0:0.4 and 0.06)-- cycle;

\fill[gray, semitransparent] (0,-0.3*\R) ellipse (0.4 and 0.05); 
 \filldraw[fill=white, draw=black, fill opacity=0] (0.4*\R,-0.3*\R) circle (0.06);
 \filldraw[fill=white, draw=black, fill opacity=0] (0,-0.5*\R) circle (0.06);
 \draw[thick, double] (0.4*\R,-0.3*\R) arc (-asin(0.6):-90:0.5*\R);
\end{tikzpicture} 
\end{aligned}
\times \com\times\Ztwo\ \sqcup
\begin{aligned}
\begin{tikzpicture}
\def\R{1} 
\shadedraw[shading=ball,ball color=red](0.4*\R,-0.3*\R) arc (-asin(0.6):-180+asin(0.6):0.5*\R) arc (180:360:0.4 and 0.06) -- cycle;
\filldraw[gray, ultra nearly transparent] (0.4*\R,-0.3*\R) arc (-asin(0.6):180+asin(0.6):0.5*\R) arc (180:0:0.4 and 0.06)-- cycle;

\fill[gray, semitransparent] (0,-0.3*\R) ellipse (0.4 and 0.05); 
\filldraw[fill=white, draw=black, fill opacity=0] (-0.4*\R,-0.3*\R) circle (0.06);
\filldraw[fill=white, draw=black, fill opacity=0] (0,-0.5*\R) circle (0.06);
\draw[thick, double] (0,-0.5*\R) arc (-90:-180+asin(0.6):0.5*\R);
\end{tikzpicture} 
\end{aligned}
\times\com\times\Ztwo\ &&\sqcup \\
(
\begin{aligned}
\begin{tikzpicture}
 \def\R{1} 
\shadedraw[shading=ball,ball color=red](0.5*\R,0) arc (0:-180:0.5*\R) arc (180:360:0.5 and 0.1) -- cycle;
  \fill[gray, ultra nearly transparent] (0.5*\R,0) arc (0:180:0.5*\R) arc (180:0:0.5 and 0.1) -- cycle;
\filldraw[fill=gray, draw=black,  semitransparent] (0,0) ellipse (0.5 and 0.1); 
 \filldraw[fill=white, draw=black, fill opacity=0] (0.5*\R,0) circle (0.06);
 \filldraw[fill=white, draw=black, fill opacity=0] (-0.5*\R,0) circle (0.06);
\draw[thick, double] (0.5*\R,0) arc (0:-180:0.5*\R);
\end{tikzpicture} 
\end{aligned}
\times\com\times \Ztwo\ && \sqcup
\begin{aligned}
\begin{tikzpicture}
 \def\R{1} 
\shadedraw[shading=ball,ball color=red](0.5*\R,0) arc (0:180:0.5*\R) arc (180:360:0.5 and 0.1) -- cycle;
 \fill[gray, ultra   nearly transparent] (0.5*\R,0) arc (0:-180:0.5*\R) arc (-180:0:0.5 and 0.1)-- cycle;
\filldraw[fill=gray, draw=gray, semitransparent] (0,0) ellipse (0.5 and 0.1); 
\draw[densely dashed] (0.5*\R,0) arc (0:180:0.5 and 0.1);
\draw (-0.5*\R,0) arc (180:360:0.5 and 0.1);  
\filldraw[fill=white, draw=black, fill opacity=0] (-0.5*\R,0) circle (0.06);
 \filldraw[fill=white, draw=black, fill opacity=0] (0.5*\R,0) circle (0.06);
 \draw[thick, double] (0.5*\R,0) arc (0:180:0.5*\R);
\end{tikzpicture} 
\end{aligned}
\times\com\times \Ztwo)/rel_{Mor,2}\\
&& \sqcup
\begin{aligned}
\begin{tikzpicture}
 \def\R{1} 
\shadedraw[shading=ball,ball color=red](0.4*\R,0.3*\R) arc (asin(0.6):180-asin(0.6):0.5*\R) arc (180:360:0.4 and 0.06) -- cycle;
 \fill[gray, ultra nearly transparent] (0.4*\R,0.3*\R) arc (asin(0.6):-180-asin(0.6):0.5*\R) arc (-180:0:0.4 and 0.06) -- cycle;
\fill[gray, semitransparent] (0,0.3*\R) ellipse (0.4 and 0.05); 
\filldraw[fill=white, draw=black, fill opacity=0] (0.4*\R,0.3*\R) circle (0.06);
\filldraw[fill=white, draw=black, fill opacity=0] (0,0.5*\R) circle (0.06);
\draw[thick, double] (0,0.5*\R) arc (90:asin(0.6):0.5*\R);
\end{tikzpicture}
\end{aligned}
\times \com\times\Ztwo\  \sqcup
\begin{aligned}
\begin{tikzpicture}
 \def\R{1} 
\shadedraw[shading=ball,ball color=red](0.4*\R,0.3*\R) arc (asin(0.6):180-asin(0.6):0.5*\R) arc (180:360:0.4 and 0.06) -- cycle;
 \fill[gray, ultra nearly transparent] (0.4*\R,0.3*\R) arc (asin(0.6):-180-asin(0.6):0.5*\R) arc (-180:0:0.4 and 0.06) -- cycle;
\fill[gray, semitransparent] (0,0.3*\R) ellipse (0.4 and 0.05); 
\filldraw[fill=white, draw=black, fill opacity=0] (-0.4*\R,0.3*\R) circle (0.06);
\filldraw[fill=white, draw=black, fill opacity=0] (0,0.5*\R) circle (0.06);
\draw[thick, double] (-0.4*\R,0.3*\R) arc (180-asin(0.6):90:0.5*\R);
\end{tikzpicture}
\end{aligned}
\times\com\times\Ztwo\\
\sqcup
\begin{aligned}
\begin{tikzpicture}
 \def\R{1} 
\shadedraw[shading=ball,ball color=red](0.4*\R,-0.3*\R) arc (-asin(0.6):-180+asin(0.6):0.5*\R) arc (180:360:0.4 and 0.06) -- cycle;
 \filldraw[gray, ultra nearly transparent] (0.4*\R,-0.3*\R) arc (-asin(0.6):180+asin(0.6):0.5*\R) arc (180:0:0.4 and 0.06)-- cycle;

\fill[gray, semitransparent] (0,-0.3*\R) ellipse (0.4 and 0.05); 
\end{tikzpicture} 
\end{aligned}
\times \com\times\Ztwo\  \sqcup (
\begin{aligned}
\begin{tikzpicture}
 \def\R{1} 
\shadedraw[shading=ball,ball color=red](0.5*\R,0) arc (0:-180:0.5*\R) arc (180:360:0.5 and 0.1) -- cycle;
  \fill[gray, ultra nearly transparent] (0.5*\R,0) arc (0:180:0.5*\R) arc (180:0:0.5 and 0.1) -- cycle;
\filldraw[fill=gray, draw=black, semitransparent] (0,0) ellipse (0.5 and 0.1); 
 \filldraw[fill=white, draw=black, fill opacity=0] (0.5*\R,0) circle (0.06);
 \filldraw[fill=white, draw=black, fill opacity=0] (0,-0.5*\R) circle (0.06);
\draw[thick, double] (0.5*\R,0) arc (0:-90:0.5*\R);
\end{tikzpicture} 
\end{aligned}
\times\com\times\Ztwo\ && \sqcup
\begin{aligned}
\begin{tikzpicture}
 \def\R{1} 
\shadedraw[shading=ball,ball color=red](0.5*\R,0) arc (0:180:0.5*\R) arc (180:360:0.5 and 0.1) -- cycle;
 \fill[gray, ultra   nearly transparent] (0.5*\R,0) arc (0:-180:0.5*\R) arc (-180:0:0.5 and 0.1)-- cycle;
\filldraw[fill=gray, draw=gray, semitransparent] (0,0) ellipse (0.5 and 0.1); 
\draw[densely dashed] (0.5*\R,0) arc (0:180:0.5 and 0.1);
\draw (-0.5*\R,0) arc (180:360:0.5 and 0.1);  
\filldraw[fill=white, draw=black, fill opacity=0] (0,0.5*\R) circle (0.06);
 \filldraw[fill=white, draw=black, fill opacity=0] (0.5*\R,0) circle (0.06);
 \draw[thick, double] (0.5*\R,0) arc (0:90:0.5*\R);
\end{tikzpicture} 
\end{aligned}
\times\com\times\Ztwo)/rel_{Mor,3}\\
\ \sqcup 
\begin{aligned}
\begin{tikzpicture}
 \def\R{1} 
\shadedraw[shading=ball,ball color=red](0.4*\R,0.3*\R) arc (asin(0.6):180-asin(0.6):0.5*\R) arc (180:360:0.4 and 0.06) -- cycle;
 \fill[gray, ultra nearly transparent] (0.4*\R,0.3*\R) arc (asin(0.6):-180-asin(0.6):0.5*\R) arc (-180:0:0.4 and 0.06) -- cycle;
\fill[gray, semitransparent] (0,0.3*\R) ellipse (0.4 and 0.05); 
\end{tikzpicture}
\end{aligned}
 \times\com \times\Ztwo\sqcup 
(   
\begin{aligned}
\begin{tikzpicture}
\def\R{1} 
\shadedraw[shading=ball,ball color=red](0.5*\R,0) arc (0:-180:0.5*\R) arc (180:360:0.5 and 0.1) -- cycle;
 \fill[gray, ultra nearly transparent] (0.5*\R,0) arc (0:180:0.5*\R) arc (180:0:0.5 and 0.1) -- cycle;
\filldraw[fill=gray, draw=black,  semitransparent] (0,0) ellipse (0.5 and 0.1); 
\filldraw[fill=white, draw=black, fill opacity=0] (-0.5*\R,0) circle (0.06);
\filldraw[fill=white, draw=black, fill opacity=0] (0,-0.5*\R) circle (0.06);
\draw[thick, double] (0,-0.5*\R) arc (-90:-180:0.5*\R);
\end{tikzpicture} 
\end{aligned}
\times  \com\times\Ztwo\  && \sqcup 
\begin{aligned}
\begin{tikzpicture}
 \def\R{1} 
\shadedraw[shading=ball,ball color=red](0.5*\R,0) arc (0:180:0.5*\R) arc (180:360:0.5 and 0.1) -- cycle;
 \fill[gray, ultra   nearly transparent] (0.5*\R,0) arc (0:-180:0.5*\R) arc (-180:0:0.5 and 0.1)-- cycle;
\filldraw[fill=gray, draw=gray, semitransparent] (0,0) ellipse (0.5 and 0.1); 
\draw[densely dashed] (0.5*\R,0) arc (0:180:0.5 and 0.1);
\draw (-0.5*\R,0) arc (180:360:0.5 and 0.1);
\filldraw[fill=white, draw=black, fill opacity=0] (0,0.5*\R) circle (0.06);
\filldraw[fill=white, draw=black, fill opacity=0] (-0.5*\R,0) circle (0.06);
\draw[thick, double] (-0.5*\R,0) arc (-180:-270:0.5*\R);
\end{tikzpicture} 
\end{aligned}
\times\com\times \Ztwo)/ rel_{Mor,4}.
\end{eqnarray*}

For notational convenience, we identify hemispheres with unit disks.

Now we specify the glueing along the boundary:
\begin{itemize}
\item $rel_{Ob,1}$: for $(e^{i\theta},z)$ in the boundary of 
$
\begin{aligned}
\begin{tikzpicture}
 \def\R{1} 
\shadedraw[shading=ball,ball color=red](0.5*\R,0) arc (0:-180:0.5*\R) arc (180:360:0.5 and 0.1) -- cycle;
  \fill[gray, ultra nearly transparent] (0.5*\R,0) arc (0:180:0.5*\R) arc (180:0:0.5 and 0.1) -- cycle;
\filldraw[fill=gray, draw=black,  semitransparent] (0,0) ellipse (0.5 and 0.1); 
 \filldraw[fill=white, draw=black, fill opacity=0] (0.5*\R,0) circle (0.06);
 \filldraw[fill=white, draw=black, fill opacity=0] (0,-0.5*\R) circle (0.06);
\draw[thick, double] (0.5*\R,0) arc (0:-90:0.5*\R);
\end{tikzpicture} 
\end{aligned}
\times\com$
, $(e^{i\theta'},z')$ in the boundary of 
$
\begin{aligned}
\begin{tikzpicture}
 \def\R{1} 
\shadedraw[shading=ball,ball color=red](0.5*\R,0) arc (0:180:0.5*\R) arc (180:360:0.5 and 0.1) -- cycle;
 \fill[gray, ultra   nearly transparent] (0.5*\R,0) arc (0:-180:0.5*\R) arc (-180:0:0.5 and 0.1)-- cycle;
\filldraw[fill=gray, draw=gray, semitransparent] (0,0) ellipse (0.5 and 0.1); 
\draw[densely dashed] (0.5*\R,0) arc (0:180:0.5 and 0.1);
\draw (-0.5*\R,0) arc (180:360:0.5 and 0.1);  
\filldraw[fill=white, draw=black, fill opacity=0] (0,0.5*\R) circle (0.06);
 \filldraw[fill=white, draw=black, fill opacity=0] (0.5*\R,0) circle (0.06);
 \draw[thick, double] (0.5*\R,0) arc (0:90:0.5*\R);
\end{tikzpicture} 
\end{aligned}
\times\com$, $(e^{i\theta},z)\sim (e^{i\theta'},z')$ i.f.f. $\theta'=-\theta$, $z'=e^{-i\frac{\theta}{2}}z$. Note that here $\theta\in(0,2\pi)$ and $\theta'\in (-2\pi,0)$.

\item $rel_{Ob,2}$: for $(e^{i\theta},z)$ in the boundary of 
$
\begin{aligned}
\begin{tikzpicture}
\def\R{1} 
\shadedraw[shading=ball,ball color=red](0.5*\R,0) arc (0:-180:0.5*\R) arc (180:360:0.5 and 0.1) -- cycle;
 \fill[gray, ultra nearly transparent] (0.5*\R,0) arc (0:180:0.5*\R) arc (180:0:0.5 and 0.1) -- cycle;
\filldraw[fill=gray, draw=black,  semitransparent] (0,0) ellipse (0.5 and 0.1); 
\filldraw[fill=white, draw=black, fill opacity=0] (-0.5*\R,0) circle (0.06);
\filldraw[fill=white, draw=black, fill opacity=0] (0,-0.5*\R) circle (0.06);
\draw[thick, double] (0,-0.5*\R) arc (-90:-180:0.5*\R);
\end{tikzpicture} 
\end{aligned}
\times \com$, $(e^{i\theta'},z')$ in the boundary of 
$
\begin{aligned}
\begin{tikzpicture}
 \def\R{1} 
\shadedraw[shading=ball,ball color=red](0.5*\R,0) arc (0:180:0.5*\R) arc (180:360:0.5 and 0.1) -- cycle;
 \fill[gray, ultra   nearly transparent] (0.5*\R,0) arc (0:-180:0.5*\R) arc (-180:0:0.5 and 0.1)-- cycle;
\filldraw[fill=gray, draw=gray, semitransparent] (0,0) ellipse (0.5 and 0.1); 
\draw[densely dashed] (0.5*\R,0) arc (0:180:0.5 and 0.1);
\draw (-0.5*\R,0) arc (180:360:0.5 and 0.1);
\filldraw[fill=white, draw=black, fill opacity=0] (0,0.5*\R) circle (0.06);
\filldraw[fill=white, draw=black, fill opacity=0] (-0.5*\R,0) circle (0.06);
\draw[thick, double] (-0.5*\R,0) arc (-180:-270:0.5*\R);
\end{tikzpicture}
\end{aligned} 
\times \com$,
 $(e^{i\theta},z)\sim (e^{i\theta'},z')$ i.f.f. $\theta'=-\theta$, $z'=e^{-i\frac{\theta}{2}}z$. Note that here $\theta,\theta'\in (-\pi,\pi)$.

 \item $rel_{Mor,i}, \ i=1,2,3,4$: for boundary elements $(e^{i\theta_{1}}\to e^{i\theta_{2}}, \xymatrix{z \ar[r]^{g} & g\cdot z})$ and $(e^{i\theta_{1}'}\to e^{i\theta_{2}'}, \xymatrix{z' \ar[r]^{g'} & g'\cdot z'})$, $(e^{i\theta_{1}}\to e^{i\theta_{2}}, \xymatrix{z \ar[r]^{g} & g\cdot z}) \sim(e^{i\theta_{1}'}\to e^{i\theta_{2}'}, \xymatrix{z' \ar[r]^{g'} & g'\cdot z'})$ if and only if 
$$\theta_{1}=-\theta_{1}',\ \ \theta_{2}=-\theta_{2}',$$ 
$$z'=e^{-i\frac{\theta_{1}}{2}}z,\  g'\cdot z'=e^{-i\frac{\theta_{2}}{2}}g\cdot z.$$
\end{itemize}

We remark that $(e^{i\theta}\to e^{i(\theta+2\pi)}, \xymatrix{z \ar[r]^{g} & g\cdot z}) \sim(e^{-i\theta}\to e^{-i(\theta+2\pi)}, \xymatrix{e^{-i\frac{\theta}{2}}z\ar[r]^{-g} & -g\cdot e^{-i\frac{\theta}{2}}z})$.
 
Then the projection $\pmb{\pi}:\E\to \B$ is defined by the obvious projection to the first component.
\end{example}

For $\pmb{\pi}:\E\to B$ in Example \ref{exmCtwo}, we have the following result about its sectional morphisms:
\begin{prop}\label{constantsection}
 If $\pmb{s}: (S^{2}_{orb},p)\to \E$ is a sectional orbifold morphism from an orbisphere with one orbipoint (possibly a trivial one) to the orbifiber bundle in Example \ref{exmCtwo}, and $\pmb{s}$ lifts the zero section in $|\pmb{\pi}|:|\E|\to S^{2}$, then the orbifold structure group at $p$ is $\Ztwo$. In particular, there is no sectional morphism from the smooth sphere $S^{2}$ to $\E$ which lifts the zero section.
\end{prop}
\begin{proof}
Given any orbifold morphism from a sphere with at most one orbifold point at $p$, it determines an orbifold morphism $\check{\pmb{s}}$ away from $p$ by restriction. Without loss of generality, we may assume $p$ to be the north pole of the sphere. Since there is no other orbifold point on the sphere and $\pmb{s}$ lifts the zero section, $\check{\pmb{s}}$ can be represented by the groupoid morphism:

\begin{eqnarray*}
\check{\pmb{s}}_{0}=(id,0): \ Ob(U_{\check{S}^{2}})=
\begin{aligned}
\begin{tikzpicture}
 \def\R{1} 
\shadedraw[shading=ball,ball color=red](0.4*\R,-0.3*\R) arc (-asin(0.6):-180+asin(0.6):0.5*\R) arc (180:360:0.4 and 0.06) -- cycle;
 \filldraw[gray, ultra nearly transparent] (0.4*\R,-0.3*\R) arc (-asin(0.6):180+asin(0.6):0.5*\R) arc (180:0:0.4 and 0.06)-- cycle;

\fill[gray, semitransparent] (0,-0.3*\R) ellipse (0.4 and 0.05); 
\end{tikzpicture} 
\end{aligned}
 \sqcup  
 \begin{aligned}
\begin{tikzpicture}
 \def\R{1} 
 \shadedraw[shading=ball,ball color=red, white] (0,0) circle (0.5*\R);
 \filldraw[fill=white, draw=black, fill opacity=0] (0,0.5*\R) circle (0.06);
 \filldraw[fill=white, draw=black, fill opacity=0] (0,-0.5*\R) circle (0.06);
\draw[thick, double] (0,0.5*\R) arc (90:-90:0.5*\R);
\end{tikzpicture} 
\end{aligned}
\sqcup 
\begin{aligned}
\begin{tikzpicture}
 \def\R{1} 
 \shadedraw[shading=ball,ball color=red, white] (0,0) circle (0.5*\R);
  \filldraw[fill=white, draw=black, fill opacity=0] (0,0.5*\R) circle (0.06);
 \filldraw[fill=white, draw=black, fill opacity=0] (0,-0.5*\R) circle (0.06);
 \draw[thick, double] (0,-0.5*\R) arc (-90:-270:0.5*\R);
\end{tikzpicture} 
\end{aligned}
\ \ \ \ \ \ \ \ \ \ \ \ \ \ \ \ \ \ \ \ \ \ \ \ \ \ \ \ \ \ \ \ \ \ \ \ \ \ \ \ \ \ \ \ \longrightarrow\ \ \ \ \ \ \ \ \ \ \ \ Ob(\gG_{\E})
\end{eqnarray*}

\begin{eqnarray*}
\check{\pmb{s}}_{1}=(id,0,\eta): \ Mor(U_{\check{S}^{2}})=
\begin{aligned}
\begin{tikzpicture}
 \def\R{1} 
\shadedraw[shading=ball,ball color=red](0.4*\R,-0.3*\R) arc (-asin(0.6):-180+asin(0.6):0.5*\R) arc (180:360:0.4 and 0.06) -- cycle;
 \filldraw[gray, ultra nearly transparent] (0.4*\R,-0.3*\R) arc (-asin(0.6):180+asin(0.6):0.5*\R) arc (180:0:0.4 and 0.06)-- cycle;

\fill[gray, semitransparent] (0,-0.3*\R) ellipse (0.4 and 0.05); 
 \filldraw[fill=white, draw=black, fill opacity=0] (0.4*\R,-0.3*\R) circle (0.06);
 \filldraw[fill=white, draw=black, fill opacity=0] (0,-0.5*\R) circle (0.06);
 \draw[thick, double] (0.4*\R,-0.3*\R) arc (-asin(0.6):-90:0.5*\R);
\end{tikzpicture} 
\end{aligned}
\sqcup    
\begin{aligned}
\begin{tikzpicture}
\def\R{1} 
\shadedraw[shading=ball,ball color=red](0.4*\R,-0.3*\R) arc (-asin(0.6):-180+asin(0.6):0.5*\R) arc (180:360:0.4 and 0.06) -- cycle;
\filldraw[gray, ultra nearly transparent] (0.4*\R,-0.3*\R) arc (-asin(0.6):180+asin(0.6):0.5*\R) arc (180:0:0.4 and 0.06)-- cycle;

\fill[gray, semitransparent] (0,-0.3*\R) ellipse (0.4 and 0.05); 
\filldraw[fill=white, draw=black, fill opacity=0] (-0.4*\R,-0.3*\R) circle (0.06);
\filldraw[fill=white, draw=black, fill opacity=0] (0,-0.5*\R) circle (0.06);
\draw[thick, double] (0,-0.5*\R) arc (-90:-180+asin(0.6):0.5*\R);
\end{tikzpicture} 
\end{aligned}
\sqcup   
\begin{aligned}
\begin{tikzpicture}
\def\R{1} 
\shadedraw[shading=ball,ball color=red, white] (0,0) circle (0.5*\R);
\draw[thick, double] (0,0) circle (0.5*\R);
\end{tikzpicture} 
\end{aligned}
&& \sqcup\\
\begin{aligned}
\begin{tikzpicture}
 \def\R{1} 
\shadedraw[shading=ball,ball color=red](0.4*\R,-0.3*\R) arc (-asin(0.6):-180+asin(0.6):0.5*\R) arc (180:360:0.4 and 0.06) -- cycle;
 \filldraw[gray, ultra nearly transparent] (0.4*\R,-0.3*\R) arc (-asin(0.6):180+asin(0.6):0.5*\R) arc (180:0:0.4 and 0.06)-- cycle;

\fill[gray, semitransparent] (0,-0.3*\R) ellipse (0.4 and 0.05); 
 \filldraw[fill=white, draw=black, fill opacity=0] (0.4*\R,-0.3*\R) circle (0.06);
 \filldraw[fill=white, draw=black, fill opacity=0] (0,-0.5*\R) circle (0.06);
 \draw[thick, double] (0.4*\R,-0.3*\R) arc (-asin(0.6):-90:0.5*\R);
\end{tikzpicture} 
\end{aligned}
\sqcup
\begin{aligned}
\begin{tikzpicture}
\def\R{1} 
\shadedraw[shading=ball,ball color=red](0.4*\R,-0.3*\R) arc (-asin(0.6):-180+asin(0.6):0.5*\R) arc (180:360:0.4 and 0.06) -- cycle;
\filldraw[gray, ultra nearly transparent] (0.4*\R,-0.3*\R) arc (-asin(0.6):180+asin(0.6):0.5*\R) arc (180:0:0.4 and 0.06)-- cycle;

\fill[gray, semitransparent] (0,-0.3*\R) ellipse (0.4 and 0.05); 
\filldraw[fill=white, draw=black, fill opacity=0] (-0.4*\R,-0.3*\R) circle (0.06);
\filldraw[fill=white, draw=black, fill opacity=0] (0,-0.5*\R) circle (0.06);
\draw[thick, double] (0,-0.5*\R) arc (-90:-180+asin(0.6):0.5*\R);
\end{tikzpicture} 
\end{aligned}
&& \sqcup 
\begin{aligned}
\begin{tikzpicture}
\def\R{1} 
\shadedraw[shading=ball,ball color=red, white] (0,0) circle (0.5*\R);
\draw[thick, double] (0,0) circle (0.5*\R);
\end{tikzpicture} 
\end{aligned}
\ \ \ \ \ \ \ \ \ \ \ \ \ \ \ \ \ \ \ \longrightarrow\ \ \ \ \ \ \ \  Mor(\gG_{\E})\\
\sqcup
\begin{aligned}
\begin{tikzpicture}
 \def\R{1} 
\shadedraw[shading=ball,ball color=red](0.4*\R,-0.3*\R) arc (-asin(0.6):-180+asin(0.6):0.5*\R) arc (180:360:0.4 and 0.06) -- cycle;
 \filldraw[gray, ultra nearly transparent] (0.4*\R,-0.3*\R) arc (-asin(0.6):180+asin(0.6):0.5*\R) arc (180:0:0.4 and 0.06)-- cycle;

\fill[gray, semitransparent] (0,-0.3*\R) ellipse (0.4 and 0.05); 
\end{tikzpicture} 
\end{aligned}
&& \sqcup
\begin{aligned}
\begin{tikzpicture}
 \def\R{1} 
 \shadedraw[shading=ball,ball color=red, white] (0,0) circle (0.5*\R);
 \filldraw[fill=white, draw=black, fill opacity=0] (0,0.5*\R) circle (0.06);
 \filldraw[fill=white, draw=black, fill opacity=0] (0,-0.5*\R) circle (0.06);
\draw[thick, double] (0,0.5*\R) arc (90:-90:0.5*\R);
\end{tikzpicture} 
\end{aligned}
\sqcup   
\begin{aligned}
\begin{tikzpicture}
 \def\R{1} 
 \shadedraw[shading=ball,ball color=red, white] (0,0) circle (0.5*\R);
  \filldraw[fill=white, draw=black, fill opacity=0] (0,0.5*\R) circle (0.06);
 \filldraw[fill=white, draw=black, fill opacity=0] (0,-0.5*\R) circle (0.06);
 \draw[thick, double] (0,-0.5*\R) arc (-90:-270:0.5*\R);
\end{tikzpicture} 
\end{aligned}
\end{eqnarray*}
where 
\begin{eqnarray*}
\eta(re^{i\theta}\to re^{i\theta'})= && -1\in \Ztwo \ \ \ \ \ \  for \ \ re^{i\theta}\to re^{i\theta'}\in
\begin{aligned}
\begin{tikzpicture}
 \def\R{1} 
\shadedraw[shading=ball,ball color=red](0.5*\R,0) arc (0:180:0.5*\R) arc (180:360:0.5 and 0.1) -- cycle;
\fill[fill=gray,  semitransparent] (0,0) ellipse (0.5 and 0.1); 
 \fill[gray, ultra nearly transparent] (-0.5*\R,0) arc (-180:0:0.5*\R) arc (0:-180:0.5 and 0.1) -- cycle;
 \draw[gray, loosely dashed] (0.5*\R,0) arc (0:180:0.5 and 0.1);
 \filldraw[fill=white, draw=black, fill opacity=0] (0.5*\R,0) circle (0.06);
 \filldraw[fill=white, draw=black, fill opacity=0] (-0.5*\R,0) circle (0.06);
\draw[thick, double] (0.5*\R,0) arc (0:180:0.5*\R);
\end{tikzpicture} 
\end{aligned}
\ \ \ \ (where \ \theta'=\theta\pm2\pi)\\
&& 1\ \ \ \in \Ztwo \ \ \ \ \ \ \ \ \  everywhere\ \ else.
\end{eqnarray*}

To extend  $\check{\pmb{s}}$ to $p$, we define $\pmb{s}$ on
$$
\begin{aligned}
\begin{tikzpicture}
 \def\R{1} 
\shadedraw[shading=ball,ball color=red](0.4*\R,0.3*\R) arc (asin(0.6):180-asin(0.6):0.5*\R) arc (180:360:0.4 and 0.06) -- cycle;
 \fill[gray, ultra nearly transparent] (0.4*\R,0.3*\R) arc (asin(0.6):-180-asin(0.6):0.5*\R) arc (-180:0:0.4 and 0.06) -- cycle;
\fill[gray, semitransparent] (0,0.3*\R) ellipse (0.4 and 0.05); 
\end{tikzpicture}
\end{aligned}
\times \Ztwo \twoarrows 
\begin{aligned}
\begin{tikzpicture}
 \def\R{1} 
\shadedraw[shading=ball,ball color=red](0.4*\R,0.3*\R) arc (asin(0.6):180-asin(0.6):0.5*\R) arc (180:360:0.4 and 0.06) -- cycle;
 \fill[gray, ultra nearly transparent] (0.4*\R,0.3*\R) arc (asin(0.6):-180-asin(0.6):0.5*\R) arc (-180:0:0.4 and 0.06) -- cycle;
\fill[gray, semitransparent] (0,0.3*\R) ellipse (0.4 and 0.05); 
\end{tikzpicture}
\end{aligned}
$$
as: 
$$\pmb{s}_{0}(re^{i\theta})=(re^{i2\theta},0),$$
$$\pmb{s}_{1}(\xymatrix{re^{i\theta}\ar[r]^{\pm 1} & \pm 1\cdot re^{i\theta}})=\xymatrix{(re^{i2\theta},0)\ar[r]^{\pm1} & (re^{i2\theta},0)}.$$
It is straightforward to verify the morphism above together with $\check{\pmb{s}}$ define a morphism from an orbisphere with $\Ztwo$-orbipoint at $p$ to $\gG_{\E}$. 

By uniqueness of this extension, this is the only morphism satisfying the given condition. Actually, if there is no orbifold structure at $p$, then one restrict $\pmb{s}$ to 
$
\begin{aligned}
\begin{tikzpicture}
 \def\R{1} 
\shadedraw[shading=ball,ball color=red](0.4*\R,0.3*\R) arc (asin(0.6):180-asin(0.6):0.5*\R) arc (180:360:0.4 and 0.06) -- cycle;
 \fill[gray, ultra nearly transparent] (0.4*\R,0.3*\R) arc (asin(0.6):-180-asin(0.6):0.5*\R) arc (-180:0:0.4 and 0.06) -- cycle;
\fill[gray, semitransparent] (0,0.3*\R) ellipse (0.4 and 0.05); 
\end{tikzpicture}
\end{aligned}
$.  Consider arrows:

$a_{1}:=re^{i\theta}\to re^{i\theta}\in
\begin{aligned}
\begin{tikzpicture}
 \def\R{1} 
\shadedraw[shading=ball,ball color=red](0.4*\R,0.3*\R) arc (asin(0.6):180-asin(0.6):0.5*\R) arc (180:360:0.4 and 0.06) -- cycle;
 \fill[gray, ultra nearly transparent] (0.4*\R,0.3*\R) arc (asin(0.6):-180-asin(0.6):0.5*\R) arc (-180:0:0.4 and 0.06) -- cycle;
\fill[gray, semitransparent] (0,0.3*\R) ellipse (0.4 and 0.05); 
\filldraw[fill=white, draw=black, fill opacity=0] (0.4*\R,0.3*\R) circle (0.06);
\filldraw[fill=white, draw=black, fill opacity=0] (0,0.5*\R) circle (0.06);
\draw[thick, double] (0,0.5*\R) arc (90:asin(0.6):0.5*\R);
\end{tikzpicture}
\end{aligned}
=
\begin{aligned}
\begin{tikzpicture}
 \def\R{1} 
\shadedraw[shading=ball,ball color=red](0.4*\R,0.3*\R) arc (asin(0.6):180-asin(0.6):0.5*\R) arc (180:360:0.4 and 0.06) -- cycle;
 \fill[gray, ultra nearly transparent] (0.4*\R,0.3*\R) arc (asin(0.6):-180-asin(0.6):0.5*\R) arc (-180:0:0.4 and 0.06) -- cycle;
\fill[gray, semitransparent] (0,0.3*\R) ellipse (0.4 and 0.05); 
\filldraw[fill=white, draw=black, fill opacity=0] (0.4*\R,0.3*\R) circle (0.06);
\filldraw[fill=white, draw=black, fill opacity=0] (0,0.5*\R) circle (0.06);
\draw[thick, double] (0,0.5*\R) arc (90:asin(0.6):0.5*\R);
\end{tikzpicture}
\end{aligned}
\times_{S^{2}}
\begin{aligned}
\begin{tikzpicture}
 \def\R{1} 
\shadedraw[shading=ball,ball color=red](0.4*\R,0.3*\R) arc (asin(0.6):180-asin(0.6):0.5*\R) arc (180:360:0.4 and 0.06) -- cycle;
 \fill[gray, ultra nearly transparent] (0.4*\R,0.3*\R) arc (asin(0.6):-180-asin(0.6):0.5*\R) arc (-180:0:0.4 and 0.06) -- cycle;
\fill[gray, semitransparent] (0,0.3*\R) ellipse (0.4 and 0.05); 
\end{tikzpicture}
\end{aligned}
, \ \ \theta\in (0,2\pi),$

$a_{2}:=re^{i\theta' }\to re^{i\theta'}\in
\begin{aligned}
\begin{tikzpicture}
 \def\R{1} 
\shadedraw[shading=ball,ball color=red](0.4*\R,0.3*\R) arc (asin(0.6):180-asin(0.6):0.5*\R) arc (180:360:0.4 and 0.06) -- cycle;
 \fill[gray, ultra nearly transparent] (0.4*\R,0.3*\R) arc (asin(0.6):-180-asin(0.6):0.5*\R) arc (-180:0:0.4 and 0.06) -- cycle;
\fill[gray, semitransparent] (0,0.3*\R) ellipse (0.4 and 0.05); 
\filldraw[fill=white, draw=black, fill opacity=0] (-0.4*\R,0.3*\R) circle (0.06);
\filldraw[fill=white, draw=black, fill opacity=0] (0,0.5*\R) circle (0.06);
\draw[thick, double] (-0.4*\R,0.3*\R) arc (180-asin(0.6):90:0.5*\R);
\end{tikzpicture}
\end{aligned}
=\begin{aligned}\begin{tikzpicture}
 \def\R{1} 
\shadedraw[shading=ball,ball color=red](0.4*\R,0.3*\R) arc (asin(0.6):180-asin(0.6):0.5*\R) arc (180:360:0.4 and 0.06) -- cycle;
 \fill[gray, ultra nearly transparent] (0.4*\R,0.3*\R) arc (asin(0.6):-180-asin(0.6):0.5*\R) arc (-180:0:0.4 and 0.06) -- cycle;
\fill[gray, semitransparent] (0,0.3*\R) ellipse (0.4 and 0.05); 
\end{tikzpicture}\end{aligned}
\times_{S^{2}}
\begin{aligned}
\begin{tikzpicture}
 \def\R{1} 
\shadedraw[shading=ball,ball color=red](0.4*\R,0.3*\R) arc (asin(0.6):180-asin(0.6):0.5*\R) arc (180:360:0.4 and 0.06) -- cycle;
 \fill[gray, ultra nearly transparent] (0.4*\R,0.3*\R) arc (asin(0.6):-180-asin(0.6):0.5*\R) arc (-180:0:0.4 and 0.06) -- cycle;
\fill[gray, semitransparent] (0,0.3*\R) ellipse (0.4 and 0.05); 
\filldraw[fill=white, draw=black, fill opacity=0] (-0.4*\R,0.3*\R) circle (0.06);
\filldraw[fill=white, draw=black, fill opacity=0] (0,0.5*\R) circle (0.06);
\draw[thick, double] (-0.4*\R,0.3*\R) arc (180-asin(0.6):90:0.5*\R);
\end{tikzpicture}\end{aligned}
, \ \ \ \ \theta'\in (-\pi,\pi),$

When $re^{i\theta'}=re^{i\theta}$:

$a_{3}:=a_{2}\circ a_{1}\in
\begin{aligned}
\begin{tikzpicture}
 \def\R{1} 
\shadedraw[shading=ball,ball color=red](0.4*\R,0.3*\R) arc (asin(0.6):180-asin(0.6):0.5*\R) arc (180:360:0.4 and 0.06) -- cycle;
 \fill[gray, ultra nearly transparent] (0.4*\R,0.3*\R) arc (asin(0.6):-180-asin(0.6):0.5*\R) arc (-180:0:0.4 and 0.06) -- cycle;
\fill[gray, semitransparent] (0,0.3*\R) ellipse (0.4 and 0.05); 
\filldraw[fill=white, draw=black, fill opacity=0] (0.4*\R,0.3*\R) circle (0.06);
\filldraw[fill=white, draw=black, fill opacity=0] (0,0.5*\R) circle (0.06);
\filldraw[fill=white, draw=black, fill opacity=0] (-0.4*\R,0.3*\R) circle (0.06);
\draw[thick, double] (-0.4*\R,0.3*\R) arc (180-asin(0.6):asin(0.6):0.5*\R);
\end{tikzpicture}\end{aligned}
=
\begin{aligned}
\begin{tikzpicture}
 \def\R{1} 
\shadedraw[shading=ball,ball color=red](0.4*\R,0.3*\R) arc (asin(0.6):180-asin(0.6):0.5*\R) arc (180:360:0.4 and 0.06) -- cycle;
 \fill[gray, ultra nearly transparent] (0.4*\R,0.3*\R) arc (asin(0.6):-180-asin(0.6):0.5*\R) arc (-180:0:0.4 and 0.06) -- cycle;
\fill[gray, semitransparent] (0,0.3*\R) ellipse (0.4 and 0.05); 
\filldraw[fill=white, draw=black, fill opacity=0] (0.4*\R,0.3*\R) circle (0.06);
\filldraw[fill=white, draw=black, fill opacity=0] (0,0.5*\R) circle (0.06);
\draw[thick, double] (0,0.5*\R) arc (90:asin(0.6):0.5*\R);
\end{tikzpicture}
\end{aligned}
\times_{S^{2}}
\begin{aligned}
\begin{tikzpicture}
 \def\R{1} 
\shadedraw[shading=ball,ball color=red](0.4*\R,0.3*\R) arc (asin(0.6):180-asin(0.6):0.5*\R) arc (180:360:0.4 and 0.06) -- cycle;
 \fill[gray, ultra nearly transparent] (0.4*\R,0.3*\R) arc (asin(0.6):-180-asin(0.6):0.5*\R) arc (-180:0:0.4 and 0.06) -- cycle;
\fill[gray, semitransparent] (0,0.3*\R) ellipse (0.4 and 0.05); 
\filldraw[fill=white, draw=black, fill opacity=0] (-0.4*\R,0.3*\R) circle (0.06);
\filldraw[fill=white, draw=black, fill opacity=0] (0,0.5*\R) circle (0.06);
\draw[thick, double] (-0.4*\R,0.3*\R) arc (180-asin(0.6):90:0.5*\R);
\end{tikzpicture}\end{aligned}.
$

Note that there are two choices for $\pmb{s}(a_{1})$ and $\pmb{s}(a_{2})$:
$$\pmb{s}(a_{1})=\xymatrix{(re^{i\theta},0)\ar[r]^{\pm1} & (re^{i\theta},0)},\ \theta\in(0,2\pi)$$
$$\pmb{s}(a_{2})=\xymatrix{(re^{i(\theta')},0)\ar[r]^{\pm1} & (re^{i(\theta')},0)}, \ \theta'\in(-\pi,\pi).$$
But 
$$\pmb{s}(a_{3})=\xymatrix{(re^{i\theta},0)\ar[r]^{1} & (re^{i\theta},0)}\ \ for \ \theta\in (0,\pi)$$
$$\pmb{s}(a_{3})=\xymatrix{(re^{i\theta},0)\ar[r]^{-1} & (re^{i\theta},0)}\ \ for \ \theta\in (\pi,2\pi)$$
Thus it is impossible to choose the values of $\pmb{s}(a_{1})$ and $\pmb{s}(a_{2})$ so that
$$\pmb{s}(a_{3})=\pmb{s}(a_{2})\circ\pmb{s}(a_{1}).$$
\end{proof}

\begin{numrmk}
One can see that the nontrivial orbifold structure of the sectional orbifold morphism comes from the fact that the cocycle condition holds only with respect to injections of different orbifold charts. With respect to embeddings of open sets, the cocycle condition holds only upto a 2-morphism.
\end{numrmk}

\subsection{Hamiltonian Orbifiber Bundles over Sphere}\label{orbifiberbundleS2}

In this section we will construct orbifiber bundles over sphere for any Hamiltonian loop of any symplecitic orbifold generalizing Example \ref{exmCtwo}. 

Let $(\X,\omega)$ be a symplectic orbifold. Identify $S^{2}$ with $\com P^{1}=(\com^{2}- \{0\})/\com^{*}$. For a small positive number $\delta$, denote $D_{-}(1+\delta):=\{[1,z]\in \com P^{1}||z|<1+\delta\}$, $D_{+}(1+\delta):=\{[z,1]\in \com P^{1}||z|<1+\delta\}$. The intersection of $D_{-}(1+\delta)\cap D_{+}(1+\delta)$ is an annulus $Ann=\{[z,1]\in \com P^{1}|1/(1+\delta)<|z|<1+\delta\}$. Denote by $i_{+}$ and $i_{-}$ the inclusions of $Ann$ into $D_{+}(1+\delta)$ and $D_{-}(1+\delta)$ respectively. Denote by $\Theta:Ann\to S^{1}$ the map sending $[re^{i\theta},1]$ to $e^{i\theta}$.
By definition, a smooth Hamiltonian loop $\pmb{\gamma}$ comes with a stack map $\tilde{\gamma}:S^{1}\times \X\to\X$. This determines an open embedding of stacks:
$$
em_{+}:=(i_{+}\circ pr_{Ann})\times\tilde{\gamma}\circ(\Theta,Id_{\X}):Ann\times \X\to D_{+}(1+\delta)\times \X.
$$
Together with the obvious embedding 
$$
em_{-}:=(i_{-}, Id_{\X}):Ann\times \X\to D_{-}(1+\delta)\times \X, 
$$
we glue the two stacks $D_{+}(1+\delta)\times \X$ and $D_{-}(1+\delta)\times \X$, and denote the resulting stack by $\E_{\gamma}$. There is an obvious projection $\pmb{\pi}$ from $\E_{\gamma}$ to $S^{2}$ determined by projection to the first factor on each piece. Thus we get an orbifiber bundle $\pmb{\pi}:\E_{\gamma}\to S^{2}$.

\begin{prop}\label{homotopylem}
If two smooth loops of orbifold diffeomorphisms $\pmb{\gamma}_{0}$ and $\pmb{\gamma}_{1}$ are homotopy equivalent, then there is an orbifiber bundle isomorphism between $\E_{\pmb{\gamma}_{0}}$ and $\E_{\pmb{\gamma}_{1}}$.
\end{prop}
\begin{proof}
Let $\pmb{\gamma}_{0}^{-1}:I\times \X\to\X$ be the composition of the inverse of $pr_{I}\times \pmb{\gamma}_{0}:I\times \X\to I\times \X$ with the projection $I\times \X\to\X$. Then $\pmb{\gamma}_{0}^{-1}$ represents the inverse of the homotopy class of $\pmb{\gamma}_{0}$. Since $\pmb{\gamma}_{0}$ and $\pmb{\gamma}_{1}$ are homotopy equivalent, $\pmb{\gamma}_{1}\cdot_{cp} \pmb{\gamma}_{0}^{-1}$  is homotopy equivalent to the constant loop $\pmb{\gamma}_{Id}:=\pi_{\X}:I\times \X\to \X$ where $\pi_{\X}$ is the projection to $\X$. By Lemma \ref{smthhomotopylem}, there is smooth relative homotopy between $\pmb{\gamma}_{0}^{-1}\cdot_{cp} \pmb{\gamma}_{0}$ and the constant identity loop $\pmb{\gamma}_{Id}$. 

Recall that the smooth loops $\pmb{\gamma}_{0}$ and $\pmb{\gamma}_{1}$ comes with stack maps $\tilde{\gamma}_{0}:S^{1}\times \X\to \X$ and $\tilde{\gamma}_{1}:S^{1}\times \X\to \X$ respectively. Let $\tilde{\gamma}_{0}^{-1}:S^{1}\times \X\to\X$ be the composition of the inverse of $pr_{S^{1}}\times \tilde{\gamma}_{0}:S^{1}\times \X\to S^{1}\times \X$ with the projection $S^{1}\times \X\to\X$. Then $\pmb{\gamma}_{1}\cdot_{cp} \pmb{\gamma}_{0}^{-1}$ factors through the exponential map $(exp, Id_{\X}):I\times \X\to S^{1}\times\X$ and  $\tilde{\gamma}_{01}:=\tilde{\gamma}_{1}\circ (pr_{S^{1}}\times \tilde{\gamma}_{0}^{-1}): S^{1}\times \X \to \X$.

Therefore there exists $\tilde{\Gamma}: [0,1]\times S^{1} \times \X\to \X$ such that there are 2-morphisms $\tilde{\Gamma}|_{\{0\}\times S^{1}\times \X}\Rightarrow \tilde{\gamma}_{01}$ and $\tilde{\Gamma}|_{\{1\}\times S^{1}\times \X}\Rightarrow \tilde{\gamma}_{Id}$. Here $\tilde{\gamma}_{Id}$ is the projection from $S^{1}\times \X$ to $\X$ which corresponds to the constant identity loop $\pmb{\gamma}_{Id}$.

Denote 
$$\check{A}:=\{[re^{i\theta},1]|\delta< r < 1+\delta \},\ \ \check{A}_{out}:=\{[re^{i\theta},1]|1\le r < 1+\delta \},$$
$$\check{A}_{in}:=\{[re^{i\theta},1]|\delta< r < 2\delta \},\ \ \check{A}_{mid}:=\{[re^{i\theta},1]|2\delta\le r \le 1 \},$$
 and $$D_{+}(2\delta):=\{[re^{i\theta},1]|r <2 \delta \}.$$ 
 It is easy to construct a differentiable map $\zeta:\check{A}\to [0,1]\times S^{1}$  satisfying the following properties:
\begin{enumerate}
\item $\zeta|_{\check{A}_{mid}}$ is a diffeomorphism from $\check{A}_{mid}$ to $[0,1]\times S^{1}$;
\item $\zeta|_{\check{A}_{out}}$ maps $\check{A}_{out}$ to $\{1\}\times S^{1}$;
\item $\zeta|_{\check{A}_{in}}$ maps $\check{A}_{in}$ to $\{0\}\times S^{1}$.
\end{enumerate}

Denote $\check{\Gamma}:=\tilde{\Gamma}\circ (\zeta\times pr_{\X}): \check{A}\times \X\to \X$. Then $pr_{\check{A}}\times \check{\Gamma}:\check{A}\times\X\to \check{A}\times \X$ and $(Id_{D_{+}(2\delta)}, Id_{\X}): D_{+}(2\delta)\times \X\to D_{+}(2\delta)\times \X$ differ by a 2-morphism when restricted to $\check{A}\cap D_{+}(2\delta) \times \X$. So we can glue them to define an orbifold diffeomorphism 
$$\Psi_{+}: D_{+}(1+\delta)\times \X\to  D_{+}(1+\delta)\times \X.$$
Let
$$\Psi_{-}:=(Id_{D_{-}(1+\delta)}, Id_{\X}) : D_{-}(1+\delta)\times \X\to  D_{-}(1+\delta)\times \X.$$

Then the two maps
$$
\xymatrix{
Ann\times \X\ar[r]^{em_{+}} & D_{+}(1+\delta)\times \X \ar[r]^{\Psi_{+}} & D_{+}(1+\delta)\times \X \ar[r] & \E_{\pmb{\gamma}_{1}}
},
$$
$$
\xymatrix{
Ann\times \X\ar[r]^{em_{-}} & D_{-}(1+\delta)\times \X \ar[r]^{\Psi_{-}} & D_{-}(1+\delta)\times \X \ar[r] & \E_{\pmb{\gamma}_{1}}
}
$$
differ by a 2-morphism. Thus $\Psi_{+}$ and $\Psi_{-}$ can be glued into a diffeomorphism $\Psi:\E_{\pmb{\gamma}_{0}}\to\E_{\pmb{\gamma}_{1}}$. From the definition, $\Psi$ commutes with the projections to $S^{2}$, thus it defines an orbifiber bundle isomorphism between $\E_{\pmb{\gamma}_{0}}$ and $\E_{\pmb{\gamma}_{1}}$.
\end{proof}

Since every Hamiltonian loop is homotopy equivalent to a smooth Hamiltonian loop, and homotopy equivalent smooth Hamiltonian loops define isomorphic orbifiber bundles, thus we have constructed a well-defined orbifiber bundle for every Hamiltonian loop.

We remark that the above construction works without assuming the loop satisfying Hamiltonian equation. On the other hand, for Hamiltonian loop, the orbifiber bundle is a Hamiltonian orbifiber bundle which is an analogue of Hamiltonian fibration in the manifold case. The definition is given below:

\begin{defn}\label{Hamorbifiberbundle} 
An orbifiber bundle $\pmb{\pi}:\E\to \B$ with a symplectic form $\Omega$ on $\E$ and a symplectic form on $\B$ is called a {\em Hamiltonian orbifiber bundle} if $\pmb{\pi}^{*}\omega=\Omega$.
\end{defn}
Similar to the manifold case, when $\pmb{\gamma}$ is Hamiltonian, the orbifiber bundle $\E_{\pmb{\gamma}}$ will be a Hamiltonian orbifiber bundle.
The symplectic form on $\E_{\pmb{\gamma}}$ can be defined in a similar fashion. For completeness, we shall explain below.

For  a given open cover of $S^{1}$, represent $S^{2}$ with the groupoid $U_{S^{2}}$ determined by the following atlas: 
\begin{itemize}
\item two small open disks centered at $0$ and $\infty$ with radius $\delta$,
\item for each open set $(e^{ia},e^{ib})$ in the cover of $S^{1}$, the atlas have an open sets $\{z=re^{i\theta}|0<r<\infty, \theta\in(a,b) \}$. 
\end{itemize}

Similar to the manifold case as in \cite{MS}, in polar coordinate $z=e^{s+i t}$, the closed 2-forms $\omega^{+}$ on $D_{+}(1+\delta)\times \X$ and $\omega^{-}$ on $D_{-}(1+\delta)\times \X$, which are given by
$$\omega^{\pm}=\omega-d'F^{\pm}\wedge ds-d'G^{\pm}\wedge dt+(\partial_{t}F^{\pm}-\partial_{s}G^{\pm})ds\wedge dt,$$
glue to a closed 2-form on $\E_{\pmb{\gamma}}$, when $F^{\pm}$ and $G^{\pm}:Ann\times \X\to \real$ satisfy
\begin{equation}\label{matchingcondi}
\begin{aligned}
F_{s,t}^{+}\circ\pmb{\gamma}^{t}+F_{-s,-t}^{-}  =0,\\
G_{s,t}^{+}\circ \pmb{\gamma}^{t}+G_{-s,-t}^{-}  =H_{t}\circ \pmb{\gamma} ^{t},
\end{aligned}
\end{equation}
where $H_{t}$ is the normalized Hamiltonian function of the loop $\pmb{\gamma}$, $\pmb{\gamma} ^{t}:=\pmb{\gamma}(t,\cdot)$.

Denote by $\int^{fiber}:\Omega^{2n+2}(\E_{\pmb{\gamma}})\to \Omega^{2}(S^{2})$ the integration along fiber.
\begin{defn}\label{couplingcl}
When $\Omega$ satisfies $\int^{fiber}\Omega^{n+1}=0$, we denote it as $\Omega_{0}$, and call
$$u_{\gamma} :=  [\Omega_0]$$ the {\em coupling class} of the Hamiltonian orbifiber bundle over sphere. 
\end{defn}
Note that $u_{\gamma}$ does not depends on the choice of $F^{\pm}$ and $G^{\pm}$ as long as they satisfy the required conditions.

The condition $\int^{fiber}\Omega^{n+1}=0$ is satisfied if and only if $F^{\pm}$ and $G^{\pm}$ have mean value zero over $\{e^{s+it}\}\times \X$.

\begin{rmk}\label{couplingformcutoff}
In this paper, we choose $F_{s,t}^{\pm}\equiv 0$ and $G_{s,t}^{-}\equiv 0$. In order to satisfy (\ref{matchingcondi}), let $\rho :D_{+}(1+\delta)\to [0,1]$ be a smooth function such that  $\rho(z)=0$ when $|z|<1/(1+2\delta)$ and  $\rho(z)=1$ when $|z|>1/(1+\delta)$. Define $G_{s,t}^{+}=\rho(e^{s+it})H_{t}\circ \pmb{\gamma} ^{t}$. Then (\ref{matchingcondi}) is satisfied.
\end{rmk}

For sufficiently large $c>0$, 
\begin{equation}\label{fiberbundlesymplecticform}
\Omega_{c}:=\Omega_{0}+c\pmb{\pi}^{*} \omega_{S^{2}}
\end{equation}
 defines a symplectic form on $\E_{\gamma}$.

\section{Seidel Representation}\label{part2}
In this section we construct Seidel representation for effective symplectic orbifolds.
 
\subsection{Review of Kuranishi Structure}\label{secReviewoKura}
In this subsection, we review the notion of Kuranishi structures following closely \cite[Appendix]{FOOOtoric1} and \cite[Sections 5--6]{FO}. 
Let $\mathcal{M}$ be a compact space. 
\begin{defn} A {\em Kuranishi chart} of $\M$ is a quintuple $(V,E,\Gamma,\psi,s)$ satisfying the following conditions:
\begin{enumerate}
\item The finite group $\Gamma$ acts on the smooth manifold (possibly with boundaries or corners) $V$ effectively, i.e., $[V/\Gamma]$ is an effective orbifold.
\item $pr: E \to V$ is a $\Gamma$-equivariant vector bundle, i.e., it defines an orbibundle over $[V/\Gamma]$.
\item The map $s:V\to E$ is a $\Gamma$-equivariant section of the equivariant bundle $pr : E \to V$ and will be called the Kuranishi map of the Kuranishi chart.
\item The map $\psi : s^{-1}(0)/\Gamma \to \M$ is a homeomorphism from the quotient space of the zero set by $\Gamma$ to its image.
\end{enumerate}
\end{defn}

In practice, Kuranishi charts are constructed from lower strata with larger isotropy group to higher strata with smaller isotopy group. This determines a partial order on the collection of Kuranishi charts. In particular, for two Kuranishi charts $(V_{\alpha_{1}},E_{\alpha_{1}},\Gamma_{\alpha_{1}}, \psi_{\alpha_{1}},s_{\alpha_{1}})$ and $(V_{\alpha_{2}},E_{\alpha_{2}},\Gamma_{\alpha_{2}}, \psi_{\alpha_{2}},s_{\alpha_{2}})$ on $\M$, if $\alpha_{1} \prec\alpha_{2}$ then there exists:
\begin{enumerate}
\item a $\Gamma_{\alpha_{1}}$-invariant open subset $V_{\alpha_{2},\alpha_{1}} \subset V_{\alpha_{1}}$,
\item a smooth embedding $\varphi_{\alpha_{2},\alpha_{1}} : V_{\alpha_{2},\alpha_{1}} \to V_{\alpha_{2}}$,
\item a bundle map $\widehat{\varphi}_{\alpha_{2},\alpha_{1}} : E_{\alpha_{1}}\vert_{V_{\alpha_{2},\alpha_{1}}} \to E_{\alpha_{2}}$,  which covers $\varphi_{\alpha_{2},\alpha_{1}}$,
\item  an injective homomorphism $\widehat{\widehat{\varphi}}_{\alpha_{2},\alpha_{1}}: \Gamma_{\alpha_{1}} \to \Gamma_{\alpha_{2}}$.
\end{enumerate}

\begin{defn}(See \cite[Definition 6.1]{FO} or \cite[Lemma A1.11]{FOOObook})
A collection of Kuranishi charts on $\M$, $\{(V_{\alpha},E_{\alpha},\Gamma_{\alpha}, \psi_{\alpha},s_{\alpha})|\alpha\in\frak{A}\}$, is called a {\em Kuranishi structure on $\M$ with a good coordinate system}  if the following conditions hold:
\begin{enumerate}
\item The space $\M$ is covered by the union of $\psi_{\alpha}(s_{\alpha}^{-1}(0)/\Gamma_{\alpha})$ for all $\alpha$.
\item $dim V_{\alpha}- rank E_{\alpha}=n$ is independent of $\alpha$ and will be called the {\em dimension}\footnote{Sometimes this is called the {\em virtual dimension} of $\M$.} of the Kuranishi structure.
\item The partial order ``$\prec$'' on $\mathfrak{A}$ satisfies: $\forall\alpha_{1}, \alpha_{2} \in \mathfrak{A}$, if
$\psi_{\alpha_{1}}(s_{\alpha_{1}}^{-1}(0)/\Gamma_{\alpha_{1}}) \cap\psi_{\alpha_{2}}(s_{\alpha_{2}}^{-1}(0)/\Gamma_{\alpha_{2}}) \ne \emptyset,$
then we have  either $\alpha_{1} \prec \alpha_{2}$ or $\alpha_{2} \prec \alpha_{1}$. 
\item The maps $\varphi_{\alpha_2,\alpha_1}$,
${\widehat{\varphi}}_{\alpha_2,\alpha_1}$ are $(\Gamma_{\alpha_{1}} , \Gamma_{\alpha_{2}})$-equivariant with respect to the group homomorphism $\widehat{\widehat{\varphi}}_{\alpha_2,\alpha_1}$.
\item
The map $\varphi_{\alpha_2,\alpha_1}$ and the group homomorphism $\widehat{\widehat{\varphi}}_{\alpha_2,\alpha_1}$ induce an embedding of orbifold
$\overline{\varphi}_{\alpha_{2},\alpha_{1}} : [V_{\alpha_{2},\alpha_{1}}/\Gamma_{\alpha_{1}}]\to [V_{\alpha_{2}}/\Gamma_{\alpha_{2}}]. $
\item The sections on different Kuranishi charts are compatible: $s_{\alpha_{2}} \circ \varphi_{\alpha_{2},\alpha_{1}}= {\widehat{\varphi}}_{\alpha_{2},\alpha_{1}} \circ s_{\alpha_{1}}.$
\item On $ V_{\alpha_{2},\alpha_{1}} \cap s_{\alpha_{1}}^{-1}(0)/\Gamma_{\alpha_{1}}$, we have $\psi_{\alpha_{2}}\circ \overline{\varphi}_{\alpha_{2},\alpha_{1}}= \psi_{\alpha_{1}}$.
\item If $\alpha_{1} \prec \alpha_{2} \prec \alpha_{3}$ then
$\varphi_{\alpha_3,\alpha_2}\circ \varphi_{\alpha_2,\alpha_1}= \varphi_{\alpha_3,\alpha_1},$
on $\varphi_{\alpha_2,\alpha_1}^{-1}(V_{\alpha_3,\alpha_2})$. Similarly, 
$\widehat{\varphi}_{\alpha_{3},\alpha_{2}}\circ \widehat{\varphi}_{\alpha_{2},\alpha_{1}}
= \widehat{\varphi}_{\alpha_{3},\alpha_{1}}$ and $
\widehat{\widehat{\varphi}}_{\alpha_{3},\alpha_{2}}\circ \widehat{\widehat{\varphi}}_{\alpha_{2},\alpha_{1}}
= \widehat{\widehat{\varphi}}_{\alpha_{3},\alpha_{1}}$.
\item
The space $V_{\alpha_{2},\alpha_{1}}/\Gamma_{\alpha_{1}}$ contains $\psi_{\alpha_1}^{-1}(\psi_{\alpha_1}(s_{\alpha_1}^{-1}(0)/\Gamma_{\alpha_1}) \cap
\psi_{\alpha_2}(s_{\alpha_2}^{-1}(0)/\Gamma_{\alpha_2}))$.
\end{enumerate}
\end{defn}
In addition, we need the Kuranishi structure to have a tangent bundle, which means the following:
\begin{defn}
We say a Kuranishi structure on $\M$ {\em has a tangent bundle} if the differential of
$s_{\alpha_2}$ in the normal direction induces a
bundle isomorphism
\begin{equation}
ds_{\alpha_2} : \frac{{\varphi}_{\alpha_2,\alpha_1}^*TV_{\alpha_2}}{TV_{\alpha_2,\alpha_1}}
\to \frac{\widehat{\varphi}_{\alpha_2,\alpha_1}^*E_{\alpha_2}}{E_{\alpha_1}}.
\nonumber\end{equation}
A Kuranishi structure on $\M$ is called {\em oriented} if $V_{\alpha}$ and $E_{\alpha}$ are oriented,
the $\Gamma_{\alpha}$ action is orientation preserving, and $ds_{\alpha}$ preserves the orientation of bundles.
\end{defn}

\begin{defn}
Let $\M$ be a compact space with oriented Kuranishi structure of dimension $n$, and $Y$ a topological space. A strongly continuous map $f:\M\to Y$ is a system of germs of maps $f_p: U_p\to Y$ for each $p$ such that $f_p\circ\varphi_{pq}=f_q$. If $Y$ is an orbifold, $f$ is said to be strongly smooth if each $f_p$ is smooth. 
\end{defn}

We would like to construct a homology class in $H_n(Y,\mathbb{Q})$. Naively we can take the sum over $p$ of $s_p^{-1}(0)$ pushed-forward by $f_p$. The difficulty is that sections $s_p$ may not be transversal to the zero sections of $E_p$. However, if the Kuranishi structure has a good coordinate system, $s_p$ can be approximated by multisections $s_{p,n}$ which are transversal to the zero section. More precisely, we have the following result.

\begin{thm}[\cite{FO}, Theorem 6.4]
Let $(P,((U_p,\psi_p,s_p):p\in P),\varphi_{pq},\hat{\varphi}_{pq})$ be a good coordinate system of a compact space $M$ with a Kuranishi structure. Suppose the Kuranishi structure has a tangent bundle given by $\Phi_{pq}:N_{U_p}U_q\simeq E_p/E_q$. Then for each $p\in P$ there exists a sequence of smooth multisections $s_{p,n}$ such that
\begin{itemize}
\item $s_{p,n}\circ\varphi_{pq}=\hat{\varphi}_{pq}\circ s_{q,n}$;
\item $lim_{n\to\infty}s_{p,n}=s_p$ in the $C^{\infty}$-topology;
\item $s_{p,n}$ is transversal to $0$;
\item at any point in $U_{pq}$, the restriction of the differential of the composition of any branch of $s_{p,n}$ and the projection $E_p\to E_p/E_q$ coincides with the isomorphism $\Phi_{pq}:N_{U_p}U_q\simeq E_p/E_q$.
\end{itemize}
\end{thm}

With suitably defined multiplicities $m_p$, the sum $\sum_{p\in P} m_p\cdot {f_p}_*[s_{p,n}^{-1}(0)]$ forms a chain in $Y$. This chain is closed if the Kuranishi structure is oriented \cite[Lemma 6.11]{FO}. Moreover this chain is independent of the choices of multisections \cite[Theorem 6.12]{FO}. This is the fundamental class of this Kuranishi structure,  denoted by $f_*[M]\in H_n(Y,\mathbb{Q})$.

\subsection{Orbifold Gromov-Witten Invariants}\label{secReviewoGW}
In this subsection, we review orbifold Gromov-Witten invariants constructed in \cite{CR2}. Recall that an orbi-curve  is a complex orbifold of complex dimension $1$ with finitely many orbifold points whose stabilizers groups are cyclic. 
\begin{defn}\label{defNodalOrbicurve}
A {\em nodal orbi-curve with marked points} is a tuple $(\Sigma,j,\vec v,D)$, where $\Sigma=\sqcup \Sigma_{\nu}$ is an orbi-curve with $\Sigma_{\nu}$ its connected components, $j$ is the complex structure on $\Sigma$, $\vec v$ is a finite ordered collection of (orbi)points in $\Sigma$, and $D$ is a finite collection of un-ordered pairs $(w,w')$ of (orbi)points in $\Sigma$ such that $w\neq w'$, $w$ and $w'$ have the same orbifold structure, and two pairs which intersect are identical. The union of all subsets $\{w,w'\}$ of $\Sigma$, denoted by $|D|$, is disjoint from $\vec v$. 

The points in $\vec{v}$ are called marked points, and $D$ is called the set of nodal pairs. The points in $|D|$ are called nodal points (nodes). Denote $\vec{v}_{\nu}=\Sigma_{\nu}\cap\vec{v}$ and $|D|_{\nu}=\Sigma_{\nu}\cap |D|$. The points in $\vec{v}_{\nu}\cup |D|_{\nu}$ are called special points of the component $\Sigma_{\nu}$. Following convention in algebraic geometry, we call a connected component of $\Sigma$ an irreducible component.
\end{defn}
Let $(\X,\omega,J)$ be a compact symplectic orbifold with symplectic form $\omega$ and compatible almost complex structure $J$. 
\begin{defn}\label{defStableOrbiMorph}
A {\em stable orbifold morphism} $\pmb{f}$ from a nodal orbi-curve $(\Sigma, j, \vec z, D)$ to $\X$ is a collection of orbifold morphisms $\{(\pmb{f}_{\nu}:\Sigma_{\nu}\to\X)\}$ such that:
\begin{itemize}
\renewcommand{\labelitemi}{--}
\item orbipoints of $\Sigma$ are contained in $\vec z$ or $|D|$,
\item $\pmb{f}_{\nu}$ is $j$-$J$-holomorphic,
\item representability: $\pmb{f}_{\nu}$ induces injective group homorphism at the orbipoints,
\item balance condition at the nodes: $ev(\pmb{f},w)=\mathcal{I}(ev(\pmb{f},w'))$ for $(w,w')\in D$, where $ev(\pmb{f},w)$, $ev(\pmb{f},w')$ $\in I\X$ are the evaluation\footnote{See \cite{CR2} for definition of the evaluation map.} of the orbifold morphism $\pmb{f}$ at $w$ and $w'$, $\mathcal{I}:I\X\to I\X$ is the involution which sends $(x,(g))\in \X_{(g)}$ to $(x,(g^{-1}))\in \X_{(g^{-1})}$.
\end{itemize}
\end{defn}

Let $\mathbf{x}=(\X_{(g_1)},...,\X_{(g_k)})$ be an $n$-tuple of twisted sectors, and let $\sigma\in H_{2}(|\X|,\inte)$. Consider the moduli space $\overline{\M}_{g,k}(\X,J,\sigma, \mathbf{x})$ of orbifold stable maps $\pmb{f}$ with the following conditions, modding out the automorphisms (biholomorphic diffeomorphisms of $(\Sigma, j, \vec{v}, D)$ that preserve $\pmb{f}$):
\begin{enumerate}
\item domain of $\pmb{f}$ has genus $g$ and $k$ marked points,
\item $|\pmb{f}|$ represents the homology class $\sigma\in H_{2}(|\X|,\inte)$,
\item $\pmb{f}$ satisfies the given boundary condition $\mathbf{x}$, i.e., the image of the evaluation map at the $i$-th marked point $ev_i$ is contained in $\X_{(g_i)}$.
\end{enumerate}
The main task of defining Gromov-Witten invariants is to construct a Kuranishi structure over $\overline{\M}_{g,k}(\X,J,\sigma,\mathbf{x})$ and define a virtual fundamental cycle using it. We now explain the construction of the Kuranishi chart of a stable orbifold map which is carried out in \cite{CR2}. The exposition here is taken from the appendix of \cite{FOOOtoric1} with modification for the orbifold context. 

Consider $ \tau=[\pmb{f}:(\Sigma,j,\vec z, D)\to\X]=[\{(\pmb{f}_{\nu}:\Sigma_{\nu}\to\X)\}]\in \Mbar_{g,k}(\X,J,\sigma, \mathbf{x})$. We first recall the definition of the {\em isotropy group} of $\tau$, as follows. 

First suppose $(\Sigma,j,\vec z, D)$ is an orbi-curve with no nodal point. If $|\pmb{f}|(|\Sigma|)\subset \Sigma\X$ where $\Sigma\X$ denote the set of points in $|\X|$ which are covered by orbipoints in $\X$, there is a set $\vec{z}'$ of (orbi)points on $\Sigma$ with $\vec z\cup D \subset \vec z'$ such that for all $p\in f(\Sigma\setminus\vec z')$ the local group $G_p$ at $p$ is isomorphic to a fixed group $G$.
There is a principal $G$-bundle over $\Sigma\setminus\vec z'$ induced from $\pmb{f}$. Sections of this bundle which extend to $\Sigma$ form a group $G_{\tau}$ which is called the isotropy group of $\tau$ (see \cite{CR2}). If $|\pmb{f}|(|\Sigma|)$ is not contained in $\Sigma\X$, then $|\pmb{f}|(|\Sigma|)$ meets $\Sigma\X$ only at finitely many points $\vec{z}''$. For a point $p\in\pmb{f}(\Sigma\setminus \vec{z}'')$ the group $G_p$ is trivial. Thus in this case we define $G_{\tau}$ to be the trivial group.

In general, we have the following
\begin{defn}
For $\tau$ with possibly reducible $\Sigma$, define the isotropy group of $\tau$ to be
$$G_{\tau}:=\{(g_{\nu})\in\prod_{\nu} G_{\nu}\,|\,g_{\nu}(w)=g_{\omega}(w') \quad {\rm if}\quad (w,w')\in D\},$$
where $G_{\nu}$ is the isotropy group of $\pmb{f}_{\nu}:(\Sigma_{\nu},j,\vec z_{\nu}, D_{\nu})\to \X$ constructed as in the irreducible case.
\end{defn}

Let $(T\X)^{\times}_{f_\nu(z)}$ be the linear subspace of (local group) invariant vectors in $T\X |_{f_{\nu}(z)}$. We define several Sobolev spaces needed in the construction.

\begin{defn}
For $p>2$, define $L^p_{1,\delta}(\pmb{f}_{\nu}^*T\X)$
to be the space of local $L_1^p$ sections of $\pmb{f}_{\nu}^*T\X$ which decay exponentially with weight $\delta$ to elements in $(T\X)^{\times}_{f_\nu(z)}$ for all $z\in  |D|_{\nu}$.
Define $L^p_{\delta}(\pmb{f}_{\nu}^*T\X \otimes \Lambda^{0,1})$ to be the subspace of local $L^p$ sections of $\pmb{f}_{\nu}^*T\X \otimes \Lambda^{0,1}$ which decay exponentially with weight $\delta$ at nodes\footnote{See \cite[(3.2.1)]{CR2} for more precise definition.}.
Define the spaces $$L^p_{1,\delta}(\pmb{f}^*T\X):=\{(u_\nu)\in \oplus_\nu L^p_{1,\delta}(\pmb{f}_{\nu}^*T\X)| u_{\nu}(w)= u_{\omega}(w'),\hspace{2mm} \mbox{if}\hspace{2mm}(w,w')\in D\}, $$ 
and $$L^p_{\delta}(\pmb{f}^*T\X\otimes \Lambda^{0,1})= \{(u_\nu)\in\oplus_\nu L^p_{\delta}(\pmb{f}_{\nu}^*T\X \otimes\Lambda^{0,1})\}.$$
\end{defn}

The group $G_{\tau}$ acts linearly on $L^p_{1,\delta}(\pmb{f}^*T\X)$ and  $L^p_{\delta}(\pmb{f}^*T\X\otimes \Lambda^{0,1})$ since each $G_{\nu}$ acts linearly on $L^p_{1,\delta}(\pmb{f}_{\nu}^*T\X)$ and  $L^p_{\delta}(\pmb{f}_{\nu}^*T\X\otimes \Lambda^{0,1})$. Moreover, the automorphism group $Aut(\tau)$ of $\tau$ acts linearly on $L^p_{1,\delta}(\pmb{f}^*T\X)$ and  $L^p_{\delta}(\pmb{f}^*T\X\otimes \Lambda^{0,1})$ covering the action on the domain orbicurve. Let $g\to \gamma^{*}(g)$ be the automorphism on $G_{\tau}$ induced by pull-back via $\gamma\in Aut(\tau)$, then for any section $u$ in $L^p_{1,\delta}(\pmb{f}^*T\X)$ and  $L^p_{\delta}(\pmb{f}^*T\X\otimes \Lambda^{0,1})$, we have $$(\gamma_{*})^{-1}\circ g\circ \gamma_{*}(u)=\gamma^{*}(g)(u).$$ 
Then define $\Gamma_{\tau} $ to be the group generated by $Aut(\tau)$ and $G_{\tau}$ with the above relation. Consequently there is a short exact sequence 
$$ 1\to G_{\tau}\to\Gamma_{\tau}\to Aut(\tau)\to 1,$$
and $\Gamma_{\tau}$ acts on $L^p_{1,\delta}(\pmb{f}^*T\X)$ and  $L^p_{\delta}(\pmb{f}^*T\X\otimes \Lambda^{0,1})$ linearly.

If every component of $(\Sigma,j,\vec{z}, D)$ is stable, then by forgetting the map of $\tau$, we get an element in the moduli space of stable nodal orbicurves $\Mbar_{g,k,\vec{m}}$, where $\vec{m}=\{m_{i}\}$ records the orbifold structures $\inte_{m_{i}}$ at the $i$-th marked point in $\vec{v}$. We remark that the compactification used here depends on the target orbifold $\X$: the local groups of $\X$ provide a bound of the complexity of the orbifold structure at the nodal points. Otherwise, if one allows all possible orbifold structures at the nodal points, the moduli space will be non-compact. In this case, we can pick a neighborhood $\frak U$ of $[\Sigma,j,\vec z, D]$ inside $\Mbar_{g,k,\vec{m}}$, and there exists a manifold $\frak{V}$ such that $\frak{U}=\frak{V}/Aut(\Sigma)$, where $Aut(\Sigma)$ denotes the automorphism group of $[\Sigma,j,\vec z, D]$.

Otherwise if there exists unstable component $\Sigma_{\nu}$, then pick points $z_{\nu,i}\in \Sigma_{\nu}$ ($i=1,\cdots ,d_{\nu}$) with the following properties:
\begin{conds}\label{addmarked}
\begin{enumerate}
\item $\pmb{f}$ is immersed at $z_{\nu,i}$ for all $\nu$ and $i$;
\item $z_{\nu,i}\neq z_{\nu,j}$ for $i\neq j$ and $z_{\nu,i}\notin \vec z$;
\item $(\Sigma_{\nu};\vec{z}_{\nu}\cup |D|_{\nu}\cup (z_{\nu,1},\cdots,z_{\nu,d_{\nu}}))$ is stable;
\item \label{nodalconsistent}
if $\gamma \in \Gamma_{\tau}$ such that $\gamma(\Sigma_{\nu}) = \Sigma_{\nu'}$, 
then $d_{\nu} = d_{\nu'}$  and
$$\{\gamma(z_{\nu,i})\mid i=1,\cdots,d_{\nu}\} = \{z_{\nu',i'}\mid i'=1,\cdots,d_{\nu}\};$$
\item\label{orbittype} the orbit type of $\pmb{f}(z_{\nu,i})$ is the same as its nearby points.
\end{enumerate}
\end{conds}

By adding extra marked points $z_{\nu,i}$ satisfying the above condition to each unstable branch component of $(\Sigma,j,\vec z, D)$, we obtain a stable nodal orbicurve, which will be denoted by $(\Sigma,j,\vec z^{+}, D)$. Here $\vec{z}^{+}=\vec{z} \sqcup \{z_{\nu,i}|\nu,i\}$.

For each $\nu,i$ we choose a submanifold $N_{\nu,i}$ of codimension $2$ in $\X_{(g)}$ that transversely intersects with $\pmb{f}_{\nu}$ at $ev(\pmb{f}_{\nu},z_{\nu,i})$, where $\X_{(g)}$ is the twisted sector containing the evaluation of $\pmb{f}_{\nu}$ at $z_{\nu,i}$. Corresponding to Condition \ref{addmarked} (\ref{nodalconsistent}) above we choose $N_{\nu,i} = N_{\nu',i'}$ if $\gamma(z_{\nu,i}) = z_{\nu',i'}$. Note that Condition \ref{addmarked} (\ref{orbittype}) indicates that $dim \X_{(g)}\ge 2$.

Let $k'$ be the number of points in $\vec{z}^{+}$. Consider a neighborhood $\frak U_{0}$ of $[\Sigma,j,\vec{z}^+, D]$ in $\Mbar_{g,k',\vec{m}^{+}}$ as in the stable case discussed before, where $\vec{m}^{+}$ records $\vec{m}$ the orbifold structures of the added marked points as well. Let $Aut(\Sigma^{+})$ be the group of automorphisms of $[\pmb{f}:(\Sigma,j,\vec z^+,D)\to \X]$. Then  $Aut(\Sigma^{+})$ is a subgroup of $Aut(\Sigma)$. Then the neighborhood $\frak{U}_0$ is covered by an orbifold chart $\frak{V}_0/Aut(\Sigma^{+})$.

An element $\gamma \in Aut(\Sigma)$ defines an homeomorphism $\gamma_{*}$ of $[\Sigma,j,\vec z^+,D]$ to itself. In particular, $\gamma_{*}$ maps $[(\Sigma,j,\vec z^+,D)]$ to $[\gamma_*(\Sigma,j,\vec z^+,D)]$ which differs from $[\Sigma,j,\vec z^+,D]$ only by reordering of the added marked points by Condition \ref{addmarked} (\ref{nodalconsistent}). The open set $\gamma_{*}\frak{U}_{0}$ is an open neighborhood of $[\gamma_*(\Sigma,j,\vec z^+,D)]$. Let $\frak{U}=\cup_{\gamma\in Aut(\Sigma)} \gamma_{*}\frak{U}_{0}$. Then there exists a manifold $\frak{V}$ on which $Aut(\Sigma)$ acts such that $\frak{V}/Aut(\Sigma^{+}) \simeq  \frak{U}$ and $\frak{V}/Aut(\Sigma) \simeq \frak U_{0}$.

For $\frak{V}$ constructed as above, there is a universal family $\frak{M} \to \frak{V}$ where the fiber $\Sigma(\text{\bf v})$ over $\text{\bf v} \in \frak{V}$ is identified with the marked stable orbicurve that represents the element $[\text{\bf v}] \in \frak{V}/Aut(\Sigma)\subset \Mbar_{g,k+1,\vec{m}}$ (or in $\frak{V}/Aut(\Sigma^{+})\subset \Mbar_{g,k+1,\vec{m}^{+}}$ if $(\Sigma,j,\vec{z}, D)$ is unstable and marked points were added). There is an $Aut(\Sigma)$ action on $\frak{M}$ such that $\frak{M} \to \frak{V}$ is $Aut(\Sigma)$ equivariant.

By the construction, each fiber $\Sigma(\text{\bf v})$ of the  universal family is diffeomorphic to $\Sigma$ away from the special points. More precisely, let $\Sigma_0 = \Sigma \setminus S$ where $S$ is a small neighborhood of the special point set. Following \cite{CR2}, orbipoints are assumed to be either marked or nodal, thus $\Sigma_{0}$ is smooth, i.e. with no orbipoint. Then $\forall\text{\bf v}\in\frak{V}$ there exists a smooth embedding $i_{\text{\bf v}} : \Sigma_{0} \to \Sigma(\text{\bf v})$, which need not be biholomorphic. The defect for $i_{\text{\bf v}}$ to be biholomorphic tends to $0$ as $\text{\bf v}$ goes to $\text{\bf v}_{0}$, where $\text{\bf v}_{0}$ corresponds to the orbicurve $[\Sigma,j,\vec{z}, D]\in\Mbar_{g,k',\vec{m}}$ or $[\Sigma,j,\vec{z}^+, D]\in\Mbar_{g,k',\vec{m}^{+}}$ if additional marked points are needed. Moreover, we may assume the map $\text{\bf v} \mapsto i_{\text{\bf v}}$ is $Aut(\Sigma)$ (or $Aut(\Sigma^{+})$) invariant and $i_{\text{\bf v}}$ depends smoothly on $\text{\bf v}$.

For each $\nu$ we may choose $W_{\nu}$ so that $W_{\nu} \subset \Sigma_0$ and the closure of $W_\nu$ in $\Sigma$ is disjoint from the special points (singular or marked). By the unique continuation theorem, we can choose a finite-dimensional subset $E_{0,\nu}$ of the space $\mathcal{C}_{0}^{\infty}(W_{\nu};\pmb{f}_{\nu}^{*}T\X)$ of smooth sections of $\pmb{f}_{\nu}^{*}T\X$ with compact supports contained in $W_\nu$ such that 
\begin{equation*}
ImD_{\pmb{f},\nu}\bar{\partial}+E_{0,\nu}=L^p_{\delta}(\pmb{f}_{\nu}^*T\X \otimes \Lambda^{0,1}).
\end{equation*}
Moreover we assume that $\bigoplus_{\nu} E_{0,\nu}$ is invariant under the $\Gamma_{\tau}$ action in the following sense: if $\gamma \in \Gamma_{\tau}$ and $\Sigma_{\nu'}= \gamma(\Sigma_{\nu})$ then the isomorphism induced by $\gamma$ maps $E_{0,\nu}$ to $E_{0,\nu'}$.

Now we consider a pair $(\text{\bf v},\pmb{f}')$ where $\text{\bf v}\in \frak{V}$ and $\pmb{f}' : \Sigma(\text{\bf v}) \to \X.$ We assume:
\begin{conds}\label{condnbh}
There exists a scalar $\epsilon > 0$ depending only on $\tau$ with the following properties.
\begin{enumerate}
\item $\sup_{x \in \Sigma_0} \text{\rm dist} (\pmb{f}'(i_{\text{\bf v}}(x)),\pmb{f}(x))\le \epsilon$.
\item Let $D_c$ be a connected component of $\Sigma(\text{\bf v}) \setminus Im(i_{\text{\bf v}})$, then the diameter\footnote{This is measured with a Riemannian metric on $\X$, which we choose throughout the paper.} of $\pmb{f}'(D_c)$ in $\X$ is smaller than $\epsilon$.
\end{enumerate}
\end{conds}
For $x \in W_{\nu}$ we have an identification
$$
T_{\pmb{f}(x)} \X \otimes \Lambda^{0,1}_x(\Sigma) \cong T_{\pmb{f}'(i_{\text{\bf v}}(x))} \X \otimes \Lambda^{0,1}_{i_{\text{\bf v}}(x)}(\Sigma(\text{\bf v})),
$$
given by parallel transport. This identification gives an embedding
$$
I_{(\text{\bf v},\pmb{f}')} : \bigoplus_{\nu} E_{0,\nu}
\longrightarrow \pmb{f}^{\prime *}(T\X) \otimes \Lambda^{0,1}
(\Sigma(\text{\bf v})).
$$
Then we consider the equation
\begin{equation}\label{perturbedeq}
\overline{\partial} \pmb{f}' \equiv 0
\mod \bigoplus_{\nu} I_{(\text{\bf v},\pmb{f}')}(E_{0,\nu}),
\end{equation}
and the additional conditions
\begin{equation}\label{markedadd}
\pmb{f}'(z_{\nu,i}) \in N_{\nu,i},\ \ \ \ \ \ \ \text{for\ all\ added\ marked\ points}\ z_{\nu,i}.
\end{equation}
Let $V_{\tau}$ be the set of solutions of (\ref{perturbedeq}) subject to the condition (\ref{markedadd}). It follows from the implicit function theorem and a glueing argument  that $V_\tau$ is a smooth manifold, see \cite{CR2} for the orbifold case, \cite{FO} and \cite[Section A1.4]{FOOObook} for the manifold case. We can make all the construction above invariant under the $Aut(\Sigma)$ or $Aut(\Sigma^{+})$ action. Note that $Aut(\tau)\subset Aut(\Sigma)$ or $\Aut(\Sigma^{+})$, and the space $V_{\tau}$ has a $Aut(\tau)$ action. Together with the $G_{\tau}$ action on $L^p_{1,\delta}(\pmb{f}^*T\X)$ and  $L^p_{\delta}(\pmb{f}^*T\X\otimes \Lambda^{0,1})$, this defines a $\Gamma_{\tau}$ action on $V_{\tau}$ (See \cite[Proposition 3.2.5]{CR2} for more details).

We define the obstruction bundle $E$ as follows: the fiber of $E$ at $\tau$ is defined to be the space $\bigoplus_{\nu} E_{0,\nu}$, and the fiber of $E$ at $(\text{\bf v},\pmb{f}')$ is defined to be $\bigoplus_{\nu} I_{(\text{\bf v},\pmb{f}')}(E_{0,\nu})$.

\begin{rmk}
An alternative way of constructing Kuranishi chart is given in \cite{CLSZ} using Banach groupoid. Our proofs of the properties of Seidel representation are not sensitive to the way how the Kuranishi charts are constructed and can be easily adapted into their setting.
\end{rmk}

We omit the construction of coordinate changes, which is similar to \cite[Section 15]{FO}. We remark that there are obvious coordinate changes from a Kuranishi chart around a stable map in a lower stratum to the main stratum, because the Kuranishi structure is constructed inductively.

Now we define the Kuranishi map by
\begin{equation}\label{kuramap}
(\pmb{f}':(\Sigma,j,\vec{z},D)\to \X) \mapsto\overline{\partial} \pmb{f}' \in E.
\end{equation}
This completes the review of the construction of Kuranishi charts. 

The evaluation map 
$$ev:=(ev_1,...,ev_n): \overline{\M}_{g,n}(\X,J,\sigma,\mathbf{x})\to \X_{(g_1)}\times...\times\X_{(g_n)},$$ is strongly continuous. The standard machinery of Kuranishi structures developed in \cite{FO} applies to the Kuranishi structure on $\overline{\M}_{g,n}(\X,J,\sigma, \mathbf{x})$, giving the fundamental class 
$$ev_*[\overline{\M}_{g,n}(\X,J,\sigma, \mathbf{x})]\in H_*(\overline{\M}_{g,n}(\X,J,\sigma, \mathbf{x}),\mathbb{Q}),$$
which is shown in \cite{CR2} to depend only on $g$, $n$, $\beta$, $\{(g_i)\}$ and the symplectic structure on $\X$.
For $\alpha_i\in H^*(\X_{(g_i)})\subset H^*(I\X), i=1,..., n$, the {\em orbifold Gromov-Witten invariant} $\<\alpha_1,...,\alpha_n\>_{g,n,\sigma}^\X$ is defined by 
$$\<\alpha_1,...,\alpha_n\>^\X_{g,n,\sigma}:=\int_{ev_*[\overline{\M}_{g,n}(\X,J,\sigma,\mathbf{x})]} pr_1^*\alpha_1\cup...\cup pr_n^*\alpha_n,$$
where $pr_i: \X_{(g_1)}\times...\times\X_{(g_n)}\to \X_{(g_i)}$ is the $i$-th projection.

One of the important properties of $\<\alpha_1,...,\alpha_n\>_{g,n,\sigma}^\X$ is that it is an invariant of the symplectic structure, and does not depend on the choice of almost complex structures.

To organize these invariants in a more informative way, we use the universal Novikov ring:

\begin{defn}
Define a ring $\Lambda^{univ}$ as $$\Lambda^{univ}=\left\{\sum_{k\in\mathbb{R}}r_kt^k\,\,\, \vert \,\,\, r_k\in\mathbb{Q}, \#\{k<c|r_k\neq 0\}<\infty \,\,\forall c\in\mathbb{R}\right\},$$ and equip it with a grading given by $deg(t)=0$. Define $\Lambda:=\Lambda^{univ}[q,q^{-1}]$  with the grading given  by $deg(q)=2$. 
\end{defn}
Genus zero $3$-point orbifold Gromov-Witten invariants can be used to define an associative multiplication $*$ on the cohomology group $H^{*}(I\X,\mathbb{Q})\otimes \Lambda$: for $\alpha_{1}, \alpha_2, \alpha_3\in H^{*}(I\X,\mathbb{Q})$, define
$$\<\alpha_{1}*\alpha_{2},\alpha_{3}\>_{orb}:=\<\alpha_{1},\alpha_{2},\alpha_{3} \>:= \sum_{A\in H_{2}(|\X|,\inte)}\<\alpha_{1},\alpha_{2},\alpha_{3} \>^\X_{0,3,A}q^{c_{1}[A]} t^{\omega[A]}.$$
Here $\<-,-\>_{orb}$ is the {\em orbifold Poincar\'e pairing} (see \cite{CR1}). The product $*$ is called the {\em (small) quantum cup product}. The ring $(H^{*}(I\X,\mathbb{Q})\otimes \Lambda, *)$ is called the {\em (small) orbifold quantum cohomology ring} of $(\X,\omega)$, and is often denoted by $QH_{orb}^{*}(\X,\Lambda)$. 

Let $$QH^*_{orb}(\X,\Lambda)^\times\subset QH^*_{orb}(\X,\Lambda)$$ be the set of elements invertible with respect to the quantum product $*$. The group $(QH^*_{orb}(\X,\Lambda)^\times, *)$ plays an important role in Seidel representations.

\subsection{$J$-holomorphic Curves in Hamiltonian Orbifiber Bundle}\label{secCurvHamBundle}
Given a Hamiltonian orbifiber bundle $$(\pmb{\pi}:\E\to S^2, \Omega)$$ as in Definition \ref{orbifiberbundle}, we consider the orbifold Gromov-Witten invariants of the total orbifold $\E$. Since orbifold Gromov-Witten invariants do not depend on choice of almost complex structures, we can use the following particular class of almost complex structures on $\E$:
\begin{defn}
Let $j$ be a complex structure on $S^{2}$, an almost complex structure $J$ on $\E$ is called $j$-compatible with the Hamiltonian orbifiber bundle if 
\begin{enumerate}
\item $\Omega(\cdot, J\cdot)|_{T^{vert}_{x}\E}$ is symmetric and positive definite (here $T^{vert}\E:=ker(d\pmb{\pi}: T\E\to TS^2)$ is the vertical tangent bundle of $\pmb{\pi}:\E\to S^2$);
\item the projection $\E\to S^2$ is $j$-$J$-holomorphic.
\end{enumerate}
\end{defn} 

To prove the existence of such almost complex structures, one can follow the proof for the manifold case word for word. We refer the readers to \cite{Se}.

From now on we fix the complex structure $j_{0}$ on $S^{2}$ which identifies it with $\com P^{1}$ and choose an almost complex structure $J$ which is $j_{0}$-compatible with $(\E,\Omega)$. Then we consider $J$-holomorphic orbifold morphisms to $\E$ which represent a section class $\sigma\in H^{sec}_{2}(|\E|,\inte)\subset H_{2}(|\E|,\inte)$, where $H^{sec}_{2}(|\E|,\inte)$ consists of homological classes which can be represented by a section of $|\pmb{\pi}|:|\E|\to S^{2}$. Such kind of orbifold morphisms are sectional orbifold morphisms in the sense of Definition \ref{orbisection} up to reparametrizations of the domains.  

More precisely, let $\M_{0,k}(\E,J,\sigma, \mathbf{x})$ be the moduli space of $J$-holomorphic orbifold morphisms that represent $\sigma$, and $\M^{sec}_{0,k}(\E,J,\sigma, \mathbf{x})$ be the moduli space of $J$-holomorphic sectional orbifold morphisms representing $\sigma$. Consider the following equivalence relation: $(\pmb{u}_1:(S^2_{orb},\bz)\to\E)\sim(\pmb{u}_2:(S^2_{orb},\bz')\to \E)$ if and only if there is a biholomorphism $\pmb{\phi}:(S^2_{orb},\bz)\to(S^2_{orb},\bz') $ s.t. $\pmb{u}_2=\pmb{u}_1\circ \pmb{\phi}$. Here $(S^{2}_{orb},\bz)$  and $(S^2_{orb},\bz')$ are orbispheres with orbipoints at $\bz=\{z_{0},...,z_{k-1}\}$ and $\bz'=\{z'_{0},...,z'_{k-1}\}$ respectively.

\begin{lem}
There is a $1$-$1$ correspondence between $\M_{0,k}(\E,J,\sigma, \mathbf{x})/\sim$ and 
$\M^{sec}_{0,k}(\E,J,\sigma, \mathbf{x})$.
\end{lem}
\begin{proof}
For any $J$-holomorphic orbifold morphism $(\pmb{s}:(S^{2}_{orb},\bz)\to \E)$ that represents a section class $\sigma$, $\pmb{\pi}\circ \pmb{s}:(S^{2}_{orb},\bz)\to S^2$ is holomorphic and determines a biholomorphism $\pmb{\phi}$ from $(S^{2}_{orb},\bz)$ to $(S^{2}_{orb},\pmb{\pi}\circ \pmb{s}(\bz))$ by giving orbifold structure on the target $S^2$. Then $\pmb{s}\circ \pmb{\phi}^{-1}\in\M^{sec}_{0,k}(\E,J,\sigma, \mathbf{x})$. It is easy to see this is a 1-1 correspondence.
\end{proof}

\begin{numrmk}
\begin{enumerate}
\item For $\pmb{s}\in\M^{sec}_{0,k}(\E,J,\sigma, \mathbf{x})$, let $$D_{\pmb{s}}\partial_{J}: \Omega^{0}(\pmb{s}^{*}T^{vert}\E)\to \Omega^{0,1}(\pmb{s}^{*}T^{vert}\E)$$ be the linearization of $\partial_{J}$ at $\pmb{s}$. Here $\Omega^{0}(\pmb{s}^{*}T^{vert}\E)$ is the Banach space of $\Wkp$ sections in $\pmb{s}^{*}T^{vert}\E$,  $\Omega^{0,1}(\pmb{s}^{*}T^{vert}\E)$ is the Banach space of  $\pmb{s}^{*}T^{vert}\E$ valued 1-form with $\mathcal{W}^{k-1,p}$ smoothness. Using the proof of \cite[Proposition 4.1.4]{CR1}, we compute the Fredholm index of  $D_{\pmb{s}}\bar{\partial}_{J}$ as:
$$index(D_{\pmb{s}}\bar{\partial}_{J})=dim_{\mathbb{R}}\X+2c_{1}(T^{vert}\E)(\sigma)+2k-2\iota(\bf{x}).$$
Here $\iota(\mathbf{x})$ is the degree shifting number associated to the twisted sector $\mathbf{x}$ of $\E$ (see \cite[Lemma 3.2.4]{CR2} and \cite[Section 3.2]{CR1} for details). From the above lemma, one can see that $$dim\M^{sec}_{0,k}(\E,J,\sigma, \mathbf{x})=dim\M_{0,k}(\E,J,\sigma, \mathbf{x})-6.$$
On the other hand, assuming $J$ is regular, i.e. the dimension of moduli spaces equal the Fredholm index, one can see that 
\begin{eqnarray*}
\ \ \ \ dim\M_{0,k}(\E,J,\sigma, \mathbf{x})-6 & & =dim_{\mathbb{R}}\E+2c_{1}(T\E)(\sigma)+2k-2\iota(\mathbf{x})-6\\
& & =dim_{\mathbb{R}}\X+2+2(c_{1}(T^{vert}\E)(\sigma)+2)+2k-2\iota(\mathbf{x})-6\\
& & =dim_{\mathbb{R}}\X+2c_{1}(T^{vert}\E)(\sigma)+2k-2\iota(\mathbf{x})\\
& & =dim_{\mathbb{R}}\M^{sec}_{0,k}(\E,J,\sigma, \mathbf{x}).
\end{eqnarray*}
Therefore we see that the expected dimension formula is consistent with the above lemma.
\item 
By Gromov compactness of $J$-holomorphic orbifold morphisms, a sequence of $J$-holomorphic orbifold morphisms with bounded energy has a subsequence converging to a stable orbifold map (\cite{CR2}). Note that the proof of Gromov convergence indicates that there is a component of the stable orbifold map, such that the subsequence of $J$-holomorphic orbifold morphisms $C^{\infty}$ converges to that component on any compact subset of $S^2$ with finite points removed. If the sequence of $J$-holomorphic orbifold morphisms are sectional, this component is also sectional. Thus the Gromov compactification $\Mbar^{sec}_{0,k}(\E,J,\sigma,\mathbf{x})$ of $\M^{sec}_{0,k}(\E,J,\sigma, \mathbf{x})$ is well-defined. 
\end{enumerate}
We will not need the above result in this paper, thus will not pursue a detailed proof here.
\end{numrmk}

Now we consider the structure of orbifold stable morphisms into $\E$.  Since $J$ on $T\E$ restricts to an almost complex structure on the vertical tangent bundle $T^{vert}\E:=ker (d\pi) \subset T\E$, a component $\pmb{f}_{\nu}:(\Sigma_{\nu}, \bz_{\nu})\to\E$  of an orbifold stable morphism $\pmb{f}$ is either a $J$-holomorphic sectional orbifold morphism (up to a domain reparametrization) or a $J$-holomorphic orbifold morphism with image contained in a fiber. Furthermore, since it represents a section class, there is only one sectional component, i.e.
\begin{itemize}
\item There is an irreducible component $\Sigma_0$ such that the composition $S^2 \simeq |\Sigma_0|\overset{|\pmb{f}|}{\to} |\E| \overset{|\pmb{\pi}|}{\to} S^2$ is surjective and the induced map $H_2(|\Sigma_0|,\mathbb{Z})\to H_2(S^2,\mathbb{Z})$ is an isomorphism.
\item All other components are mapped into fibers of $\E\to S^2$.
\end{itemize}
Any $J$-holomorphic stable orbifold map representing a section class has exactly one sectional component (up to a domain reparametrization). This component is called the {\em stem component} of the orbifold stable morphism, all other components are called {\em branch components}. 

In the Hamiltonian orbifiber bundle case, we can construct Kuranishi structure as in Section \ref{secReviewoGW}.  Moreover we can choose the obstruction bundle to satisfy additional properties. 

\begin{lem}\label{lemvertobs}
Let $\sigma$ be a section class. Then the cokernel $cokerL_{\pmb{f}}$ of an orbifold stable morphism $\pmb{f}:(\Sigma,j,\vec z, D)\to\E$ representing the class $\sigma$ is spanned by elements in $L^p_{\delta}(\pmb{f}^{*}T^{vert}\E\otimes\Lambda^{0,1})$.
\end{lem}
\begin{proof} 
As observed above, a component $\pmb{f}_{\nu}:\Sigma_{\nu}\to \E$ of $\pmb{f}$ is either a stem component which is a sectional morphism or a branch component contained in a fiber. If $\pmb{f}_{\nu}$ is a stem component, then we have
$$\pmb{f}_{\nu}^*T\E=\pmb{f}_{\nu}^*T(\pmb{f}_{\nu}(\Sigma_{\nu}))\oplus \pmb{f}_{\nu}^*T^{vert}\E=\pmb{f}_{\nu}^*T(\pmb{f}_{\nu}(\Sigma_{\nu}))\oplus \pmb{f}^{*}T^{vert}\E.$$
Hence we have $$L_{\delta}^p(\pmb{f}^{*}T\E\otimes\Lambda^{0,1})=L_{\delta}^p(\pmb{f}_{\nu}^*T(\pmb{f}_{\nu}(\Sigma_{\nu}))\otimes\Lambda^{0,1})\oplus L_{\delta}^p(\pmb{f}^{*}T^{vert}\E \otimes\Lambda^{0,1}).$$ A direct calculation shows that $L_{\pmb{f}}$ maps $L_{1,\delta}^p(\pmb{f}_{\nu}^*T(\pmb{f}_{\nu}(\Sigma_{\nu})))$ onto $L_{\delta}^p(\pmb{f}_{\nu}^*T(\pmb{f}_{\nu}(\Sigma_{\nu}))\otimes\Lambda^{0,1})$.

If $\Sigma_{\nu}$ is mapped into a fiber, then $\pmb{f}_{\nu}^*T\E=\underline{\mathbb{C}}\oplus \pmb{f}_{\nu}^*T^{vert}\E$ and clearly $L_{\pmb{f}}$ maps $L_{1,\delta}^p(\underline{\mathbb{C}})$ onto $L_{\delta}^p(\underline{\mathbb{C}}\otimes\Lambda^{0,1})$.

Therefore, over each component $\Sigma_{\nu}$, $cokerL_{\pmb{f}}$ is spanned by elements in $L^p_{\delta}(\pmb{f}^{*}T^{vert}\E\otimes\Lambda^{0,1})$. From the definitions of these Sobolev spaces it is straightforward to check that the assertion holds for the whole $\Sigma$.
\end{proof}
Observe that if there is a subbundle $V$ of $T\E$ such that the cokernel $cokerL_{\pmb{f}}$ is spanned by elements in $L_{\delta}^p(\pmb{f}^*V\otimes\Lambda^{0,1})$, then we can choose the fiber of the obstruction bundle over $\pmb{f}$ to be contained in $L_{\delta}^p(\pmb{f}^*V\otimes\Lambda^{0,1})$. By Lemma \ref{lemvertobs} we can choose the obstruction bundle so that its fiber over $\pmb{f}$ is contained in $L^p_{\delta}(\pmb{f}^*T^{vert}\E\otimes\Lambda^{0,1})$. Although such kind of obstruction bundle is not necessary to define Seidel representation, it will be useful when we prove its properties. In what follows we choose the obstruction bundles this way.

\subsection{Definition of Seidel Representations}\label{secdefineSeidel}
In this subsection we give the definition of Seidel representation for symplectic orbifolds. Let $(\X, \omega)$ be a compact symplectic orbifold. Given a homotopy class $a\in \pi_{1}(Ham(\X,\omega))$, represent it by a Hamiltonian loop $\gamma$, then we can construct Hamiltonian orbifiber bundle $(\E_{\gamma},\Omega_{c})$ as in Section \ref{orbifiberbundleS2}. Denote $c_1(T^{vert}\E)$ by $c_1^{v}$. Let $\iota$ be an inclusion of a fiber over a point\footnote{Unless otherwise stated, we choose the north pole.} in $S^2$ into $\E_\gamma$. There is a Gysin map induced by this inclusion: $\iota_{*}: H^{*}(I\X,\mathbb{Q})\to H^{*+2}(I\E_{\gamma},\mathbb{Q})$. One can think of this map as a union of maps  from $H^{*}(\X_{(g)},\mathbb{Q})$ to $H^{*+2}(\E_{\gamma,(g)},\mathbb{Q})$, which makes sense because there is no orbifold structures along the horizontal direction.

\begin{defn}\label{defseidel}
Given a symplectic orbifold $(\X,\omega)$. The {\em Seidel element} for a homotopy class $a\in \pi_{1}(Ham(\X,\omega))$ is defined to be
$$
\mathcal{S}(a):=\sum_{\sigma\in H_2^{sec}(|\E_{\gamma}|)} (\sum_{i} \< \iota_*f_i\>^{\E_{\gamma}}_{0,1,\sigma}f^i) \otimes q^{c_1^{v}(\sigma)}t^{u_{\gamma}(\sigma)},
$$
where $\{f_i\}\subset H^*(I\X,\mathbb{Q})$ is an additive basis, and $\{f^i\}\subset H^*(I\X,\mathbb{Q})$ its dual basis with respect to the orbifold Poincar\'e pairing. 
\end{defn}
Because of Proposition \ref{homotopylem}, $\mathcal{S}(a)$ does not depend on the choice of Hamiltonian loop representing the homotopy class $a$, thus we have a well-defined map of sets $$\mathcal{S}: \pi_1(Ham(\X,\omega))\to QH^*_{orb}(\X, \Lambda)^\times,$$ called the {\em Seidel representation} of $(\X,\omega)$. The use of the term `` representation'' is justified by the following important properties of the map $\mathcal{S}$.

\begin{thm}\label{SeidelAxiom}
\hfill
\begin{enumerate}
\item {\bf Triviality property}:
Let $e\in \pi_1(Ham(\X, \omega))$ be the identity loop, then 
\begin{equation}\label{triviality}
S(e)=1.
\end{equation}
\item {\bf Composition property}: Let $a, b\in \pi_1(Ham(\X,\omega))$ and let $a\cdot b\in \pi_1(Ham(\X, \omega))$ be their product defined by loop composition (see Definition \ref{defn_fund_gp_Ham}). Then
\begin{equation}\label{composition}
\mathcal{S}(a\cdot b)=\mathcal{S}(a)*\mathcal{S}(b).
\end{equation}
\end{enumerate}
\end{thm}
Theorem \ref{SeidelAxiom} will be proved in the next two sections.

\begin{cor}\label{seidel_hom}
The map $$\mathcal{S}: \pi_1(Ham(\X,\omega))\to QH^*_{orb}(\X, \Lambda)^\times$$ is a group homomorphism.
\end{cor}
\begin{proof}
In view of Theorem \ref{SeidelAxiom}, it remains to show that the image of $\mathcal{S}$ is contained in $QH^*_{orb}(\X, \Lambda)^\times$. For $a\in \pi_{1}(Ham(\X,\omega))$, let $a^{-1}\in \pi_{1}(Ham(\X,\omega))$ be the inverse loop. Then $a\cdot a^{-1}=e$ in $\pi_1(Ham((\X,\omega))$. By (\ref{triviality})--(\ref{composition}), we calculate $\mathcal{S}(a)*\mathcal{S}(a^{-1})=\mathcal{S}(a\cdot a^{-1})=\mathcal{S}(e)=1$. Thus $\mathcal{S}(a)^{-1}=\mathcal{S}(a^{-1})$. This completes the proof.
\end{proof}
\section{Triviality property}\label{trivaxiom}
The purpose of this section is to prove the triviality property (\ref{triviality}).
\subsection{Proof of Triviality: set-up}\label{trivideal}
We can choose the constant loop to represent the trivial element $e\in \pi_{1}(Ham(\X,\omega))$, and the corresponding Hamiltonian orbifiber bundle is $\XxS\to S^2$. Let $c_{1}^{v}$ be the first Chern class of the vertical subbundle of $T(\XxS)$, $u_{e}$ the coupling class (see Definition \ref{couplingcl}) of this trivial Hamiltonian orbifiber bundle, $\sigma_{0}$ the section class determined by the constant section. Let $\tilde{J}_{0}$ be the direct sum of an almost complex structure $J$ on $\X$ and the complex structure $j_{0}$ on $S^{2}$. Then by definition the triviality property \eqref{triviality} may be written as 
\begin{equation}
\sum_{i}(\!\!\! \sum_{B\in H_{2}(|\X|,\mathbb{Z})}\int_{ev_{*}[\Mbar_{0,1}(\XxS,\sigma_{0}+\iota_{*}B,\tJ_{0},(g))]^{vir}} \iota_{*}f_{i})\cdot f^{i}\cdot q^{c_{1}^{v}(\sigma_{0}+\iota_{*}B)}t^{u_{e}(\sigma_{0}+\iota_{*}B)}=1\ (=PD_{\X_{(0)}}([\X_{(0)}])).
\end{equation}
Here $\X_{(0)}$ is the untwisted sector, $\X_{(g)}$ is the twisted sector supporting $\iota_{*}f_{i}$.

Note that $c_{1}^{v}(\sigma_{0})=0$ and $u_{e}(\sigma_{0})=0$, so to prove (\ref{triviality}) it is enough to show:
\begin{prop}\label{ModuliofTrivial}
\begin{itemize}
\item Case 1: If $B=0$ and the twisted sector $\mathbf{x}=(g)$ is nontrivial, then $$\Mbar_{0,1}(\XxS,\sigma_{0}+\iota_{*}B,\tJ_{0},(g))=\emptyset;$$
\item Case 2: If $B=0$ and the twisted sector $(g)=(0)$ is trivial, then $$ev_{*}[\Mbar_{0,1}(\XxS,\sigma_{0}+\iota_{*}B,\tJ_{0},(g))]^{vir}=[\X_{(0)}];$$
\item Case 3: If $B\neq 0$, then $$\int_{ev_{*}[\Mbar_{0,1}(\XxS,\sigma+\iota_{*}B,\tJ_{0},(g))]^{vir}} \iota_{*}f_{i}=0 \,\,\text{ for any }\,\, i.$$
\end{itemize}
\end{prop}

Case 1 of Proposition \ref{ModuliofTrivial} is a corollary of the following Lemma \ref{constantsectiontwistsector} (\ref{1markpoint}). 

\begin{lem}\label{constantsectiontwistsector} Let $\pmb{s}$ be a sectional orbifold morphism lifting a constant section of $|\pmb{\pi}|:|\XxS|\to S^{2}$,
\begin{enumerate}
\item \label{1markpoint} if $\pmb{s}$ allows at most one orbipoint on the domain, then it has trivial twisted sector, i.e. there is no orbipoint on the domain;
\item  if $\pmb{s}$ allows at most two orbipoints on the domain, then the two twisted sectors at the two orbipoints (if any), denoted by $(g_{1})$ and $(g_{2})$, are inverse to each other: $(g_{1})=(g_{2}^{-1})$.
\end{enumerate}
\end{lem}
\begin{proof}
The proof is a straightforward calculation as in the proof of Proposition \ref{constantsection}.
\end{proof}

Case 2 of Proposition \ref{ModuliofTrivial} is true because of the following Lemma:
\begin{lem}
For any $[(\pmb{s},S^{2}, z)]\in \M_{0,1}(\XxS,\sigma_{0},\tJ_{0},(0))$, $\pmb{s}$ is a constant section of $\XxS$ upto an automorphism of $(S^{2},j_{0})$.
\end{lem}
\begin{proof}
By Lemma \ref{constantsectiontwistsector} (\ref{1markpoint}), $z$ is not an orbipoint. So this is actually a special case of \cite[Lemma 3.1]{MT}. Recall the definition of symplectic form $\Omega_{c}$ in (\ref{fiberbundlesymplecticform}), in the special case of trivial orbifiber bundle, $\Omega_{c}=\omega_{\X}+c\omega_{S^{2}}$, $c>0$.
For $\xi =h+v \in T_{(x,z)}(\XxS)$, 
\begin{equation}\label{positivityofnonconstant}
\Omega_{c}(\xi,\tilde{J}_{0}\xi)  =  (\omega_{\X} + c\omega_{S^{2}})(h+v,j\cdot h+Jv)  = \omega_{\X}(v,Jv)+c\omega_{S^{2}}(h,j_{0}\cdot h)  \ge c\omega_{S^{2}}(h,j_{0}\cdot h)
\end{equation}
The last identity holds if and only if $v=0$. 

Choose an open cover $\{U_{\alpha}\}$ of $\Sigma$ so that on each open set, we can choose a conformal coordinate $z=s+i\,t$. Denote $\pmb{s}=(u,\phi):S^{2}\to \XxS$.  On each open set $U$ we have
  
 \begin{eqnarray*}
  \int_{U} \Omega_{c}(d\pmb{s}\frac{\partial}{\partial s},d\pmb{s}\frac{\partial}{\partial t})ds\wedge dt & & \negthickspace\negthickspace  =\int_{U} \Omega_{c}(d\pmb{s}\frac{\partial}{\partial s},d\pmb{s}\cdot i\frac{\partial}{\partial s})ds\wedge dt =\int_{U} \Omega_{c}(d\pmb{s}\frac{\partial}{\partial s},\tilde{J}_{0}\cdot d\pmb{s}\cdot\frac{\partial}{\partial s})ds\wedge dt\\
  & & \negthickspace\negthickspace =\int_{U} \Omega_{c}(\frac{\partial\pmb{s}}{\partial s},\tilde{J}_{0}\cdot\frac{\partial\pmb{s}}{\partial s})ds\wedge dt=\int_{U} \Omega_{c}((\frac{\partial u}{\partial s},\frac{\partial \phi}{\partial s}),\tilde{J}_{0}\cdot(\frac{\partial u}{\partial s},\frac{\partial \phi}{\partial s}))ds\wedge dt\\
 & & \negthickspace\negthickspace \ge \int_{U}c\omega_{S^{2}}(\frac{\partial \phi}{\partial s},i\cdot \frac{\partial \phi}{\partial s}).
\end{eqnarray*}
 
By (\ref{positivityofnonconstant}), the last equality holds if and only if $\frac{\partial u}{\partial s}=0$ which means $u$ is constant on $U$ since we have the freedom of choice of $s$.

 If $u$ is not a constant map, then there exists at least one open set over which 
 $$\int_{U} \Omega_{c}(d\pmb{s}\frac{\partial}{\partial s},d\pmb{s}\frac{\partial}{\partial t})ds\wedge dt > \int_{U}c\omega_{S^{2}}(\frac{\partial \phi}{\partial s},i\cdot \frac{\partial \phi}{\partial s}).$$
Thus $\Omega_{c}([\pmb{s}])>\Omega_{c}(\sigma_{0})$. This contradicts the fact that $[\pmb{s}]=\sigma_{0}$ and $\Omega_{c}([\pmb{s}])=\Omega_{c}(\sigma_{0})$. Therefore $u$ is a constant map and $\pmb{s}$ is a constant section of $\XxS$ upto an automorphism of $(S^{2},j_{0})$.
\end{proof}

From now on we focus on the remaining Case 3. The idea of our proof of Case 3 is that there is a ``Lie group action'' on $\Mbar_{0,1}(\XxS,\sigma+\iota_{*}B,\tJ_{0},(g))$ under which the evaluation map is invariant. This ``Lie group action'' has two additional properties:
\begin{enumerate}
\item away from the fixed loci, the ``group action'' is  locally free, and preserves the evaluation map;
\item near the fixed loci, the evaluation map can be perturbed to another map which is homotopy equivalent to the original one, and whose image of a neighborhood of the fixed loci has dimension lower the the virtual dimension.
\end{enumerate}

Because of the above two properties, the virtual fundamental cycle has lower dimension than its expected dimension. Thus a generic choice of a cycle representing the Poincare dual of $\iota_{*} a$ will not intersect the virtual fundamental cycle. This proves Case 3.

However to make the idea rigorous, there are two points one needs to be careful. First, the ``group action'' is actually a parametrized group action, which we explain in the next subsection. Second, all the statements above should be understood in the context of Kuranishi spaces rather than topological spaces. In this context, a lower dimensional subset in the moduli space may have nontrivial contribution. So the contribution from the fixed loci needs to be understood carefully. That is why we need to modify the evaluation map. The purpose of the rest of this section is to treat these points.

\subsection{Parametrized Group Action on Manifolds and Orbifolds}\label{secParaAction}
We first recall the definition of group bundle, which is an analogue of group scheme in the topological/smooth setting.
\begin{defn}\label{groupbundle}
Let $G$ be a Lie group, a fiber bundle $\pi: P\to B$ with fiber the manifold $G$ is called a {\em $G$-group bundle} if there is a product operation on each fiber and the local trivialization maps are fiberwise group isomorphisms.
\end{defn}
Note that a $G$-group bundle has more structure than a principal $G$-bundle. It is the counterpart of group scheme in the differential topology context.

\begin{defn}\label{paraaction}
Let $G$ be a Lie group, $\pi: P\to B$ a $G$-group bundle, $M$ a topological space (respectively smooth manifold/orbifold). Then a \emph{parametrized $G$-action on $M$} consists of the following data:
\begin{enumerate}
\item an surjective continuous map (respectively submersive smooth map/orbifold morphism) $a: M\to B$, called the anchor map;
\item a continuous map (respectively smooth map/orbifold morphism) $\alpha: P  {}_{\pi}  \! \times_{a} M\to M$, called the action map, such that $\forall b\in B$, the restriction of $\alpha$ to $\pi^{-1}(b)\times a^{-1}(b)$ defines a group action on $a^{-1}(b)$.
\end{enumerate}
We will also say $M$ has a $G$-action parametrized by $B$.
\end{defn}

\begin{rmk}
\begin{enumerate}
\item Note that if the group bundle is trivial, then the parametrized group action is actually a group action.
\item A group bundle $\pi:P\to B$ is a groupoid whose objects are $B$, morphisms are $P$, and both source and target maps send $p\in P$ to $\pi(p)$. Then our definition of parametrized group action on manifold is a special case of a groupoid action on a manifold. (See \cite[Definition 3.16]{Le}.)
\end{enumerate}
\end{rmk}

Similar to group action, we have the following definitions:
\begin{defn}
For $b\in B$ and $x\in a^{-1}(b)$, $\pi^{-1}(b)\cdot x$ is called an {\em orbit} of the parametrized group action. The {\em fixed locus} of a parametrized $G$-action on $M$ is defined as  $$ \cup_{b\in B} Fix(\pi^{-1}(b)),$$ where $Fix(\pi^{-1}(b))$ is the fixed locus of the group action of $\pi^{-1}(b)$ on $M$.

A parametrized group action is called {\em free} if for any $x$ in $M$, $\alpha(p,x) =\alpha(p',x)$ implies $p=p'$. A parametrized group action is called {\em locally free} if there is a small neighborhood $U_{b}$ of the identity in $\pi^{-1}(b)$ such that the restriction of the action to $U_{b}$ is free.

A parametrized group action is said to {\em have finite stabilizer everywhere} if for any $x$ in $M$, the set $\{p|\alpha(p,x) =x\}$ is finite.
\end{defn}
Note that if the group is compact, a locally free parametrized group action always has finite stabilizers everywhere. But if the group is not compact, this may not be true.

The following toy example of parametrized $G$-actions will be used later in this paper.
\begin{exm}\label{paraAction}
We construct a parametrized group action on the manifold $\com P^{1}\times \com P^{1}$. We start with the trivial principal $PSL(2,\com)$-bundle  $\com P^{1}\times PSL(2,\com)\to \com P^{1}$ with the projection to the first component $p_{1}:\com P^{1}\times \com P^{1} \to \com P^{1}$ as anchor map, and the action map is given by: $\forall z\in \com P^{1}$,
\begin{equation*}
\begin{split}
&\alpha: \{z\}\times PSL(2,\com)\times \{z\}\times \com P^{1} \to \{z\}\times \com P^{1}\\
&\quad \quad \alpha((z,g);(z,w)):=(z,g\cdot w).
\end{split}
\end{equation*}
This parametrized action itself is not interesting. We consider its restriction to the subbundle 
$$P:=\{(z,g)|z\in \com P^{1},\ g\in PSL(2,\com),\ g\cdot z=z\}\to \com P^1.$$ 
The fiber of the group bundle $P\to \com P^{1}$ is the subgroup of $PSL(2,\com)$ which fixes one point on $\com P^{1}$. It is isomorphic to  $\mathbb{A}:=\{f:\com\to \com| f(z)=az+b\}$.
Then $\alpha$ restricts to a parametrized $\mathbb{A}$-action on $\com P^{1}\times \com P^{1}$.

The fixed locus of this parametrized $\mathbb{A}$-action is the diagonal $\Delta$ of $\com P^{1}\times \com P^{1}$.

The group bundle $P\to \com P^{1}$ has two subbundles: one principal $S^{1}$-bundle whose fiber at $z$ consists of rotations around $z$, fixing $z$ and its antipodal, is denoted by $P^{S^{1}}\to \com P^{1}$; one principal $\com$-bundle whose fiber at $z$ consists of translations of $\com P^{1}\setminus\{z\}$, is denoted by $P^{\com}\to \com P^{1}$. Restriction of the parametrized $\mathbb{A}$-action $\alpha$ to either of the two subbundles defines a new parametrized group action. The fixed locus of the parametrized $\mathbb{C}$-action is again the diagonal $\Delta$ of $\com P^{1}\times \com P^{1}$, and the parametrized $\mathbb{C}$-action is free on $\com P^{1}\times \com P^{1}\setminus \Delta$.
\end{exm}

\begin{rmk}
In Example \ref{paraAction}, since $P\to \com P^{1}$ is not a trivial group bundle, there is no canonical way of identifying the fibers with $\mathbb{A}$, so it cannot be understood as an $\mathbb{A}$ action on $\com P^{1}$. This is the reason we need to introduce parametrized group action. Note that if we work with a contractible parameter space, then the parametrized group action can always be identified with a group action.
\end{rmk}
\begin{defn}\label{actioncommute}
If $M$ has a group action by $\Gamma$, and a parametrized $G$-action by $\pi: P\to B$, then we say the two actions {\em commute} if the $\Gamma$-action commute with the group action determined by the parametrized $G$-action on each fiber. 
\end{defn}

\begin{defn}\label{paraequivmap}
Let $M$, $N$ be two manifolds with parametrized $G$-action by $\pi: P\to B$. A map $f:M\to N$ is called parametrized $G$-equivariant if $f$ is $\pi^{-1}(b)$-equivariant for every $b\in B$. 
\end{defn}

\begin{defn}\label{paraequivbundle}
Let $pr:E\to M$ be a vector bundle,  $(a_{E},\alpha_{E})$ and $(a_{M},\alpha_{M})$ two parametrized $G$-action on $E$ and $M$ respectively by $\pi:P\to B$. Then $(pr:E\to M, a_{E}, a_{M},\alpha_{E},\alpha_{M})$ is called a parameterized $G$-equivariant vector bundle if $a_{E}=a_{M}\circ pr$ and the following diagram is commutative for any $b\in B$:
$$
\xymatrix{
 \pi^{-1}(b)\times a_{E}^{-1}(b)  \ar[d]_{(id_{P},pr)}\ar[r]^{\ \ \ \alpha_{E}} &a_{E}^{-1}(b) \ar[d]^{pr}\\
 \pi^{-1}(b)\times a_{M}^{-1}(b) \ar[r]^{\ \ \ \alpha_{M}} & a_{M}^{-1}(b).
}
$$
\end{defn}

\subsection{Parametrized Equivariant Kuranishi Structure}\label{secEqKura}
In \cite{FOOOtoric1},  Kuranishi structure with a group action, called  equivariant  Kuranishi structure, was considered. In this paper, we need the generalization of Kuranishi structures with parametrized group actions.

\begin{defn}\label{eqkurastr}
Let  $\M$ be a topological space with a parametrized $G$-action by a group bundle $P\to B$. A Kuranishi structure $\{(V_{\alpha},E_{\alpha}, \Gamma_{\alpha},\psi_{\alpha},s_{\alpha})\}_{\alpha \in \mathfrak A}$ on $\mathcal M$ is called {\em  parametrized $G$-equivariant in the strong sense} if the following holds :
\begin{enumerate}
\item $V_{\alpha}$ has a parametrized $G$-action  which commutes with the $\Gamma_{\alpha}$-action.
\item $E_{\alpha}$ is a parametrized $G$-equivariant bundle.
\item The Kuranishi map $s_{\alpha}$
is parametrized $G$-equivariant and $\psi_{\alpha}$ is a parametrized $G$-equivariant map.
\item
The coordinate changes $\varphi_{\alpha_2,\alpha_1}$ are parametrized $G$-equivariant.
\end{enumerate}
\end{defn}

The reason we use ``in the strong sense'' here is the same as in \cite[Remark B.5]{FOOOtoric1}: the parametrized $G$-action is on $V$ instead of on the orbifold $[V/\Gamma]$. Hereafter, we refer it as  parametrized $G$-equivariant for simplicity.

The main purpose of this subsection is to construct a parametrized $\mathbb{A}$-equivariant Kuranishi structure on $\Mbar_{0,1}(\XxS,\sigma_{0}+\iota_{*}B,\tJ_{0},(g))$. More precisely, let $\pi: P\to \com P^{1}$ be the group bundle in Example \ref{paraAction} with fiber $\mathbb{A}$. We will define a parametrized $\mathbb{A}$-action on the topological space $\Mbar_{0,1}(\XxS,\sigma_{0}+\iota_{*}B,\tJ_{0},(g))$ by $P\to \com P^1$. 

For a stable orbifold map representing $\tau=[\pmb{f}]\in\Mbar_{0,1}(\XxS,\sigma_{0}+\iota_{*}B,\tJ_{0},(g))$, let $z$ be its only marked point, we define $p(z)$ as follows:
\begin{itemize}
\item  $p(z)=z$ if it lies on the stem component;
\item if $z$ is not on the stem component, then let $p(z)$ be the nodal point on the stem component to which the branch containing the marked point $z$ is attached.
\end{itemize}

Note that there is an obvious $PSL(2,\com)$ action on the moduli space by post-composing $g\in PSL(2,\com)$ to the second component of $\pmb{f}$. This $PSL(2,\com)$-action induces a parametrized $\mathbb{A}$-action on the moduli space. More precisely:
\begin{defn}\label{defparaAction} 
The parametrized $\mathbb{A}$-action on $\Mbar_{0,1}(\XxS,\sigma_{0}+\iota_{*}B,\tJ_{0},(g))$ is defined by:
\begin{enumerate}
\item the anchor map $a: \Mbar_{0,1}(\XxS,\sigma_{0}+\iota_{*}B,\tJ_{0},(g))\to \com P^{1}$ is the evaluation of $\pmb{\phi}$ at $p(z)$, for $\tau=[\{(\pmb{f_{\nu}}=(\pmb{u}_{\nu},w_{\nu})\ or\ (\pmb{u}_{stm},\pmb{\phi}))\}]\in\Mbar_{0,1}(\XxS,\sigma_{0}+\iota_{*}B,\tJ_{0},(g))$, and $z$ the marked point of a stable orbifold map representing $\tau$. 
\item the action map $\alpha$ is defined by: for $(w,h)\in P\subset \com P^{1}\times PSL(2,\com)$ and $\tau\in a^{-1}(w)$, $$\alpha((w,h), \tau):=(w,h)\cdot \tau:=[(h\cdot\pmb{f}_{\nu})],$$ where $h\cdot\pmb{f}_{\nu}$ is defined by $h$ post-composed to the horizontal component ($S^{2}$-component) of $\pmb{f}_{\nu}$. More explicitly, we define $h\cdot \tau$ as follows:
\begin{itemize}
\item the stem component $(\pmb{u}_{stm},\pmb{\phi})$ of $\tau$ is mapped to $(\pmb{u}_{stm},h\cdot \pmb{\phi})$;
\item branch components $(\pmb{u}_{\nu},w_{\nu})$ (if any) are mapped to $(\pmb{u}_{\nu},h\cdot w_{\nu})$.
\end{itemize}
\end{enumerate}
\end{defn}

Note that since a branch component only lies in a fiber, one can also think of the anchor map as evaluating at the marked point first, and then projecting to $S^{2}$. So the anchor map does not depend on the choice of the stable orbifold map representing $\tau$ because evaluation maps are well-defined for equivalence classes of stable orbifold maps.

\begin{defn}\label{defsubparaAction}
Similar to Example \ref{paraAction}, the parametrized $\mathbb{A}$-action in Definition \ref{defparaAction} restricts to a parametrized $S^{1}$-action by $P^{S^{1}}$ and a parametrized $\com$-action by $P^{\com}$.
\end{defn}

We denote $\Mbar_{0,2}(\XxS,\iota_{*}B,\tJ_{0},(g,0)) {}_{ev_{1}}  \! \times_{ev_{0}} \Mbar_{0,1}(\XxS,\sigma_{0},\tJ_{0},(0))$ by $\M^{1fib}$, where $0$ stands for the untwisted sector, and $B\in H_{2}(|\X|,\mathbb{Z})$, and use the notation $\M^{2fib}$ for $$\Mbar_{0,2}(\XxS,\iota_{*}B_{1},\tJ_{0},(g,0)) {}_{ev_{1}}  \! \times_{ev_{0}} \Mbar_{0,2}(\XxS,\sigma_{0},\tJ_{0},(0)) {}_{ev_{1}}  \! \times_{ev} \Mbar_{0,1}(\XxS,\iota_{*}B_{2},\tJ_{0},(g,0)). $$ 
Let $\M^{freak}:=\M^{1fib}\cup \M^{2fib}$. 

We remark that while $\M^{1fib}$ is compact, $\M^{2fib}$ is not compact. The space $\Mbar^{2fib}\cap\M^{1fib}$ contains stable maps with one fiber component and for which the sphere attached to the stem component is a ghost component.

\begin{lem}\label{paraActionLemma}
\begin{enumerate}
\item\label{welldefined} The parametrized $\mathbb{A}$-action in Definition \ref{defparaAction} is well-defined, and thus so are the parametrized $S^{1}$-action and parametrized $\com$-action in Definition \ref{defsubparaAction};
\item \label{fixedloc} The fixed locus of the parametrized $\mathbb{A}$-action is $\M^{1fib}$, and thus $\M^{1fib}$ is also fixed by the parametrized $S^{1}$-action and parametrized $\com$-action;
\item \label{2fibfree}  The restriction of the parametrized $\com$-action on $\M^{2fib}$ is free;
\item \label{locfree}  The restriction of the parametrized $\mathbb{A}$-action on $$\Mbar_{0,1}(\XxS,\sigma_{0}+\iota_{*}B,\tJ_{0},(g))\setminus \M^{freak}$$ is locally free.
\end{enumerate}
\end{lem}

\begin{proof} 
We prove the statements one by one:
\begin{itemize}
\item[(\ref{welldefined})] 
It is well-known that evaluation map is continuous, so the anchor map is continuous. By moving around the marked point or the branch component containing the marked point on the stem component, one can see the anchor map is surjective. 

Next we show that the map $\alpha((w,h),\tau):=(w,h)\cdot \tau$ defined in Definition \ref{defparaAction} is continuous. It is obvious that $\alpha$ is continuous with respect to $(w,h)$. To see it is also continuous with respect to $\tau$, take any sequence of $\{t_{k}\}$ in $\Mbar_{0,1}(\XxS,\sigma_{0}+\iota_{*}B,\tJ_{0},(g))$ which converges to $\tau$ in the sense of Gromov convergence. For simplicity, we assume $\tau_{k}$'s are represented by $J$-holomorphic maps with irreducible domains, i.e. $\tau_{k}=[\pmb{s}_{k}]=[(\pmb{u}_{k},\pmb{\phi}_{k})]$. More complicated cases follow from the same argument with messier notation. Let $\tau=[\{(\pmb{f_{\nu}}=(\pmb{u}_{\nu},w_{\nu})\ or\ (\pmb{u}_{stm},\pmb{\phi}))\}]$, then by the definition of Gromov compactness, for each $\nu$,  there is a subsequence of $\pmb{s}_{k_{j}}$ and a sequence of automorphisms $\xi_{\nu}^{j}$ of $S^{2}$ such that $\pmb{s}_{k_{j}}\circ\xi_{\nu}^{j}\to \pmb{f}_{\nu}$. For $(w,h)\in P$, $(w,h)\cdot \tau_{n}=[(\pmb{u}_{k},h\cdot \pmb{\phi}_{k})]$. For each $\nu$, we have $ h\cdot \pmb{s}_{k_{j}}\circ\xi_{\nu}^{j} \to h\cdot \pmb{f}_{\nu}$. Thus $(w,h)\cdot \tau_{k}$ converge to $(w,h)\cdot \tau=[\{(h\cdot\pmb{f}_{\nu}\}]$.

It is straightforward to check that Definition \ref{defparaAction} defines a group action when restricted to $\pi^{-1}(w)$, for any $w\in \com P^{1}$. So the parametrized group action is well-defined.

\item[(\ref{fixedloc})]
Now we examine the fixed locus of $\pi^{-1}(w)$. For $\tau\in \Mbar_{0,1}(\XxS,\sigma_{0}+\iota_{*}B,\tJ_{0},(g))$, if $\forall (w,h)\in \pi^{-1}(w)\subset P$,  $(w,h)\cdot \tau=\tau$, then there exists $\psi_{(w,h)}\in PSL(2,\com)$ such that: 
\begin{equation}
(u_{stm},h\cdot \pmb{\phi})=(u_{stm}\circ \psi_{(w,h)},\pmb{\phi}\circ\psi_{(w,h)})
\end{equation}
and 
$$\psi_{(w,h)}(p(z))=p(z).$$ 
Here $\pmb{\phi}\circ\psi_{(w,h)}$ should be understood in an orbifold fashion. Namely lift $\psi_{(w,h)}:S^{2}\to S^{2}$ to an orbifold diffeomorphism $\pmb{\psi}_{(w,h)}:S^{2}_{orb}\to S^{2}_{orb}$, and then precompose $\pmb{\psi}_{(w,h)}$ to $\pmb{\phi}$.

Evaluate the identity of the second components $h\cdot \pmb{\phi}=\pmb{\phi}\circ\psi_{(w,h)}$ at $p(z)$, we get $h\cdot \pmb{\phi}(p(z))=\pmb{\phi}(p(z))$, therefore $w=\pmb{\phi}(p(z))$. Note that the group $$\{\psi_{(w,h)}=\pmb{\phi}^{-1}\circ h\cdot \pmb{\phi}|\forall (w,h)\in \pi^{-1}(w)\}$$ acts transitively on $\com P^{1}\setminus \{w\}$. Since the identity of the first components $u_{stm}=u_{stm}\circ \psi_{(w,h)}$ holds for all $(w,h)\in \pi^{-1}(w)$, $u_{stm}$ has to be a constant map. We have shown 
$$Fix\ \pi^{-1}(w)= \{\tau\in \Mbar_{0,1}(\XxS,\sigma_{0}+\iota_{*}B,\tJ_{0},(g))|u_{stm}\ \text{of}\ \tau\ \text{is} \ \text{constant};\ \pmb{\phi}(p(z))=w. \}$$
Let $w$ run over $\com P^{1}$, we get the fixed locus of the parametrized $\mathbb{A}$-action:
\begin{equation*}
\begin{split}
\cup_{w\in \com P^{1}} Fix\ \pi^{-1}(w)&=\{\tau\in \Mbar_{0,1}(\XxS,\sigma_{0}+\iota_{*}B,\tJ_{0},(g))|u_{stm}\ \text{of}\ \tau\ \text{is} \ \text{constant}. \}\\
&=\Mbar_{0,2}(\XxS,\iota_{* }B,\tJ_{0},(g,0)) {}_{ev_{1}}  \! \times_{ev_{0}} \Mbar_{0,1}(\XxS,\sigma_{0},\tJ_{0},(0)).
\end{split}
\end{equation*}

\item[(\ref{2fibfree})] 
Let $\pi_{\com}:P^{\com}\to \com P^{1}$ be the bundle projection map. If $\tau\in \M^{2fib}$ is fixed by an element  $(w,h)\in\pi_{\com}^{-1}(w)\subset P^{\com}$ then $ (w,h)$ fixes the branch component which does not contain $w$. The only such element is the identity in $\pi_{\com}^{-1}(w)$. So the parametrized $\com$-action on $\M^{2fib}$ is free.

\item[(\ref{locfree})]
Note that $\tau\in\Mbar_{0,1}(\XxS,\sigma_{0}+\iota_{*}B,\tJ_{0},(g))\setminus \M^{freak}$ is fixed by an element  $ (w,h)\in\pi^{-1}(w)$ if and only if $ (w,h)$ fixes the stem component and permutes the branch components which do not contain $w$. By the same argument as above, the map on the stem component has to be constant. We denote by $w_{1},...w_{N}$ the points on the stem component where the branch components are attached. Since $\tau\notin \M^{freak}$, $N>1$. The number of branch components $N+1$ is bounded by $\Omega_{c}([\sigma_{0}+\iota_{*}B])/C$, where $C$ denotes the minimal energy of non-constant pseudoholomorphic spheres in $\XxS$. 
 
Without loss of generality, we assume $w=\infty$, and other situations can be adapted to this case by an automorphism of $\com P^{1}$. Then $f\in\pi^{-1}(w)$ can be written as $f(z)=az+b$. We use the obvious norm on this Lie group $\parallel f \parallel :=|a|+|b|$. If $f$ fixes $w_{1},...,w_{N}$, then there exists a $\sigma\in S_{N}$ (with $S_{N}$ being the symmetric group on $n$ letters), such that:
 \begin{equation}\label{permutingeqn}
f(w_{i})=w_{\sigma(i)},\ \ \ \ \ i.e. \ \ aw_{i}+b=w_{\sigma(i)}.
\end{equation}
\begin{lem}\label{permutationmobius}
If $a$ satisfies (\ref{permutingeqn}), then there exists $K$ such that $a^{K}=1$ and $1<K\le N$.
\end{lem}
\begin{proof}[Proof of Lemma \ref{permutationmobius}]
Let $\bar{\sigma}$ be a non-trivial cycle in $\sigma$. Without loss of generality, we assume $\bar{\sigma}=(12...K)$, $1<K\le N$. By taking difference of the $(i+1)$-th equation and the $i$-th equation in (\ref{permutingeqn}), we get
$$
a=\frac{w_{2}-w_{1}}{w_{1}-w_{K}}=\frac{w_{3}-w_{2}}{w_{2}-w_{1}}=...=\frac{w_{1}-w_{K}}{w_{K}-w_{K-1}}.
$$
Thus 
$$
a^{K}=\frac{w_{2}-w_{1}}{w_{1}-w_{K}}\cdot\frac{w_{3}-w_{2}}{w_{2}-w_{1}}\cdot...\cdot\frac{w_{1}-w_{K}}{w_{K}-w_{K-1}}=1.
$$
\end{proof}
By Lemma \ref{permutationmobius}, the action on $\Mbar_{0,1}(\XxS,\sigma_{0}+\iota_{*}B,\tJ_{0},(g))\setminus \M^{freak}$ has finite stabilizers everywhere, and $\parallel f-id \parallel\ge|e^{i\frac{2\pi}{K}}-1 |$. If we choose $\epsilon=\frac{1}{2} |e^{i\frac{2\pi}{N}}-1 |$, then on the small neighborhood $U_{\epsilon}:=\{f\in\pi^{-1}(w) |\parallel f-id \parallel <\epsilon\}$ of the identity in $\pi^{-1}(w)$, the restriction of the action  is free. Therefore the parametrized group action is locally free on $\Mbar_{0,1}(\XxS,\sigma_{0}+\iota_{*}B,\tJ_{0},(g))\setminus \M^{freak}.$
\end{itemize}
The proof of Lemma \ref{paraActionLemma} is thus complete.
\end{proof}

Next we lift the above parametrized group action to a parametrized action on the Kuranishi spaces.
\begin{prop}\label{constructequivkura}
There exists a parametrized $\mathbb{A}$-equivariant Kuranishi structure $$\{(V_{\alpha},E_{\alpha},
\Gamma_{\alpha},\psi_{\alpha},s_{\alpha})|\alpha \in \mathfrak A\}$$ on $\Mbar_{0,1}(\XxS,\sigma_{0}+\iota_{*}B,\tJ_{0},(g))$ Êwith the following properties:
\begin{enumerate}
\item  The parametrized $\mathbb{A}$-action on the topological space $\Mbar_{0,1}(\XxS,\sigma_{0}+\iota_{*}B,\tJ_{0},(g))$ is given by Definition \ref{defparaAction};

\item \label{kuracompatibility} The restriction of the Kuranishi structure to $\M^{freak}$ coincides with the fiber product of the Kuranishi structures on  $\M^{1fib}$ and $\M^{2fib}$;

\item \label{locallyfree} The parametrized $\mathbb{A}$-action on the Kuranishi charts which do not cover any element in $\M^{freak}$ is locally free. If a Kuranishi chart $V_{\alpha}$ contains $\tau\in \M^{freak}$, then there is a Kuranishi chart $\check{V}_{\tau}$ in the fiber product Kuranishi structure of $\M^{freak}$, such that the parametrized $\mathbb{A}$-action on $V_{\alpha}\setminus \check{V}_{\alpha}$ is locally free and has finite stabilizers everywhere.

\item \label{vertobscond} A fiber of the obstruction bundles $E_{\alpha}$ over any $(\text{\bf v},\pmb{f}')\in V_{\alpha}$ is contained in $L^p_{\delta}(\pmb{f}'^{*}T^{vert}(\XxS)\otimes\Lambda^{0,1})$.
\end{enumerate}
\end{prop}

\begin{proof}
We need to modify the construction of Kuranishi structures in Section \ref{secReviewoGW} so that these required properties are satisfied. 

Consider an element $\tau=[\pmb{f}] \in \Mbar_{0,1}(\XxS,\sigma_{0}+\iota_{*}B,\tJ_{0},(g)).$ We construct a  Kuranishi chart over the $PSL(2,\com)$-orbit $\mathcal{O}_{\tau}$ of $\tau$. In particular, the chart is parametrized $\mathbb{A}$-equivariant.

Let $\text{\bf v} \in \frak V$ and approximated solutions $\pmb{f}' : \Sigma(\text{\bf v}) \to \XxS$ be as in Section \ref{secReviewoGW} (assume additional marked points are added if neccessary). By Lemma \ref{lemvertobs}, we do not need to perturb a stable map along the $S^{2}$ direction, thus we may assume the composition $\Sigma(\text{\bf v}) \xrightarrow{\pmb{f}} \XxS\xrightarrow{\pi_{S^{2}}} S^{2}$ to be holomorphic. 
The stem component of $\pmb{f}'$ determines an automorphism $g\in PSL(2,\com)$ via the above composition. Then we consider the pairs $(\text{\bf v},\pmb{f}')$ satisfying the following condition:
\begin{conds}\label{condequnbh}
There exists $\epsilon > 0$ depending only on $\tau$, with the following properties.
\begin{enumerate}
\item $\sup_{x \in \Sigma_0} \text{\rm dist} (\pmb{f}'(i_{\text{\bf v}}(x)),g\cdot \pmb{f}(x)) \le \epsilon$.
\item
Let $D_{c}$ be a connected component of $\Sigma(\text{\bf v}) \setminus \text{\rm Im}(i_{\text{\bf v}})$, the diameter of $\pmb{f}'(D_c)$
in $\X$ is smaller than $\epsilon$.
\end{enumerate}
\end{conds}

Let $E_{0,\nu}$'s be as in Section \ref{secReviewoGW}. Define an embedding
$$
I_{(\text{\bf v},\pmb{f}')} : \bigoplus_{\nu} E_{0,\nu}
\longrightarrow \pmb{f}^{\prime *}(T (\XxS)) \otimes \Lambda^{0,1}
(\Sigma(\text{\bf v}))
$$
as follows: We first push $E_{0,\nu}$'s by $ g$ to get obstruction bundle at $g\cdot \pmb{f}$, and then use parallel transport to define obstruction bundle at $\pmb{f}'$. 

The action of $g$ on $\pmb{f}=(\pmb{u}, \pmb{\phi})$ induces an isomorphism
$$
g_{*} : T_{\pmb{u}(x)} \X \oplus T_{\pmb{\phi}(x)}S^{2} \otimes \Lambda^{0,1}_x(\Sigma)
\cong T_{\pmb{u}(x)} \X \oplus T_{g\cdot\pmb{\phi}(x)}S^{2} \otimes \Lambda^{0,1}_x(\Sigma).
$$
which restricts to identity on $T_{\pmb{u}(x)} \X \otimes \Lambda^{0,1}_x(\Sigma)$. 
By Lemma \ref{lemvertobs}, $E_{0,\nu}$'s evaluate in the vertical direction $T\X$. 

For $\pmb{f}'$ satisfying Condition \ref{condequnbh}, we use the parallel transport along the $\X$ direction to define
$$
\bigoplus_{\nu} g_*(E_{0,\nu})
\longrightarrow \pmb{f}^{\prime *}(T\X) \otimes \Lambda^{0,1}
(\Sigma(\text{\bf v})).
$$

Then $I_{(\text{\bf v},\pmb{f}')}$ is defined as the composition of the above map with $g_{*}$. Now we consider the equation
\begin{equation}\label{eqperturbedeq}
\overline{\partial} \pmb{f}' \equiv 0
\mod \bigoplus_{\nu} I_{(\text{\bf v},\pmb{f}')}(E_{0,\nu}),
\end{equation}
together with the additional conditions
\begin{equation}\label{markedaddeq}
\pmb{f}'(z_{\nu,i}) \in g(N_{\nu,i}), \ \ \ \ for\ all\ added\ marked\ points\ z_{\nu,i}.
\end{equation}
as before. The set of solutions of these equations is denoted by $V_{\tau}$ which will be the Kuranishi chart over $\mathcal{O}_{\tau}$. Note that $V_{\tau}$ is $\Gamma_{\tau}$ invariant since the equation and conditions are $\Gamma_{\tau}$ invariant. 

The equation (\ref{eqperturbedeq}), (\ref{markedaddeq}), and Condition \ref{condequnbh} are $PSL(2,\com)$ invariant, thus $V_{\tau}$ is $PSL(2,\com)$ invariant. Use evaluation map composed with the projection of the Hamiltonian orbifiber bundle as anchor map, and restrict the $PSL(2,\com)$ action to $$P=\{(z,g)|z\in \com P^{1},\ g\in PSL(2,\com),\ g\cdot z=z\}$$ as in Defintion \ref{defparaAction}, then we define a paremetrized $\mathbb{A}$-action on the Kuranishi chart.

Moreover $Aut(\tau)$ acts on the domain of $f'$, $G_{\tau}$ acts on the $\X$-component of the target $\XxS$, while the parametrized $\mathbb{A}$-action acts on the $S^{2}$-component of the target $\XxS$. Hence the action of the automorphism group $\Gamma_{\tau}$, which is generated by $Aut(\tau)$ and $G_{\tau}$,  commutes with the parametrized $\mathbb{A}$-action. By construction, $PSL(2,\com)$ acts on the obstruction bundle and by restricting to  $P$ one gives the obstruction bundle the structure of a parametrized $\mathbb{A}$-bundle. Furthermore the Kuranishi map (\ref{kuramap}) is  parametrized $\mathbb{A}$-equivariant. 

To construct the entire Kuranishi structure, one chooses a partial order $\prec$ (see \cite{FO}) of combinatorial types of stable maps, constructs Kuranishi charts for moduli spaces with low order (and thus has low dimension), takes fiber products of Kuranishi structures of low-ordered moduli spaces to get Kuranishi structure on lower dimensional strata of a higher-ordered moduli space, and then extends it to moduli spaces of higher order with respect to $\prec$. By induction according to the partial order $\prec$, one gets Kuranishi structures for all moduli spaces. Thus by construction the coordinate changes of the Kuranishi structure are parametrized $\mathbb{A}$-equivariant.

For $\tau\notin\M^{freak}$, the parametrized $\mathbb{A}$-action on $V_{\tau}$ is locally free because of the same reason as Lemma \ref{paraActionLemma}. For $\tau\in \M^{freak}$, let $\check{V}_{\tau}:=\{f'|\mathfrak{forget}(f')=\mathfrak{forget}(\tau)\}$ where $\mathfrak{forget}$ forgets the map $f'$ and remembers the domain stable curve. Then $V_{\tau}\setminus\check{V}_{\tau}$ contains no element with non-constant stem component. Thus the parametrized $\mathbb{A}$-action on $V_{\tau}\setminus\check{V}_{\tau}$ is free.

Finally (\ref{vertobscond}) follows from Lemma \ref{lemvertobs}.
\end{proof}

For a Kuranishi structure $\{(V_{\alpha},E_{\alpha},\Gamma_{\alpha},\psi_{\alpha},s_{\alpha})\mid \alpha \in \frak A\}$ over $\Mbar_{0,1}(\XxS,\sigma_{0}+B,\tJ_{0},(g))$, denote 
$$
\mathfrak{A}^{1fib}:=\{\alpha\in\mathfrak{A} \ \text{s.t.}\ s_{\alpha}^{-1}(0)\cap\M^{1fib}\neq \emptyset\}
$$
$$
 V:=\bigsqcup_{\alpha\in \mathfrak{A}}V_{\alpha}\ , \ \ \ \ \ \  \check{V}^{1fib}:=\bigsqcup_{\alpha\in\mathfrak{A}^{1fib}}\check{V}_{\alpha}.
 $$
The $\mathbb{A}$-group bundle has two subbundle $P^{\com}$ and $P^{S^{1}}$ as mentioned in Example \ref{paraAction}, thus induces a parametrized $\com$-action and a parametrized $S^{1}$-action on $V$. 

\begin{cor}\label{Cequivkura}
Restriction of the parametrized $\mathbb{A}$-action to $P^{\com}$ defines a parametrized $\com$-equivariant Kuranishi structure over $\Mbar_{0,1}(\XxS,\sigma_{0}+B,\tJ_{0},(g))$. Moreover, the parametrized $\mathbb{C}$-action on $V\setminus\check{V}^{1fib}$  is locally free and has finite stabilizer everywhere. 
\end{cor}

We will  use the parametrized $\com$-action to construct parametrized equivariant multi-sections, and use the parametrized $S^{1}$-action to perturb the evaluation map.

\begin{lem}\label{free_equmultisection}
There is a parametrized $\com$-equivariant Kuranishi structure over $\Mbar_{0,1}(\XxS,\sigma_{0}+B,\tJ_{0},(g))$ whose restriction to $V\setminus \check{V}^{1fib}$ has a system of multi-sections $\ms^{free}$ such that:
\begin{enumerate}
\item They are transversal to 0;
\item They are close to the original Kuranishi map s;
\item They are parametrized equivariant under the parametrized $\mathbb{C}$-action.
\end{enumerate}
\end{lem}
\begin{proof}
Although the parametrized $\com$-action on $V\setminus\check{V}^{1fib}$ is not proper due to the non-compactness of the group $\com$, we can directly construct slices transversal to the parametrized $\mathbb{C}$-action, then push transversal multi-sections on the slices out along the orbits. 

For  $f\in V\setminus\check{V}^{1fib}$, denote its orbit under the parametrized $\com$-action by $\com_{f}$. We have shown that the parametrized $\com$-action has finite stabilizers on $V\setminus\check{V}^{1fib}$. Since there is no non-trivial finite subgroup of $\com$, the stabilizer of $f$ under the parametrized $\com$-action is trivial.

The metric on $\X$ induces a metric on $V$ and denote by $Exp: TV\to V$ the exponential map determined by the metric. For a small positive number $\epsilon>0$, a neighborhood $B_{\epsilon}(f)$ of $f$ can be identified with a neighborhood $D$ of $0\in T_{f}V$ under $Exp$. Let $\mathcal{N}_{\com_{f}}V\subset T_{f}V$ be the orthogonal complement of $T_{f}\com_{f}$. Denote $\mathcal{S}_{f}:=Exp(f,\mathcal{N}_{\com_{f}}V\cap D)$. 

Now we need to modify the Kuranishi structure constructed in Proposition \ref{constructequivkura}. 
\begin{enumerate}
\item For $\alpha\in \mathfrak{A}^{1fib}$, let $\mathfrak{B}_{\alpha}$ be the set of $f$ such that the orbit  $\com\cdot \mathcal{S}_{f}\subset V_{\alpha}$ and its closure $\overline{\com\cdot \mathcal{S}_{f}}$ in $V_{\alpha}$ intersects $V^{1fib}$. Then there exists a finite subset $$\mathfrak{B}^{1fib}\subset \cup_{\alpha\in \mathfrak{A}^{1fib}}\mathfrak{B}_{\alpha}$$ such that $\{\overline{\com\cdot \mathcal{S}_{f}}|f\in \mathfrak{B}^{1fib}\}$ covers a neighborhood of $\M^{1fib}$. This follows from the compactness of the boundary of a tubular neighborhood of $\M^{1fib}$ and the fact that every orbit $\com\cdot \mathcal{S}_{f}$ for $f\in \mathfrak{B}_{\alpha}$ passes through the boundary. We replace $V_{\alpha}$ for $\alpha\in \mathfrak{A}^{1fib}$ by the union of $\overline{\com\cdot \mathcal{S}_{f}}$ for $f\in \mathfrak{B}^{1fib}$.

\item For $\alpha\in \mathfrak{A}^{1fib}$, let $\mathfrak{C}_{\alpha}$ be the set of $f$ such that the orbit $\com\cdot \mathcal{S}_{f}\subset V_{\alpha}$ and its closure $\overline{\com\cdot \mathcal{S}_{f}}$ do not intersect $V^{1fib}$. We replace $V_{\alpha}$ for $\alpha\in \mathfrak{A}^{1fib}$ by the union of $\com\cdot \mathcal{S}_{f}$ for $f\in \cup_{\alpha\in \mathfrak{A}}\mathfrak{C}_{\alpha}$ and the Kuranishi charts constructed in Case (1). Note that the Kuranishi structure neighborhood constructed in this case does not cover $\M^{1fib}$ any more.

\item For the Kuranishi structure $\{V_{\alpha},\}$ modified as in Case (1) and Case (2), if $\alpha\in \mathfrak{A}\setminus \mathfrak{A}^{1fib}$ (i.e. $V_{\alpha}$ is disjoint from $\M^{1fib}$), shrink $V_{\alpha}$ slightly to a compact subset $\hat{V}_{\alpha}$ so that such compact subsets $\hat{V}_{\alpha}$ together with $\{V_{\alpha}|\alpha\in \mathfrak{A}\setminus \mathfrak{A}^{1fib}\}$ still cover the moduli space $\Mbar_{0,1}(\XxS,\sigma_{0}+B,\tJ_{0},(g))$. On the other side, these compact sets $\hat{V}_{\alpha}$ can be covered by finite orbits $\com\cdot \mathcal{S}_{f}$. Then for $\alpha\in \mathfrak{A}\setminus\mathfrak{A}^{1fib}$, replace $V_{\alpha}$ by the union of such orbits $\com\cdot \mathcal{S}_{f}$ which are contained in $V_{\alpha}$. \footnote{This is a standard trick for Kuranishi structure which takes advantage of the compactness of the moduli space without gluing together the Kuranishi charts $V_{\alpha}$ and worrying about compactness of the glued space. This technique was used again and again by \cite{FO} and other papers by Fukaya-Oh-Ohta-Ono.} As a result, there exist finitely many slices $\mathcal{S}_{f}$'s in $V_{\alpha}$ such that their orbits $\com\cdot \mathcal{S}_{f}$ together with $\{V_{\alpha}|\alpha\in \mathfrak{A}\setminus \mathfrak{A}^{1fib}\}$ cover the entire moduli space $\Mbar_{0,1}(\XxS,\sigma_{0}+B,\tJ_{0},(g))$. 
\end{enumerate}

Now we have modified the Kuranishi structure so that the Kuranishi charts $V_{\alpha}$ are covered by a finite set of orbits of slices. Then we can use the method of \cite[Theorem 3.11]{FO} and \cite[Theorem A1.23]{FOOObook} to construct multi-sections with the required properties. Note that on each orbit we only allow multi-sections which are given by pushing out transversal multi-sections on the slice.
\end{proof}

\begin{rmk}\label{nonextension}
From the construction of multi-sections we can see that a multi-section on a slice may be squeezed to a smaller and smaller set when the orbit goes toward the fixed locus $\check{V}^{1fib}$, thus its derivative will blow up when going to $\check{V}^{1fib}$. So the parametrized equivariant multi-sections cannot be extended to the fixed locus while maintaining transversality. A similar situation appears in \cite{FOOOtoric3}. As remarked there, when the group action has isotropy groups of positive dimension, the quotient space is neither a manifold nor an orbifold, so the method in the above proof does not work.
\end{rmk}
As a consequence of Remark \ref{nonextension}, we drop the equivariant condition near the fixed locus $\check{V}^{1fib}$, and construct a system of multi-sections for the entire Kuranishi structure as follows:

\begin{lem}\label{extendmultisection}
For any compact subset $V^{c}\subset V\setminus\check{V}^{1fib}$, the parametrized equivariant Kuranishi structure in in Corollary \ref{Cequivkura} has a system of multi-sections $\ms$ on $V$ such that 
\begin{enumerate}
\item They are transversal to 0 on $V$;
\item They are close to the original Kuranishi map s;
\item $\ms=\ms^{free}$ on $V^{c}$.
\end{enumerate}
\end{lem}
\begin{proof}
This is a direct application of the relative version of existence of transversal multi-sections \cite[Lemma 3.14]{FO}. 
\end{proof}

We remark that if $V^{c}$ is not parametrized $\com$-invariant, then $\ms$ may not be equivariant since orbits may run out of $V^{c}$. But this is enough for our purpose as one will see in the next subsection.

\subsection{Proof of Triviality: conclusion}\label{wholetriv}
In this subsection we finish the proof of the triviality property (\ref{triviality}). 

We will change the evaluation map homotopically  in a small neighborhood of $\check{V}^{1fib}$ so that the virtual cycle defined by the new evaluation map has a dimension lower than the expected virtual dimension.

Note that the parametrized $\mathbb{A}$-action on $V$ restricted to $P^{S^{1}}$ defines a parametrized $S^{1}$-action on $V$. We use this action to perturb the evaluation map. 
\begin{lem}\label{tubnbd}
Fix a metric on $\X_{(g)}$. For any $\epsilon>0$, there exists an open neighborhood $U_{V}(\check{V}^{1fib})$ of $\check{V}^{1fib}$ in $V$, and a (closed) disk bundle $\pi_{1fib}: \bar{U}_{V}(\check{V}^{1fib}) \to \check{V}^{1fib}$, where $\bar{U}_{V}(\check{V}^{1fib})$ is the closure of $U_{V}(\check{V}^{1fib})$ in $V$, such that:
\begin{enumerate}
\item \label{circleinv} For any $\check{f}\in \check{V}^{1fib}$, $\partial \pi_{1fib}^{-1}(\check{f}):=\pi_{1fib}^{-1}(\check{f})\cap \partial \bar{U}_{V}(\check{V}^{1fib})$, is an orbit of the parametrized $S^{1}$-action on $V$. The evaluation map $ev$ is constant on $\partial \pi_{1fib}^{-1}(\check{f})$.

\item \label{imagesmall} Let $v(\check{f}):=ev(\partial \pi_{1fib}^{-1}(\check{f}))$, then $ev(\pi_{1fib}^{-1}(\check{f}))\subset B_{\epsilon}(v(\check{f}))$ where $B_{\epsilon}(v(\check{f}))$ is a ball neighborhood of $v(\check{f})$ with radius $\epsilon$.
\end{enumerate}

\end{lem}
\begin{proof}
Recall that by construction of Kuranishi structure, a neighborhood of $\check{V}^{1fib}$ in $V$ can be identified with $\check{V}^{1fib}\times \tilde{D}$ via a glueing map $Glue: \check{V}^{1fib}\times \tilde{D}\to V$, where $\tilde{D}\subset \com$ is a 2-disk centered at zero parametrizing the resolution of the singular point between the stem component and the branch component.

For $\delta>0$, let  
\begin{equation*}
\begin{split}
&\mathcal{S}_{\delta}:=\{Glue(\check{f},re^{i0})|\check{f}\in \check{V}^{1fib}, 0<r<\delta\},\\
&\tilde{\mathcal{S}}_{\delta}=\{Glue(\check{f},re^{i0})|\check{f}\in \check{V}^{1fib}, 0\leq r< \delta\},\\ 
&\bar{\mathcal{S}}_{\delta}=\{Glue(\check{f},re^{i0})|\check{f}\in \check{V}^{1fib}, 0\leq r\leq \delta\},\\ 
&\partial\bar{\mathcal{S}}_{\delta}=\{Glue(\check{f},re^{i0})|\check{f}\in \check{V}^{1fib},  r=\delta\}. 
\end{split}
\end{equation*}
By making $\delta$ small enough, we may assume that the orbit of $\bar{\mathcal{S}}_{\delta}$, denoted by $S^{1}\cdot \bar{\mathcal{S}}_{\delta}$, is contained in $Glue(\check{V}^{1fib}\times \tilde{D})$. 
Then the orbit $U_{\delta}:=S^{1}\cdot \tilde{\mathcal{S}}_{\delta}$ is an open neighborhood of $\check{V}^{1fib}$ in $V$. 

We define $\pi_{1fib}:S^{1}\cdot \bar{\mathcal{S}}_{\delta}\to \check{V}^{1fib}$ by $$\pi_{1fib}(f):=\check{f},\ \ \ \ \ \ \ \text{for} \ \ f\in S^{1}\cdot \mathcal{R}_{\delta}(\check{f}),$$
 where $\mathcal{R}_{\delta}(\check{f}):=\{Glue(\check{f},re^{i0})|0\leq r\leq \delta \}$.

By construction, $\partial \pi_{1fib}^{-1}(\check{f})=\pi_{1fib}^{-1}(\check{f})\cap S^{1}\cdot \partial \bar{\mathcal{S}}_{\delta}$ is an orbit of the parametrized $S^{1}$-action on $V$. The evaluation map $ev$ is constant on such an orbit. Thus $v(\check{f}):=ev(\partial \pi_{1fib}^{-1}(\check{f}))$ is well-defined. The fiber $S^{1}\cdot \mathcal{R}_{\delta}(\check{f})$ is a closed 2-disk swept out by an interval under $S^{1}$ action. Thus  for every small enough $\delta$, we get a tubular neighborhood $U_{\delta}:=S^{1}\cdot \mathcal{S}_{\delta}$ of $\check{V}^{1fib}$ satisfying (\ref{circleinv}).

For $\check{f}\in  \check{V}^{1fib}$, consider a ball neighborhood $B_{\frac{\epsilon}{2}}(v(\check{f}))$ of $v(\check{f})$ with radius $\frac{\epsilon}{2}$. Then choose small enough $\delta_{\check{f}}$ and a small enough  open neighborhood $N(\check{f})$ of $\check{f}$ in $\check{V}^{1fib}$, so that $\bar{\mathcal{S}}_{\delta_{\check{f}}}\cap \pi_{1fib}^{-1}(N(\check{f}))\subset ev^{-1}(B_{\frac{\epsilon}{2}}(ev(\check{f})))$. By construction of $\check{V}^{1fib}$, one can always shrink it slightly and make it compact. Thus there is a finite set $\{\check{f}_{i}\}$, such that the $\cup_{\{\check{f}_{i}\}}N(\check{f}_{i})$ covers $\check{V}^{1fib}$.  Let $\delta=min\{\delta_{\check{f}_{i}}\}$, define $U_{\delta}:=S^{1}\cdot \bar{\mathcal{S}}_{\delta}$. Then $ev(\pi_{1fib}^{-1}(\check{f}))\subset B_{\frac{\epsilon}{2}}(ev(\check{f}))$, in particular $v(\check{f})=ev(\partial \pi_{1fib}^{-1}(\check{f}))\in B_{\frac{\epsilon}{2}}(ev(\check{f}))$. So $ev(\pi_{1fib}^{-1}(\check{f}))\subset B_{\epsilon}(v(\check{f}))$.  Therefore $U_{V}(\check{V}^{1fib}):=U_{\delta}$ is a tubular neighborhood of $\check{V}^{1fib}$ satisfying both (\ref{circleinv}) and (\ref{imagesmall}). 
\end{proof}

We choose $\epsilon<injrad(\X_{(g)})$, where $injrad(\X_{(g)})$ is the injective radius of $\X_{(g)}$. Consider $U_{V}(\check{V}^{1fib})$ and $\phi$ as in Lemma \ref{tubnbd}. We define $Ev:V\to\X_{(g)}$ as follows:
\begin{itemize}
\item For $\ f\in V\setminus   U_{V}(\check{V}^{1fib})$, $Ev(f):= ev(f)$;
\item For $\ f \in U_{V}(\check{V}^{1fib})$, there is a unique $\check{f}\in \check{V}^{1fib}$ such that $\ f \in \pi_{1fib}^{-1}(\check{f})$, then define $Ev(f):=  v(\check{f})$.
\end{itemize}

\begin{lem}\label{evhomotopy}
The two maps $Ev$ and $ev$ are homotopy equivalent.
\end{lem}
\begin{proof}
Since the two maps coincide on $V\setminus   U_{V}(\check{V}^{1fib})$, it is enough to show they are homotopy equivalent on $U_{V}(\check{V}^{1fib})$.

By collapsing $\partial \pi_{1fib}^{-1}(\check{f})\subset\pi_{1fib}^{-1}(\check{f})$ to a point for every $\check{f}\in \check{V}^{1fib}$, we get a sphere bundle $\pi:Sp\to V^{1fib}$. Denote the quotient map by $pr$, and the image of $\partial \bar{U}_{V}(\check{V}^{1fib})$ under the quotient map by $Z$. Then $Z$ is a section of the sphere bundle. The map $ev$ induces a map $\bar{ev}: Sp \to \X_{(g)}$ which satisfies $ev=\bar{ev}\circ pr$. Similarly $Ev$ induces a map $\bar{Ev}: Sp \to \X_{(g)}$ which satisfies $Ev=\bar{Ev}\circ pr$. 

Let $gr_{\bar{ev}}:=(\bar{ev},\pi): Sp\to \X_{(g)}\times V^{1fib}$, and $gr_{\bar{Ev}}:=(\bar{Ev},\pi):Sp\to \X_{(g)}\times V^{1fib}$, then $gr_{\bar{ev}}(Z)$ and $gr_{\bar{Ev}}(Z)$ are sections of the bundle $\X_{(g)}\times V^{1fib}\to V^{1fib}$. By Lemma \ref{tubnbd} and the choice of $\epsilon$, $gr_{\bar{ev}}(Z)$ is contained in a tubular neighborhood of $gr_{\bar{Ev}}(Z)$. Thus there is a homotopy equivalence between $gr_{\bar{ev}}$ and $gr_{\bar{Ev}}$, which gives a homotopy equivalence between $\bar{ev}$ and $\bar{Ev}$, which further determines a homotopy equivalence between $ev$ and $Ev$ on $U_{V}(\check{V}^{1fib})$. So $ev$ and $Ev$ are homotopy equivalent.
\end{proof}
Now we are ready to prove Case 3 of Proposition \ref{ModuliofTrivial}:
\begin{proof}[Proof of Proposition \ref{ModuliofTrivial}, Case 3]
We choose an open neighborhood $U$ of $\check{V}^{1fib}$ such that $U\subset U_{V}(\check{V}^{1fib})$. For this $V^{c}=V\setminus U $ we get a system of multi-sections $\ms$ as in Lemma \ref{extendmultisection}. Since $ev$ and $Ev$ are homotopy equivalent by Lemma \ref{evhomotopy}, it is enough to consider the cycle $(\ms^{-1}(0),Ev)$. By suitably refining a given triangulation, we may assume that any top dimensional simplex $ \Delta$ in the triangulation of $\ms^{-1}(0)$ satisfies one of the followings:
\begin{enumerate}
\item $\Delta\subset  V\setminus  U_{V}(\check{V}^{1fib})$;
\item  $\Delta\subset  \bar{U}_{V}(\check{V}^{1fib})$.
\end{enumerate}

For the first case, since $V\setminus  U_{V}(\check{V}^{1fib})\subset V^{c}$, by Corollary \ref{Cequivkura} and Lemma \ref{extendmultisection}, $\Delta$ is foliated by orbits of the  parametrized $\com$-action, and the evaluation map is constant on each orbit. Thus the singular simplex $(\Delta,Ev)=(\Delta,ev)$ is of dimension at most $vdim\Mbar_{0,1}(\XxS,\sigma_{0}+A,\tJ_{0},(g))-dim\mathbb{C}=vdim\Mbar_{0,1}(\XxS,\sigma_{0}+A,\tJ_{0},(g))-2$. 
 
 For the second case, by the construction of $Ev$, the singular simplex $(\Delta,Ev)$ is contained in the image of the first case, thus has dimension at most $vdim\Mbar_{0,1}(\XxS,\sigma_{0}+A,\tJ_{0},(g))-2$ as well.
 
 Therefore, in any case the pseudo manifold $(\Delta,Ev)$ has dimension strictly less than the expected virtual dimension. Thus a generic cycle representing $\iota_{*}\alpha$ does intersect with $(\ms^{-1}(0),Ev)$. So Case 3 of Proposition \ref{ModuliofTrivial} is proven. 
\end{proof}
The proof of Proposition \ref{ModuliofTrivial} and the proof of Triviality property (\ref{triviality}) are now complete.

\section{Composition property}\label{compaxiom}
In this section we prove the composition property (\ref{composition}) of Seidel representation:
$$
\mathcal{S}(a\cdot b)=\mathcal{S}(a)*\mathcal{S}(b), \ \ \ \ a,b\in \pi_{1}(Ham(\X, \omega)).
$$
Let $\alpha$, $\beta$ be two Hamiltonian loops representing $a$ and $b$ respectively. Let $\varpi=\alpha\circ \beta$ be their composition. Then $\varpi$ represents $a\cdot b$. Denote by $\E_{\alpha}$, $\E_{\beta}$ and $\E_{\varpi}$ the corresponding Hamiltonian orbifiber bundles constructed as in Section \ref{orbifiberbundleS2}. Denote by $\iota^{\alpha}:\X\to \E_{\alpha}$ the inclusion of the fiber over $0\in\com P^1$ into $\E_{\alpha}$, and $\iota^{\beta}:\X\to \E_{\beta}$ the inclusion of the fiber over $\infty\in \com P^1$ into $\E_{\beta}$. The inclusions $\iota^{\alpha}$ and $\iota^{\beta}$ also induce inclusions of the corresponding inertia orbifolds, we also denote them by $\iota^{\alpha}$ and $\iota^{\beta}$ as well when there is no confusion. Let $\{f_{i}\}$ be an additive basis of $H^{*}(I\X,\ration)$, and $\{f^{i}\}$ its dual basis with respect to the orbifold Poincar\'e pairing. We calculate
\begin{eqnarray*}
& & \mathcal{S}(a)*\mathcal{S}(b)\\
= & & \negthickspace\negthickspace\negthickspace\negthickspace\negthickspace  \sum_{\sigma_{\alpha}\in H_2^{sec}(|\E_{\alpha}|)}\negthickspace \sum_{j} \< \iota^{\alpha}_*f_j\>^{\E_{\alpha}}_{0,1,\sigma_{\alpha}}f^j\otimes q^{c_{1,\alpha}^{v}(\sigma_{\alpha})}t^{u_{\alpha}(\sigma_{\alpha})} * \negthickspace\negthickspace\sum_{\sigma_{\beta}\in H_2^{sec}(|\E_{\beta}|)} \sum_{k} \< \iota^{\beta}_*f_k\>^{\E_{\beta}}_{0,1,\sigma_{\beta}}f^k\otimes q^{c_{1,\beta}^{v}(\sigma_{\beta})}t^{u_{\beta}(\sigma_{\beta})}\\
=& &  \negthickspace\negthickspace\negthickspace\negthickspace\negthickspace\sum_{\sigma_{\alpha}\in H_2^{sec}(|\E_{\alpha}|)} \negthickspace(\sum_{j} \< \iota^{\alpha}_*f_j\>^{\E_{\alpha}}_{0,1,\sigma_{\alpha}}f^j) * \negthickspace\negthickspace\sum_{\sigma_{\beta}\in H_2^{sec}(|\E_{\beta}|)} (\sum_{k} \< \iota^{\beta}_*f_k\>^{\E_{\beta}}_{0,1,\sigma_{\beta}}f^k)\otimes q^{c_{1,\alpha}^{v}(\sigma_{\alpha})+c_{1,\beta}^{v}(\sigma_{\beta})}t^{u_{\alpha}(\sigma_{\alpha})+u_{\beta}(\sigma_{\beta})}\\
=& & \negthickspace\negthickspace\negthickspace\negthickspace\negthickspace\sum_{\sigma_{\alpha}\in H_2^{sec}(|\E_{\alpha}|)}\sum_{\sigma_{\beta}\in H_2^{sec}(|\E_{\beta}|)}  \sum_{j,k} \< \iota^{\alpha}_*f_j\>^{\E_{\alpha}}_{0,1,\sigma_{\alpha}} \< \iota^{\beta}_*f_k\>^{\E_{\beta}}_{0,1,\sigma_{\beta}} (f^j * f^k)\otimes q^{c_{1,\alpha}^{v}(\sigma_{\alpha})+c_{1,\beta}^{v}(\sigma_{\beta})}t^{u_{\alpha}(\sigma_{\alpha})+u_{\beta}(\sigma_{\beta})}\\
=& & \negthickspace\negthickspace \negthickspace\negthickspace\negthickspace\sum_{\substack{\sigma_{\alpha}\in H_2^{sec}(|\E_{\alpha}|)\\\sigma_{\beta}\in H_2^{sec}(|\E_{\beta}|)}}  \sum_{i,j,k} \< \iota^{\alpha}_*f_j\>^{\E_{\alpha}}_{0,1,\sigma_{\alpha}} \< \iota^{\beta}_*f_k\>^{\E_{\beta}}_{0,1,\sigma_{\beta}} \<f^j * f^k,f_{i}\>_{orb}f^{i}\otimes q^{c_{1,\alpha}^{v}(\sigma_{\alpha})+c_{1,\beta}^{v}(\sigma_{\beta})}t^{u_{\alpha}(\sigma_{\alpha})+u_{\beta}(\sigma_{\beta})}\\
= & & \negthickspace\negthickspace\negthickspace\negthickspace\negthickspace \sum_{\substack{\sigma_{\alpha}\in H_2^{sec}(|\E_{\alpha}|)\\\sigma_{\beta}\in H_2^{sec}(|\E_{\beta}|)\\ i,j,k}} \negthickspace\negthickspace\negthickspace \negthickspace\< \iota^{\alpha}_*f_j\>^{\E_{\alpha}}_{0,1,\sigma_{\alpha}} \< \iota^{\beta}_*f_k\>^{\E_{\beta}}_{0,1,\sigma_{\beta}} \<f^j , f^k,f_{i}\>^{\X}_{0,3,A}f^{i}\otimes q^{c_{1,\alpha}^{v}(\sigma_{\alpha})+c_{1,\beta}^{v}(\sigma_{\beta})+c_{1}(A)}t^{u_{\alpha}(\sigma_{\alpha})+u_{\beta}(\sigma_{\beta})+\omega(A)}.
\end{eqnarray*}

In view of the definition of $\mathcal{S}(a\cdot b)$, (\ref{composition}) holds true if for $\sigma_{\alpha}\in H_2^{sec}(|\E_{\alpha}|,\inte)$ , $\sigma_{\beta}\in H_2^{sec}(|\E_{\beta}|,\inte)$ and $A\in H_{2}(|\X|,\inte)$ such that $c_{1,\alpha}^{v}(\sigma_{\alpha})+c_{1,\beta}^{v}(\sigma_{\beta})+c_{1}(A)=c_{1,\varpi}^{v}(\sigma)$ and $u_{\alpha}(\sigma_{\alpha})+u_{\beta}(\sigma_{\beta})+\omega(A)=u_{\varpi}(\sigma)$, the following holds for any $i$:
$$
 \sum_{\sigma_{\alpha},\sigma_{\beta}} \sum_{j,k} \< \iota^{\alpha}_*f_j\>^{\E_{\alpha}}_{0,1,\sigma_{\alpha}} \< \iota^{\beta}_*f_k\>^{\E_{\beta}}_{0,1,\sigma_{\beta}} \<f^j , f^k,f_{i}\>^{\X}_{0,3,A}=\< \iota_*f_i\>^{\E_{\varpi}}_{0,1,\sigma}.
$$
We may reformulate this as follows. By considering the connected sum of the underlying topological fiber bundles, we have a connected sum operation on $H_2^{sec}(|\E_{\alpha}|,\inte)$ , $ H_2^{sec}(|\E_{\beta}|,\inte)$ as in \cite[Page 435]{MS}, which is denoted by
$$H_2^{sec}(|\E_{\alpha}|,\inte)\times H_2^{sec}(|\E_{\beta}|,\inte)\to H_2^{sec}(|\E_{\varpi}|,\inte),\,\,\, (\sigma_{\alpha},\sigma_{\beta})\mapsto\sigma_{\alpha}\sharp\sigma_{\beta} .$$
The operation ``$\sharp$'' satisfies:
$$c_{1,\varpi}^{v}(\sigma_{\alpha}\sharp\sigma_{\beta} )=c_{1,\alpha}^{v}(\sigma_{\alpha})+c_{1,\beta}^{v}(\sigma_{\beta}),\ \ \ \ \ \omega_{\varpi}(\sigma_{\alpha}\sharp\sigma_{\beta} )=\omega_{\alpha}(\sigma_{\alpha})+\omega_{\beta}(\sigma_{\beta}).$$

Let $\iota^{\varpi}:\X\to \E_{\varpi}$ be the inclusion of the fiber over $\infty\in \com P^1$ into $\E_{\varpi}$. By the above discussion, (\ref{composition}) is equivalent to the following: 
\begin{prop}\label{compGW} For $\sigma\in H_2^{sec}(|\E_{\varpi}|,\inte)$, $f\in H^{*}(\X_{(g)},\ration)\subset H^{*}(I\X,\ration)$,
 \begin{equation}\label{compGWid}
 \sum_{\sigma_{\alpha}\sharp\sigma_{\beta}+\iota^{\varpi}_*A=\sigma} \sum_{j,k} \< \iota^{\alpha}_*f_j\>^{\E_{\alpha}}_{0,1,\sigma_{\alpha}} \< \iota^{\beta}_*f_k\>^{\E_{\beta}}_{0,1,\sigma_{\beta}} \<f^j , f^k,f\>^{\X}_{0,3,A}=\< \iota^{\varpi}_*f\>^{\E_{\varpi}}_{0,1,\sigma}
\end{equation}
\end{prop}
The proof of Proposition \ref{compGW} will occupy the remainder of this section.

\subsection{Degeneration of Hamiltonian Orbifiber Bundles}
To prove Proposition \ref{compGW}, we need to put $\E_{\alpha},\E_{\beta}$ and $\E_{\varpi}$ into a bigger orbifold, and consider the moduli space of $J$-holomophic orbifold morphisms into this big orbifold. We do this by modifying the construction of fibration given in \cite[Section 2.3.2]{M}. The main differences are the following: 
\begin{enumerate}
\item Similar to Section \ref{orbifiberbundlegeneral}, the fibers are orbifolds;
\item The total space is a closed orbifold, and the base is a manifold. 
\end{enumerate}
While the first modification is obviously necessary, the second one might be avoided if one is willing to work with Gromov-Witten theory for open manifolds/orbifolds. Since treatments of foundations of Gromov-Witten theory for non-compact orbifolds are not available in literatures at the moment, we do not use that approach.  

 Let $D$ be the open unit disks in $\com$. Define 
$$\mathcal{T}:=\{([z_{1},z_{2}],[w_{1},w_{2}],[v_{1},v_{2}])\in \com P^{1}\times \com P^{1}\times\com P^{1}|z_{1}w_{2}v_{2}=z_{2} v_{1}w_{1}\}.$$
This is the blow-up of $\com P^{1}\times \com P^{1}$ at $([1,0],[1,0])$ and $([0,1],[0,1])$. Projection to the first component defines a Lefschetz  fibration whose fibers over $[1,0]$ and $[0,1]$ are two copies of $\com P^{1}$ glued together along $[0,1]$ and $[1,0]$. The fibers over other points are $\com P^{1}$.
The following three open sets cover $\mathcal{T}$.
$$B_{a}=\{([z_{1},z_{2}],[1,w],[v,1])\in \mathcal{T}||w|<1,|v|<1\}=\{([z_{1},z_{2}],w,v)\in \com P^{1}\times D\times D|z_{1}w=z_{2} v\};$$
$$B_{b}=\{([z_{1},z_{2}],[w,1],[1,v])\in \mathcal{T}||w|<1,|v|<1\}=\{([z_{1},z_{2}],w,v)\in \com P^{1}\times D\times D|z_{1}v=z_{2}w\};$$
$$B_{0}
=\{([z_{1},z_{2}],[w_{1},w_{2}],[v_{1},v_{2}])\negthickspace \in \negthickspace \com P^{1}\times \com P^{1}\times \com P^{1}|z_{1}w_{2}v_{2}\negthickspace =\negthickspace z_{2} v_{1}w_{1}, (w_{2},\negthickspace v_{1})\negthickspace\neq \negthickspace (0,0), (w_{1},\negthickspace v_{2})\negthickspace\neq\negthickspace (0,0)\}.$$
Then
\begin{eqnarray*}
B_{a}\cap B_{0} & & =\{([z_{1},z_{2}],[1,w],[v,1])\in \com P^{1}\times D\times D|0<|w|<1,0<|v|<1, z_{1}w=z_{2}v\},\\
B_{b}\cap B_{0} & & =\{([z_{1},z_{2}],[w,1],[1,v])\in \com P^{1}\times D\times D|0<|w|<1,0<|v|<1, z_{1}v=z_{2}w\}.
\end{eqnarray*}
The projections from $B_{a}\cap B_{0}$ and $B_{b}\cap B_{0}$ to their first components determine two (non-trivial) annulus bundles. Let $D_{0}:=\{[1,z]\in \com P^{1}||z|\leq 1\}$ and $D_{\infty}:=\{[z,1]\in \com P^{1}||z|\leq 1\}$, then $\com P^{1}=\bar{D}_{0}\cup_{S^{1}} \bar{D}_{\infty}$. Trivialize the two annulus bundles over $D_{0}$ and $D_{\infty}$ by:
$$\psi_{a0,0}:B_{a}\cap B_{0}|_{\bar{D}_{0}}\to \bar{D}_{0}\times S^{1}\times (0,1),\ \psi_{a0,0}([1,z],[1,ze^{s+it}],[e^{s+it},1]):=([1,z],e^{it},e^{s}),$$
$$\psi_{a0,\infty}:B_{a}\cap B_{0}|_{\bar{D}_{\infty}}\to \bar{D}_{\infty}\times S^{1}\times (0,1),\ \psi_{a0,\infty}([z,1],[1,e^{s+it}],[ze^{s+it},1]):=([z,1],e^{it},e^{s});$$
$$\psi_{b0,0}:B_{b}\cap B_{0}|_{\bar{D}_{0}}\to \bar{D}_{0}\times S^{1}\times (0,1),\ \psi_{b0,0}([1,z],[e^{s+it},1],[1,ze^{s+it}]):=([1,z],e^{it},e^{s}),$$
$$\psi_{b0,\infty}:B_{b}\cap B_{0}|_{\bar{D}_{\infty}}\to \bar{D}_{\infty})\times S^{1}\times (0,1,\ \psi_{b0,\infty}([z,1],[ze^{s+it},1],[1,e^{s+it}]):=([z,1],e^{it},e^{s}).$$
Post-compose $\psi_{a0,0}$ and $\psi_{a0,\infty}$ with the projection to the third component $(0,1)$, we get a continuous function $l_{a}:B_{a}\cap B_{0}\to (0,1)$. Note that the two maps are consistent on the boundaries because $|z|=1$ for $z\in \partial \bar{D}_{0}$ or $\partial \bar{D}_{\infty}$. Similarly there is a map $l_{b}:B_{b}\cap B_{0}\to (0,1)$.

To make the glueing of groupoids simple, we need to shrink $B_{a}$, $B_{0}$ and $B_{b}$ slightly. Denote 
$$\check{B}_{a}=B_{a}\setminus l_{a}^{-1}( [\frac{1}{2},1)); $$
$$\check{B}_{b}=B_{b}\setminus l_{b}^{-1}([\frac{1}{2},1)); $$
$$\check{B}_{0}=B_{0}\setminus(l_{a}^{-1}((0,\frac{1}{N}])\cup l_{b}^{-1}((0,\frac{1}{N}])). $$

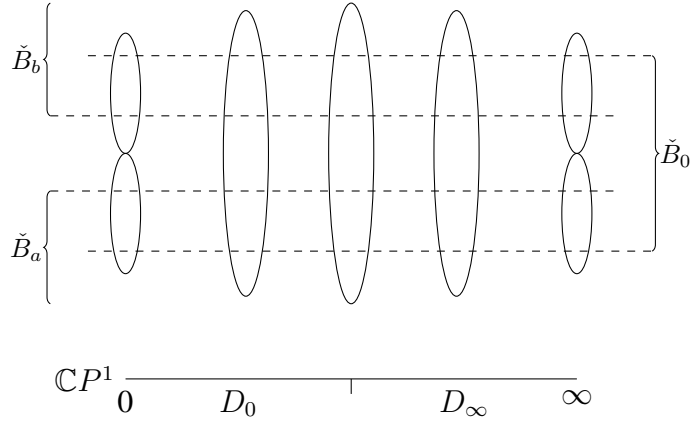
\begin{figure}[htb!]\label{fig_decomposition}
\centering
\begin{tikzpicture}
 \draw (0,0) node [below]{0}  -- (6,0) node [below]{$\infty$};
 \node at (0,0) [left] {$\com P^{1}$};
 \draw (3,0) -- (3,-0.2);
 \node at (1.5,0) [below] {$D_{0}$};
 \node at (4.5,0) [below] {$D_{\infty}$};
\draw (3,3) ellipse (0.3  and 2);
\draw (1.6,3) ellipse (0.3  and 1.9);
\draw (4.4,3) ellipse (0.3  and 1.9);
\draw (6,2.2) ellipse (0.2  and 0.8);
\draw (0,2.2) ellipse (0.2  and 0.8);
\draw (0,3.8) ellipse (0.2  and 0.8);
\draw (6,3.8) ellipse (0.2  and 0.8);
 \draw [dashed] (-0.5,1.7)  -- (7,1.7);
 \draw [dashed] (-1,3.5)  -- (6.5,3.5);
 \draw [dashed] (-0.5,4.3)  -- (7,4.3);
 \draw [dashed] (-1,2.5)  -- (6.5,2.5);
 \draw [decorate,decoration={brace,mirror}] (7,1.7) -- (7,4.3)node [black,midway,xshift=9pt] {\footnotesize $\check{B}_{0}$};
 \draw [decorate,decoration={brace,mirror}] (-1,2.5) -- (-1,1)node [black,midway,xshift=-9pt] {\footnotesize $\check{B}_{a}$};
 \draw [decorate,decoration={brace}] (-1,3.5) -- (-1,5)node [black,midway,xshift=-9pt] {\footnotesize $\check{B}_{b}$};
\end{tikzpicture} 
\caption{Decomposition of $\mathcal{T}$.}
\end{figure}

Thus $\psi_{a0,0}$ restricts to a diffeomorphism from $\check{B}_{a}\cap \check{B}_{0}|_{\bar{D}_{0}}$ to $\bar{D}_{0}\times S^{1}\times (\frac{1}{N},\frac{1}{2})$, $\psi_{a0,\infty}$ restricts to a diffeomorphism from $\check{B}_{a}\cap \check{B}_{0}|_{\bar{D}_{\infty}}$ to $\bar{D}_{\infty}\times S^{1}\times (\frac{1}{N},\frac{1}{2})$, $\psi_{b0,0}$ restricts to a diffeomorphism from $\check{B}_{b}\cap \check{B}_{0}|_{\bar{D}_{0}}$ to $\bar{D}_{0}\times S^{1}\times (\frac{1}{N},\frac{1}{2})$, and $\psi_{b0,\infty}$ restricts to a diffeomorphism from $\check{B}_{b}\cap \check{B}_{0}|_{\bar{D}_{\infty}}$ to $\bar{D}_{\infty}\times S^{1}\times (\frac{1}{N},\frac{1}{2})$.

Let $\vartheta:[0,1]\to [0,1]$ be a smooth function such that for a small $\delta>0$, $\vartheta(r)=1$ when $r>1-\delta$, $\vartheta(r)=0$ when $r<1-2\delta$. Let $\gamma_{\alpha}, \gamma_{-\beta}$ be two smooth Hamiltonian loops representing $\alpha,\beta^{-1}\in \pi_{1}(Ham(\X,\omega))$ respectively. Recall that there are stack maps $\tilde{\gamma}_{\alpha}, \tilde{\gamma}_{-\beta}:S^{1}\times \X\to\X$ associated to $\gamma_{\alpha}, \gamma_{-\beta}$. Denote $\pi_{S^{1}}:\bar{D}_{0}\times S^{1}\times (\frac{1}{N},\frac{1}{2})\to S^{1}$ the projection to the $S^{1}$ component. Define $\pi^{\vartheta}_{S^{1}}:\bar{D}_{\infty}\times S^{1}\times (\frac{1}{N},\frac{1}{2})\to S^1$ by $\pi^{\vartheta}_{S^{1}}([re^{i\theta},1],e^{it},e^{s})= e^{i(t+\vartheta(r)\theta)}$. Then the following two stack maps:
$$
(\psi^{-1}_{a0,0})\times (\tilde{\gamma}_{\alpha}\circ (\pi_{S^{1}}, Id_{\X})): \bar{D}_{0}\times S^{1}\times (\frac{1}{N},\frac{1}{2})\times \X\to \check{B}_{0}\times \X
$$
$$
(\psi^{-1}_{a0,\infty})\times (\tilde{\gamma}_{\alpha}\circ (\pi^{\vartheta}_{S^{1}}, Id_{\X})): \bar{D}_{\infty}\times S^{1}\times (\frac{1}{N},\frac{1}{2})\times \X \to \check{B}_{0}\times \X
$$
differ by a 2-morphism on their overlap. Thus they can be glued into an embedding $\check{B}_{0}\cap \check{B}_{a} \times \X\to\check{B}_{0}\times \X$. Together with the trivial embedding $\check{B}_{0}\cap \check{B}_{a} \times \X\to\check{B}_{a}\times \X$ we can glue $\check{B}_{a}\times \X$ with $\check{B}_{0}\times \X$. Similarly $\check{B}_{b}\times \X$ can be glued to $\check{B}_{0}\times \X$ . We denote the resulting stack as $\Y$. From the construction there is an obvious projection $\pmb{\pi}:\Y\to \mathcal{T}$.

Let $U_{S^{1}}$ be the cover groupoid given by an atlas $\{U'_{\tilde{\tau}}\}$ of $S^{1}$. If $\gamma_{\alpha}$ can be realized by a groupoid morphism $\gamma^{\alpha}: U_{S^{1}}\times \gG_{\X}\to \gG_{\X}$ such that $\forall g\in Mor(U_{S^{1}})$, the restriction $\gamma^{\alpha}|_{g}^{s}:=(\gamma^{\alpha}_{0}|_{s(g)},\gamma^{\alpha}_{1}|_{g}):\gG_{X}\to \gG'_{X}$ and $\gamma^{\alpha}|_{g}^{t}:=(\gamma^{\alpha}_{0}|_{t(g)},\gamma^{\alpha}_{1}|_{g}):\gG_{X}\to \gG_{X}$  are groupoid isomorphisms, and similarly $\gamma_{-\beta}$ can be realized by a groupoid morphism $\gamma^{-\beta}: U_{S^{1}}\times \gG_{\X}\to \gG_{\X}$ with the same property, then the above glueing of stacks can be realized by glueing of groupoids as follows. This is the case for all circle actions of toric orbifolds, thus will be useful in \cite{TW}.

 We introduce the following notation:
$$\bar{D}_{\infty}\times_{\vartheta} U'_{\tilde{\tau}}:=\{([re^{i\theta},1],e^{i(t-\vartheta(r)\theta)})\in \bar{D}_{\infty}\times S^{1}| e^{it}\in U'_{\tilde{\tau}} \}.$$
From the definition, we have
\begin{equation}\label{boundaryconsist}
\begin{aligned}
\psi_{a0,\infty}^{-1}(\partial\bar{D}_{\infty}\times_{\vartheta} U'_{\tilde{\tau}}\times(0,1)) & =\psi_{a0,\infty}^{-1}(\{([e^{i\theta},1],e^{i(t-\theta)})\in \bar{D}_{\infty}\times S^{1}| e^{it}\in U'_{\tilde{\tau}} \}\times (0,1))\\
& =\{([e^{i\theta},1],[1,e^{s+i(t-\theta)}],[e^{s+i(t-\theta+\theta}),1])|e^{it}\in U'_{\tilde{\tau}} \}\\
& =\{([1,e^{-i\theta}],[1,e^{s+i(t-\theta)}],[e^{s+it},1])|e^{it}\in U'_{\tilde{\tau}} \}\\
& = \psi_{a0,0}^{-1}(\partial\bar{D}_{0}\times U'_{\tilde{\tau}}\times(0,1)).
\end{aligned}
\end{equation}
Similarly $$\psi_{b0,\infty}^{-1}(\partial\bar{D}_{\infty}\times_{\vartheta} U'_{\tilde{\tau}}\times(0,1))=\psi_{b0,0}^{-1}(\partial\bar{D}_{0}\times U'_{\tilde{\tau}}\times(0,1)).$$

\begin{lem}\label{Tatlas}
There exists an atlas $\{U_{\tau}\}$ of $\mathcal{T}$, such that
\begin{enumerate}
\item $\{U_{\tau}\cap \check{B}_{a}\cap \check{B}_{0}|_{\bar{D}_{0}}\}$ is an atlas of $\check{B}_{a}\cap \check{B}_{0}|_{\bar{D}_{0}}$, which coincides with $\{\bar{D}_{0}\times U'_{\tilde{\tau}}\times(\frac{1}{N},\frac{1}{2})\}$ under $\psi_{a0,0}$.
\item $\{U_{\tau}\cap \check{B}_{a}\cap \check{B}_{0}|_{\bar{D}_{\infty}}\}$ is an atlas of $\check{B}_{a}\cap \check{B}_{0}|_{\bar{D}_{\infty}}$, which coincides with $\{\bar{D}_{\infty}\times_{\vartheta} U'_{\tilde{\tau}}\times(\frac{1}{N},\frac{1}{2})\}$ under $\psi_{a0,\infty}$.
\item $\{U_{\tau}\cap \check{B}_{b}\cap \check{B}_{0}|_{\bar{D}_{0}}\}$ is an atlas of $\check{B}_{b}\cap \check{B}_{0}|_{\bar{D}_{0}}$, which coincides with $\{ \bar{D}_{0}\times U'_{\tilde{\tau}}\times(\frac{1}{N},\frac{1}{2})\}$ under $\psi_{b0,0}$.
\item $\{U_{\tau}\cap \check{B}_{b}\cap \check{B}_{0}|_{\bar{D}_{\infty}}\}$ is an atlas of $\check{B}_{b}\cap \check{B}_{0}|_{\bar{D}_{\infty}}$, which coincides with $\{ \bar{D}_{\infty}\times_{\vartheta} U'_{\tilde{\tau}}\times(\frac{1}{N},\frac{1}{2})\}$ under $\psi_{b0,\infty}$.
\end{enumerate}
\end{lem}

\begin{proof}
For each open set $U'_{\tilde{\tau}}=\{e^{i\theta}|\theta\in(c,d)\}$, by (\ref{boundaryconsist})
$$U_{\tau}:=\psi_{a0,0}^{-1}(\bar{D}_{0}\times U'_{\tilde{\tau}}\times(\frac{1}{N+1},1))\cup \psi_{a0,\infty}^{-1}(\bar{D}_{\infty}\times_{\vartheta} U'_{\tilde{\tau}}\times(\frac{1}{N+1},1))$$
defines an open subset of $\mathcal{T}$. Similarly, we define 
$$U_{\tau'}:=\psi_{b0,0}^{-1}(\bar{D}_{0}\times U'_{\tilde{\tau}}\times(\frac{1}{N+1},1))\cup \psi_{b0,\infty}^{-1}(\bar{D}_{\infty}\times_{\vartheta} U'_{\tilde{\tau}}\times(\frac{1}{N+1},1)).$$
The union of the two kinds of open set covers $\check{B}_{a}\cap \check{B}_{0}$ and $\check{B}_{b} \cap \check{B}_{0}$, and they satisfy the conditions (1)--(4). It is easy to see that there are three open sets such that
\begin{itemize}
\item they are disjoint from $\check{B}_{a}\cap \check{B}_{0}$ and $\check{B}_{b} \cap \check{B}_{0}$;
\item these three open sets together with the collection of open sets constructed above form an atlas of $\mathcal{T}$.
\end{itemize}
The proof of the lemma is thus complete.
\end{proof}

Denote by $\gG_{\mathcal{T}}$ the cover groupoid of $\mathcal{T}$ determined by the above atlas. The maps $\psi_{a0,0},\psi_{a0,\infty},\psi_{b0,0}$, and $\psi_{b0,\infty}$ induce four groupoid isomorphisms $\pmb{\psi}_{a0,0}=(\psi_{a0,0}^{0}$, $\psi_{a0,0}^{1}), \pmb{\psi}_{a0,\infty}=(\psi_{a0,\infty}^{0}$, $\psi_{a0,\infty}^{1}),\pmb{\psi}_{b0,0}=(\psi_{b0,0}^{0},\psi_{b0,0}^{1})$, and $\pmb{\psi}_{b0,\infty}=(\psi_{b0,\infty}^{0},\psi_{b0,\infty}^{1})$.

Now define $$\gG_{\Y}:=(\gG_{\mathcal{T}}|_{\check{B}_a}\times\gG_{\X})\sqcup(\gG_{\mathcal{T}}|_{\check{B}_0}\times \gG_{\X})\sqcup (\gG_{\mathcal{T}}|_{\check{B}_b}\times \gG_{\X})/\sim,$$ where $\sim$ is given by the glueing maps 
$$\pmb{gl}_{a0}=(gl^{0}_{a0},gl^{1}_{a0}):\gG_{\mathcal{T}}|_{\check{B}_a\cap\check{B}_{0}} \times \gG_{\X}\subset \gG_{\mathcal{T}}|_{\check{B}_a}\times\gG_{\X}\to \gG_{\mathcal{T}}|_{\check{B}_a\cap\check{B}_{0}} \times \gG_{\X}\subset \gG_{\mathcal{T}}|_{\check{B}_0}\times\gG_{\X},$$
$$\pmb{gl}_{b0}=(gl^{0}_{b0},gl^{1}_{b0}):\gG_{\mathcal{T}}|_{\check{B}_b\cap\check{B}_{0}} \times \gG_{\X}\subset \gG_{\mathcal{T}}|_{\check{B}_b}\times\gG_{\X}\to \gG_{\mathcal{T}}|_{\check{B}_b\cap\check{B}_{0}} \times \gG_{\X}\subset \gG_{\mathcal{T}}|_{\check{B}_0}\times\gG_{\X},$$
defined below
:
\begin{defn}\label{glueY} The glueing map $\pmb{gl}_{a0}$ is defined by:
\begin{enumerate}
\item  Restricted to $\gG_{\mathcal{T}}|_{\check{B}_a\cap\check{B}_{0}|\bar{D}_{0}} \times \gG_{\X}$:
\begin{itemize}
\item[Ob:]For $\xi \in Ob(\gG_{\mathcal{T}})$ s.t. $\psi_{a0,0}^{0}(\xi)=([1,z],e^{it},e^{s})$ and $x \in Ob(\gG_{\X})$, set $gl^{0}_{a0}(\xi,x):= (\xi,x')$ where  $x'=\gamma_{0}^{\alpha} (t, x)$.
\item[Mor:]
For $\eta \in Mor(\gG_{\mathcal{T}})$ s.t. $\psi_{a0,0}^{1}(\xi)=([1,z],e^{i\hat{t}},e^{s})$ and $y \in Mor(\gG_{\X})$, set $gl^{1}_{a0}(\eta,y):= (\eta,y')$ where  $y'=\gamma_{1}^{\alpha} (\hat{t},y)$.
\end{itemize}

\item  Restricted to $\gG_{\mathcal{T}}|_{\check{B}_a\cap\check{B}_{0}|\bar{D}_{\infty}} \times \gG_{\X}$:
\begin{itemize}
\item[Ob:]For $\xi \in Ob(\gG_{\mathcal{T}})$ s.t. $\psi_{a0,\infty}^{0}(\xi)=([re^{i\theta},1],e^{it},e^{s})$ and $x \in Ob(\gG_{\X})$, set $gl^{0}_{a0}(\xi,x):= (\xi,x')$ where  $x'=\gamma_{0}^{\alpha} (t+\vartheta(r)\theta, x)$.
\item[Mor:]
For $\eta \in Mor(\gG_{\mathcal{T}})$ s.t. $\psi_{a0,\infty}^{1}(\xi)=([re^{i\theta},1],e^{i\hat{t}},e^{s})$ and $y \in Mor(\gG_{\X})$, set $gl^{1}_{a0}(\eta,y):= (\eta,y')$ where  $y'=\gamma_{1}^{\alpha} (\hat{t}+\vartheta(r)\theta,y)$.
\end{itemize}
\end{enumerate}

The glueing map $\pmb{gl}_{b0}$ is defined by:
\begin{enumerate}
\item  Restricted to $\gG_{\mathcal{T}}|_{\check{B}_b\cap\check{B}_{0}|\bar{D}_{0}} \times \gG_{\X}$:
\begin{itemize}
\item[Ob:]For $\xi \in Ob(\gG_{\mathcal{T}})$ s.t. $\psi_{b0,0}^{0}(\xi)=([1,z],e^{it},e^{s})$ and $x \in Ob(\gG_{\X})$, set $gl^{0}_{b0}(\xi,x):= (\xi,x')$ where  $x'=\gamma_{0}^{-\beta} (t, x)$.
\item[Mor:]
For $\eta \in Mor(\gG_{\mathcal{T}})$ s.t. $\psi_{b0,0}^{1}(\xi)=([1,z],e^{i\hat{t}},e^{s})$ and $y \in Mor(\gG_{\X})$, set $gl^{1}_{b0}(\eta,y):= (\eta,y')$ where  $y'=\gamma_{1}^{-\beta} (\hat{t},y)$.
\end{itemize}

\item  Restricted to $\gG_{\mathcal{T}}|_{\check{B}_b\cap\check{B}_{0}|\bar{D}_{\infty}} \times \gG_{\X}$:
\begin{itemize}
\item[Ob:]For $\xi \in Ob(\gG_{\mathcal{T}})$ s.t. $\psi_{b0,\infty}^{0}(\xi)=([re^{i\theta},1],e^{it},e^{s})$ and $x \in Ob(\gG_{\X})$, set $gl^{0}_{b0}(\xi,x):= (\xi,x')$ where  $x'=\gamma_{0}^{-\beta} (t+\vartheta(r)\theta, x)$.
\item[Mor:]
For $\eta \in Mor(\gG_{\mathcal{T}})$ s.t. $\psi_{b0,\infty}^{1}(\xi)=([re^{i\theta},1],e^{i\hat{t}},e^{s})$ and $y \in Mor(\gG_{\X})$, set $gl^{1}_{b0}(\eta,y):= (\eta,y')$ where  $y'=\gamma_{1}^{-\beta} (\hat{t}+\vartheta(r)\theta,y)$.
\end{itemize}
\end{enumerate}

\end{defn}

By the construction, one can check that the above definition is consistent on overlaps. Thus the two glueing maps $\pmb{gl}_{a0}$ and $\pmb{gl}_{b0}$ are well-defined.

Now we obtain from $\gG_{\Y}$ an orbifiber bundle $\pmb{\pi}:\Y\to\mathcal{T}$ with fiber $\X$. Denote by $\kappa: \mathcal{T}\to \com P^{1}$ the projection to the first component. The composition $pr:=\kappa\circ\pmb{\pi}:\Y\to \com P^{1}$ is again an orbifiber bundle whose fiber is $\E_{\varpi}$ away from $[0,1]$ and $[1,0]$, while over $[0,1]$ and $[1,0]$ we get two singular fibers which are unions of $\E_{\alpha}$ and $\E_{\beta}$ meeting along a fiber $\X$. By construction, $\Y$ carries a doubly fibered almost complex structure $\mathbb{J}$: restricted to a fiber of $pr$, it is an almost complex structure of $\E_{\varpi}$ (or $\E_{\alpha}$ and $\E_{\beta}$); restricted to each fiber $\X$, it is the almost complex structure $J$ of $\X$. Let $j$ be the complex structure on $\mathcal{T}$, then $\pmb{\pi}$ is $j$-$J$-holomorphic.
 
Recall the notation from Definition \ref{couplingcl}, for $i=\alpha, \beta, \varpi$, let $\Omega^{i}_{0}$ be the coupling form of $\E_{i}$.
\begin{lem}
There is a closed 2-form $ \Omega_{0}$ on $\Y$, such that 
\begin{enumerate}
\item  for any $[z_{1},z_{2}]\neq [0,1]$ or $[1,0] \in \com P^{1}$, there is a diffeomorphism $\Phi: pr^{-1}([z_{1},z_{2}]) \to \E_{\varpi}$ such that $\Omega_{0}|_{pr^{-1}([z_{1},z_{2}])}=\Phi^{*}\Omega_{0}^{\varpi}$;
\item for $[z_{1},z_{2}]= [0,1]$ or $[1,0] \in \com P^{1}$, there are two diffeomorphisms $\Psi_{i}: pr^{-1}([z_{1},z_{2}])|_{\com P^{1}_{i}} \to \E_{i}$, $i=\alpha$ or $\beta$, such that $ \Omega_{0}|_{pr^{-1}([z_{1},z_{2}])|_{\com P^{1}_{i}}}=\Psi^{*}\Omega_{0}^{i}$. Here $|_{\com P^{1}_{\alpha}}$ and $|_{\com P^{1}_{\beta}}$ are used to denote the restriction over the two copies of $\com P^{1}$ of the singular fibers of $\kappa:\mathcal{T}\to \com P^{1}$.
\end{enumerate}
\end{lem}
\begin{proof} Let $\gamma_{\alpha}$ and $\gamma_{\beta}$ be two Hamiltonian loops representing $\alpha$ and $\beta$ respectively. Denote by $H^{\alpha}$ and $H^{\beta}$ the Hamiltonian functions of $\gamma_{\alpha}$ and $\gamma_{\beta}$. Let 
\begin{eqnarray*}
G_{0}^{-} & & : \bar{D}_{0} \times S^{1} \times (\frac{1}{N},\frac{1}{2})\times \X \to \real,\ \ G_{0}^{-}([1,z],e^{it},e^{s}, x):=\rho(e^{s})H^{\alpha}_{t}\circ \gamma^{t}_{\alpha}(x),\\
G_{\infty}^{-} & & : \bar{D}_{\infty} \times S^{1} \times (\frac{1}{N},\frac{1}{2})\times \X \to \real, \ \  G_{\infty}^{-}([re^{i\theta},1],e^{it},e^{s}, x) :=\rho(e^{s})H^{\alpha}_{t+\vartheta(r)\theta}\circ \gamma_{\alpha}^{t+\vartheta(r)\theta} (x),\\
G_{0}^{+} & & : \bar{D}_{0} \times S^{1} \times (\frac{1}{N},\frac{1}{2})\times \X \to \real, \ \ 
G_{0}^{+} ([1,z],e^{it},e^{s}, x) :=\rho(e^{s})H^{\beta}_{t}\circ \gamma_{\beta}^{t}(x),\\
G_{\infty}^{+} & & : \bar{D}_{\infty} \times S^{1} \times (\frac{1}{N},\frac{1}{2})\times \X \to \real, \ \  G_{\infty}^{+}([re^{i\theta},1],e^{it},e^{s}, x) :=\rho(e^{s})H^{\beta}_{t+\vartheta(r)\theta}\circ \gamma_{\beta}^{t+\vartheta(r)\theta} (x),
\end{eqnarray*}
where $\rho :D_{+}(1+\delta)\to [0,1]$ is the same cutoff function as in Remark \ref{couplingformcutoff} of Section \ref{orbifiberbundleS2}, and $\gamma^t_\bullet(-)=\gamma_\bullet(t,-)$.  

On $\bar{D}_{0} \times S^{1} \times (\frac{1}{N},\frac{1}{2})\times \X$, define 
$$
\omega^{a,0}=\omega-d' G_{0}^{-}\wedge dt -\partial_{s} G_{0}^{-} ds \wedge dt.
$$
 On $\bar{D}_{\infty} \times S^{1} \times (\frac{1}{N},\frac{1}{2})\times \X$, define
$$
\omega^{a,\infty}=\omega-d' G_{\infty}^{-}\wedge dt -\partial_{s} G_{\infty}^{-} ds \wedge dt.
$$
It is a direct computation to check that $\omega^{a,0}$ and $\omega^{a,\infty}$ together define a closed 2-form $\omega^{a}$ on $(\check{B}_{a}\cap \check{B}_{0})\times \X$. Moreover, when $|z|<1-2\delta$, $\omega^{a}=\omega$. Thus $\omega^{a}$ can be extended to $\check{B}_{a} \times \X$ by defining it as $\omega$ on $\check{B}_{a}\setminus \check{B}_{0}$. We still denote this 2-form by $\omega^{a}$.

Similarly using $G_{0}^{+}$ and $G_{\infty}^{+}$, we can define a closed 2-form $\omega^{b}$ on $\check{B}_{b}\times \X$. On $\check{B}_{0}\times \X$, let $\omega^{0}:=\omega$.

Then as in Section \ref{orbifiberbundleS2}, the three 2-forms $\omega^{a}$, $\omega^{b}$ and $\omega^{0}$ can be glued into a closed 2-form on $\Y$ which satisfies the required conditions.
\end{proof}
On the other hand it is an exercise to show:
\begin{lem}
There is a closed 2-form $ \omega_{\mathcal{T}}$ on $\mathcal{T}$, such that 
\begin{enumerate}
\item  for any $[z_{1},z_{2}]\neq [0,1]$ or $[1,0] \in \com P^{1}$, $\omega_{\mathcal{T}}|_{\kappa^{-1}([z_{1},z_{2}])}=\omega_{S^{2}}$;
\item for $[z_{1},z_{2}]= [0,1]$ or $[1,0] \in \com P^{1}$, $\omega_{\mathcal{T}}|_{\kappa^{-1}([z_{1},z_{2}])|\com P^{1}_{i}}=\omega_{S^{2}}$, $i=\alpha$ or $\beta$.
\end{enumerate}
\end{lem}

Thus we have a symplectic form on $\Y$ defined by $\Omega^{\Y}_{c,c'}=\Omega_{0}+c\pmb{\pi}^{*} \omega_{\mathcal{T}}+c' pr^{*} \omega_{\com P^{1}}$ for $c$ as in Section \ref{orbifiberbundleS2}, and $c'>0$ large enough. A fiber of $pr:\Y\to \com P^{1}$ together with the restriction of $\Omega^{\Y}_{c,c'}$ to the fiber is symplectomorphic to either $(\E_{\varpi},\Omega_{c}^{\varpi})$ or the union of $(\E_{\alpha},\Omega_{c}^{\alpha})$ and $(\E_{\beta},\Omega_{c}^{\beta})$ with one fiber $\X$ identified.

Let $\mathfrak{c}^{\varpi}: \E_{\varpi}\to \Y$ be the inclusion of the fiber at $[z_{1},z_{2}]\neq [0,1]$ or $[1,0] \in \com P^{1}$, and let $\sigma\in H^{sec}(|\E_{\varpi}|,\inte)$. The main reason for constructing $\Y$ is to use the moduli space $\M_{0,1}(\Y,\mathfrak{c}^{\beta}_{*}\sigma)$ and its compactification in the proof of Proposition \ref{compGW}.

\subsection{Proof of Proposition \ref{compGW}}
From now on we use cap product ``$\smallfrown$'' between homology and cohomology instead of integral in the definition of Gromov-Witten invariants to emphasize the role of evaluation maps. Note that this is actually how we understand integrals over Kuranishi spaces. 

For twisted sectors $\X_{(g_{1})}, \X_{(g_{2})},\X_{(g)}$ of $\X$, we write $ev_{*}\M_{\sigma_{\alpha},(g_{1})}^{v}$, $ev_{*}\M_{\sigma_{\beta},(g_{2})}^{v}$, and $Ev_{*}\M_{\X,A,(g_{1}^{-1},g,g_{2}^{-1})}^{v}$ for the virtual fundamental classes associated to the moduli spaces
$\Mbar_{0,1}(\E_{\alpha},\sigma_{\alpha},(g_{1}))$,
$\Mbar_{0,1}(\E_{\beta},\sigma_{\beta},(g_{2}))$, and
$\Mbar_{0,3}(\X,A,(g_{1}^{-1},g,g_{2}^{-1}))$ respectively. Now the left-hand side of (\ref{compGWid}) may be written as
\begin{eqnarray*}
\negthickspace  & & \negthickspace\negthickspace\negthickspace\negthickspace\negthickspace\negthickspace\sum_{\sigma_{\alpha}\sharp\sigma_{b}+\iota^{\varpi}_*A=\sigma} \sum_{j,k} (ev_{*}\M_{\sigma_{\alpha},(g_{1})}^{v} \smallfrown \iota^{\alpha}_*f_j) \cdot (ev_{*}\M_{\sigma_{\beta},(g_{2})}^{v}\smallfrown\iota^{\beta}_*f_k) \cdot (Ev_{*}\M_{\X,A,(g_{1}^{-1},g,g_{2}^{-1})}^{v})\smallfrown (f^j \otimes f \otimes f^k) \\
\negthickspace = & &  \negthickspace\negthickspace\negthickspace\negthickspace\negthickspace\negthickspace \sum_{\sigma_{\alpha}\sharp\sigma_{\beta}+\iota^{\varpi}_*A=\sigma} \sum_{j,k}  (ev_{*}\M_{\sigma_{\alpha},(g_{1})}^{v} \otimes  ev_{*}\M_{\sigma_{\beta},(g_{2})}^{v} \otimes Ev_{*}\M_{\X,A,(g_{1}^{-1},g,g_{2}^{-1})}^{v}) \smallfrown (\iota^{\alpha}_*f_j  \otimes  \iota^{\beta}_*f_k  \otimes f^j \otimes f \otimes f^k)\\
\negthickspace =& & \negthickspace \negthickspace\negthickspace\negthickspace\negthickspace\negthickspace\negthickspace\sum_{\sigma_{\alpha}\sharp\sigma_{\beta}+\iota^{\varpi}_*A=\sigma} \negthickspace\negthickspace\negthickspace\negthickspace (ev_{*}\M_{\sigma_{\alpha},(g_{1})}^{v} \otimes  ev_{*}\M_{\sigma_{\beta},(g_{2})}^{v} \otimes Ev_{*}\M_{\X,A,(g_{1}^{-1},g,g_{2}^{-1})}^{v})\smallfrown ((\sum_{j}   \iota^{\alpha}_*f_j  \otimes f^j) \otimes (\sum_{k} \iota^{\beta}_*f_k  \otimes  f^k)\otimes f)\\
\negthickspace =& &  \negthickspace\negthickspace\negthickspace\negthickspace\negthickspace\negthickspace\sum_{\sigma_{\alpha}\sharp\sigma_{\beta}+\iota^{\varpi}_*A=\sigma} (ev_{*}\M_{\sigma_{\alpha},(g_{1})}^{v} \otimes  ev_{*}\M_{\sigma_{\beta},(g_{2})}^{v} \otimes Ev_{*}\M_{\X,A,(g_{1}^{-1},g,g_{2}^{-1})}^{v}) \smallfrown (PD_{1}(\Delta_{\alpha}) \otimes PD_{2}(\Delta_{\beta}) \otimes f).
\end{eqnarray*}

In the last term, $PD_{1}$ and $PD_{2}$ denote the  orbifold Poincar\'{e} duals in $\E_{\alpha,(g_{1})}\times \X_{(g_{1}^{-1})}$ and $\E_{\beta,(g_{2})}\times \X_{(g_{2}^{-1})}$ respectively. Let $\mathcal{I}:{I}\X\to {I}\X$ be the involution which send $(x, (g)) \in\X_{(g)}$ to $(x,(g^{-1}))\in\X_{(g^{-1})}$. Then $\Delta_{\alpha}$, $\Delta_{\beta}$ are defined to be
$$\Delta_{\alpha}:=\{(p,x)\in \E_{\alpha,(g_{1})}\times \X_{(g_{1}^{-1})}|p=\mathcal{I}\circ \iota^{\alpha}(x)\},$$
$$\Delta_{\beta}:=\{(p,x)\in \E_{\beta,(g_{2})}\times \X_{(g_{2}^{-1})}|p=\mathcal{I}\circ \iota^{\beta}(x)\}.$$
Let $$\M^{\Delta}(\sigma_{\alpha},\sigma_{\beta}, A)$$ be defined as
\begin{equation*}
\begin{split}
\{(\tau_{\alpha},\tau_{\X},\tau_{\beta})\in\Mbar_{0,1}(\E_{\alpha},\iota^\alpha_{*}\sigma_{\alpha}) \times  &\Mbar_{0,3}(\X,A)  \times \Mbar_{0,1}(\E_{\beta},\iota^\beta_{*}\sigma_{\beta})\\
&|\,\, ev(\tau_{\alpha})=\mathcal{I}\circ\iota^{\alpha}( ev_{0}(\tau_{\X})), \ ev(\tau_{\beta})=\mathcal{I}\circ\iota^{\beta} (ev_{2}(\tau_{\X}))\}. 
\end{split}
\end{equation*}
Note that $\M^{\Delta}(\sigma_{\alpha},\sigma_{\beta}, A)$ is a fiber product of moduli spaces. By \cite[Remark A1.44(1)]{FOOObook}, we may assume the evaluation maps are weakly subversive. Then the fiber product Kuranishi structure on $\M^{\Delta}(\sigma_{\alpha},\sigma_{\beta}, A)$ is well-defined. Let $ev_{*}[\M^{\Delta} (\sigma_{\alpha},\sigma_{\beta}, A)]^{v}\in H_{*}(\X_{(g)},\ration)$ be the virtual fundamental class defined by this Kuranishi space together with the evaluation map at the only marked point where we do not take fiber product. 

\begin{lem}\label{diagonal} 
\begin{eqnarray*}
& & (ev_{*}\M_{\sigma_{\alpha},(g_{1})}^{v} \otimes  ev_{*}\M_{\sigma_{\beta},(g_{2})}^{v} \otimes Ev_{*}\M_{\X,A,(g_{1}^{-1},g,g_{2}^{-1})}^{v}) \ \ \smallfrown \ \ (PD_{1}(\Delta_{\alpha}) \otimes PD_{2}(\Delta_{\beta}) \otimes f)\\
 = & & ev_{*}[\M^{\Delta} (\sigma_{\alpha},\sigma_{\beta}, A)]^{v}\smallfrown f.
\end{eqnarray*} 
\end{lem}
\begin{proof} 
Denote the Kuranishi charts of $\Mbar_{0,1}(\E_{\alpha},\iota^{\alpha}_{*}\sigma_{\alpha})$, $  \Mbar_{0,3}(\X,A)$ and  $\Mbar_{0,1}(\E_{\beta},\iota^{\beta}_{*}\sigma_{\beta})$ by $V_{\alpha}$, $V_{\X}$ and $V_{\beta}$ respectively. Let $\ms_{\alpha},\ms_{\beta}$ and $\ms_{\X}$ be their transversal multi-sections. The Kuranishi chart of the fiber product is by definition:
$$V^{\Delta}:=\{(\eta_{\alpha},\eta_{\X},\eta_{\beta})\in V_{\alpha}\times V_{\X} \times V_{\beta}|ev(\eta_{\alpha})=\mathcal{I}\circ\iota^{\alpha}( ev_{0}(\eta_{\X})), \ ev(\eta_{\beta})=\mathcal{I}\circ\iota^{\beta} (ev_{2}(\eta_{\X}))\}.$$
Moreover $\ms_{\Delta}:=(\ms_{\alpha},\ms_{\beta},\ms_{\X})|_{V^{\Delta}}$ is a transversal multi-section of the fiber product Kuranishi space.

Let $\cap_{\E^{2}\X^{3}}$ be the intersection product in $\E_{\alpha,(g_{1})}\times \E_{\beta,(g_{2})}\times \X_{(g_{1}^{-1})}\times \X_{(g_{2}^{-1})}\times \X_{(g)}$, $\cap_{\X_{(g)}}$ the intersection product in $\X_{(g)}$, $\hat{f}$ the Poincar\'{e} dual of $f$ in $\X_{(g)}$. Then we have 
\begin{equation}\label{LHSgraphId}
\begin{aligned}
& & (ev_{*}\M_{\sigma_{\alpha},(g_{1})}^{v} \otimes  ev_{*}\M_{\sigma_{\beta},(g_{2})}^{v} \otimes Ev_{*}\M_{\X,A,(g_{1}^{-1},g,g_{2}^{-1})}^{v}) \ \ \smallfrown \ \ (PD_{1}(\Delta_{\alpha}) \otimes PD_{2}(\Delta_{\beta}) \otimes f)\\
= & & ev(\ms^{-1}_{\alpha}(0))\times ev_{0}(\ms^{-1}_{\X}(0))\times ev(\ms^{-1}_{\beta}(0)) \times ev_{2}^{-1}(\ms^{-1}_{\X}(0))\times ev_{1}(\ms^{-1}_{\X}(0))\ \ \ \cap_{\E^{2}\X^{3}}\ \ \ (\Delta_{\alpha} \times \Delta_{\beta} \times \hat{f}).
\end{aligned}
\end{equation}
On the other hand
\begin{equation}\label{RHSgraphId}
 ev_{*}[\M^{\Delta} (\sigma_{\alpha},\sigma_{\beta}, A)]^{v}\smallfrown f\ \
=  \ \ ev_{1}(\ms_{\Delta}^{-1}(0))\ \ \ \cap_{\X_{(g)}} \ \ \ \hat{f}.
\end{equation}
Note that $$(ev(\eta_{\alpha}),ev_{0}(\eta_{\X}), ev(\eta_{\beta}) , ev_{2}(\eta_{\X}), ev_{1}(\eta_{\X}))\in \Delta_{\alpha} \times \Delta_{\beta} \times \hat{f},$$ if and only if 
$$ev(\eta_{\alpha})=\mathcal{I}\circ ev_{0}(\eta_{\X}),\ ev(\eta_{\beta})=\mathcal{I}\circ ev_{2}(\eta_{\X})\ \text{ and } \ ev_{1}(\eta_{\X})\in \hat{f},$$ which is equivalent to $$(\eta_{\alpha},\eta_{\X},\eta_{\beta})\in V^{\Delta}\ \text{ and }  \ ev_{1}(\eta_{\alpha},\eta_{\X},\eta_{\beta})\in \hat{f}.$$
Thus
\begin{equation*}
\begin{split}
ev(\ms^{-1}_{\alpha}(0))\times ev_{0}(\ms^{-1}_{\X}(0))\times ev(\ms^{-1}_{\beta}(0)) \times ev_{2}^{-1}(\ms^{-1}_{\X}(0))\times ev_{1}(\ms^{-1}_{\X}(0))\ \ \cap_{\E^{2}\X^{3}}\ \ (\Delta_{\alpha} \times \Delta_{\beta}\times \hat{f})\\
 = ev_{1}(\ms_{\Delta}^{-1}(0))\ \ \cap_{\X} \ \ \hat{f}.
\end{split}
\end{equation*}
Together with (\ref{LHSgraphId}) and (\ref{RHSgraphId}), the proof is complete.
\end{proof}

Now we shall put the moduli spaces $\M^{\Delta}(\sigma_{\alpha},\sigma_{\beta}, A)$ and $\Mbar_{0,1}(\E_{\varpi},\sigma,(g))$ into a moduli space associated to the big orbifold $\Y$. Let  $i^{nd}: \X \to \Y$ be the inclusion of the fiber of the orbifiber bundle $\Y\to \mathcal{T}$ at the nodal point of the degenerate fiber $\com P^{1} \sharp \com P^{1}$. Recall that $\mathfrak{c}^{\varpi}: \E_{\varpi} \to \Y$ is the inclusion of a fiber over $[z_{1},z_{2}]\neq [0,1]\ or\ [1,0]$. We need to construct a suitable Kuranishi structure on $\Mbar_{0,1}(\Y,\mathfrak{c}^{\varpi}_{*}\sigma)$, and a suitable system of multi-sections.

 Denote $\M^{\E_{\varpi}}(\Y):=\{\tau\in\Mbar_{0,1}(\Y,\mathfrak{c}^{\varpi}_{*}\sigma)|pr \circ ev(\tau)=[0,1]\}$. Because of the choice of the almost complex structures, this space can be identified with $\Mbar_{0,1}(\E_{\varpi},\sigma,(g))$. Choose a Kuranishi structure  on $\Mbar_{0,1}(\E_{\varpi},\sigma,(g))$ whose obstruction bundle is contained in the vertical direction as in Section \ref{secCurvHamBundle}, denoted its induced Kuranishi structure on $\M^{\E_{\varpi}}(\Y)$ by $\{(\check{V}^{\varpi}_{i},E^{\varpi}_{i}, \Gamma^{\varpi}_{i},\psi^{\varpi}_{i},s^{\varpi}_{i})|i \in \mathfrak A^{\varpi}\}$. Let $\ms_{\varpi}$ be a system of multi-sections of this Kuranishi structure. Put $\check{V}^{\varpi}=\cup_{i} \check{V}^{\varpi}_{i}$.
 
Let 
$$\M^{\Delta}(\Y):= \negthickspace\negthickspace\bigcup_{\sigma_{\alpha}\sharp\sigma_{\beta}+\iota^{\varpi}_*A=\sigma}\Mbar_{0,1}(\Y,\mathfrak{c}^{\alpha}_{*}\sigma_{\alpha}) {}_{ev} \!\times_{ev_{0}} \Mbar_{0,3}(\Y,i^{nd}_{*}A) {}_{ev_{2}}\!\times_{ev}  \Mbar_{0,1}(\Y,\mathfrak{c}^{\beta}_{*}\sigma_{\beta})\subset \Mbar_{0,1}(\Y,\mathfrak{c}^{\varpi}_{*}\sigma).$$ 
It is easy to see that $\M^{\Delta}(\Y)$ can be identified with $ \bigcup_{\sigma_{\alpha}\sharp\sigma_{\beta}+\iota^{\varpi}_*A=\sigma} \M^{\Delta}(\sigma_{\alpha},\sigma_{\beta}, A)$. The fiber product Kuranishi structures on $\M^{\Delta}(\sigma_{\alpha},\sigma_{\beta}, A)$ induce a Kuranishi structure on $\M^{\Delta}(\Y)$. We denote this Kuranishi structure by
$\{(\check{V}^{\Delta}_{j},E^{\Delta}_{j},
\Gamma^{\Delta}_{j},\psi^{\Delta}_{j},s^{\Delta}_{j}|j \in \mathfrak A^{\Delta}\}$, 
and let $\ms_{\Delta}$ be a system of multi-sections of it. Denote $\check{V}^{\Delta}=\cup_{j} \check{V}^{\Delta}_{j}$.

\begin{lem}\label{virtualY}
The moduli space $\Mbar_{0,1}(\Y,\mathfrak{c}^{\varpi}_{*}\sigma)$ has a Kuranishi structure 
$\{(V_{k},E_{k},
\Gamma_{k},\psi_{k},s_{k})|k \in \mathfrak A\}$
 such that:
\begin{enumerate}
\item \label{Xcompatible} if $s_{k}^{-1}(0)\cap \M^{\Delta}(\Y)\neq \emptyset$, then $V_{k}=D\times D\times \check{V}^{\Delta}_{j}$ and $E_{k}=D\times D\times E^{\Delta}_{j}$ for some $j \in \mathfrak A^{\Delta}$;
\item \label{Ecompatible} if $s_{k}^{-1}(0)\cap \M^{\E_{\varpi}}(\Y)\neq \emptyset$, then $V_{k}=D\times \check{V}^{\varpi}_{i}$  and $E_{k}=D\times E^{\varpi}_{i}$ for some $i \in \mathfrak A^{\varpi}$.
\end{enumerate}
 The above Kuranishi structure has a system of multi-sections $\ms$ such that $\ms|_{\check{V}^{\Delta}}=\ms_{\Delta}$ and $\ms|_{\check{V}^{\varpi}}=\ms_{\varpi}$.
\end{lem}
\begin{proof}This is because $\M^{\Delta}(\Y)$ and $\M^{\E_{\varpi}}(\Y)$ are closed and disjoint in $\Mbar_{0,1}(\Y,\mathfrak{c}^{\varpi}_{*}\sigma)$, so the Kuranishi charts over each can be chosen independently.
\end{proof}

\begin{lem}\label{identity1}
$\sum_{\sigma_{\alpha}\sharp\sigma_{\beta}+\iota^{\varpi}_*A=\sigma} ev_{*}[\M^{\Delta}(\sigma_{\alpha},\sigma_{\beta}, A)]^{v} \smallfrown f=ev_{*}[\Mbar_{0,1}(\Y,\mathfrak{c}^{\varpi}_{*}\sigma)]^{v}  \smallfrown i^{nd}_{*}f.$
\end{lem}
\begin{proof}
Let $\hat{f}:S\to \X_{(g)}$ be a representative of the Poincar\'{e} dual of $f$ in $\X_{(g)}$, then the identity becomes:
$$\sum_{\sigma_{\alpha}\sharp\sigma_{\beta}+\iota^{\varpi}_*A=\sigma} ev_{*}[\M^{\Delta}(\sigma_{\alpha},\sigma_{\beta}, A)]^{v} \cap_{\X_{(g)}} \hat{f}=ev_{*}[\Mbar_{0,1}(\Y,\mathfrak{c}^{\varpi}_{*}\sigma)]^{v}  \cap_{\Y_{(g)}} i^{nd}_{*}\hat{f}.$$
Here $\cap_{\X_{(g)}}$ and $\cap_{\Y_{(g)}}$ are the intersection products in $\X_{(g)}$ and $\Y_{(g)}$ respectively.

An intersection on the right hand side is a pair $(x,t)$ for $x\in\ms^{-1}(0)$, $t\in S$ such that $ev(x)=i^{nd}\circ\hat{f}(t)$. Then $ev(x)\in i^{nd}(\X_{(g)})$. By Lemma \ref{virtualY}, $x\in \ms_{\Delta}^{-1}(0)$. Therefore we have a corresponding intersection on the left hand side.

Conversely it is obvious that any  intersection on the left hand side gives an intersection on the right hand side.

Moreover the 1-1 correspondence preserves sign. Thus the identity holds.
\end{proof}

\begin{lem}\label{identity2}
$ev_{*}[\Mbar_{0,1}(\Y,\mathfrak{c}^{\varpi}_{*}\sigma,(g))]^{v}  \smallfrown i^{nd}_{*}f=ev_{*}[\Mbar_{0,1}(\E_{\varpi},\sigma,(g))]^{v} \smallfrown \iota^{\varpi}_*f.$
\end{lem}
\begin{proof}
The proof is similar to that of Lemma \ref{identity1} using Lemma \ref{virtualY}. 
\end{proof}

By Lemma \ref{diagonal}, Lemma \ref{identity1} and Lemma \ref{identity2}, we find that the left-hand side of (\ref{compGWid}) is equal to
\begin{equation*}
\begin{split}
\sum_{\sigma_{\alpha}\sharp\sigma_{\beta}+\iota^{\varpi}_*A=\sigma} ev_{*}[\M^{\Delta}(\sigma_{\alpha},\sigma_{\beta}, A)]^{v} \smallfrown f
& =ev_{*}[\Mbar_{0,1}(\Y,\mathfrak{c}^{\varpi}_{*}\sigma,(g))]^{v}  \smallfrown i^{nd}_{*}f\\
& =ev_{*}[\Mbar_{0,1}(\E_{\varpi},\sigma,(g))]^{v} \smallfrown \iota^{\varpi}_*f,
\end{split}
\end{equation*}
which is the right-hand side of (\ref{compGWid}). The proof of Proposition \ref{compGW} is thus complete.

\section{Appendix: Cutting and Glueing of Groupoids}
Let $\gG$ be a groupoid, $X=|\gG|$ the underlying topological space, $Y$ a subset of $X$. Then the restriction of $\gG$ to $Y$, denoted by $\gG|_{Y}$, is defined as follows:  
\begin{eqnarray*}
Ob(\gG|_{Y}) && := \{x\in Ob(\gG) \,\, | \,\,[x]\in Y\}\\
Mor(\gG|_{Y}) &&:= \{\xymatrix{x\ar[r]^{g} & y}\in Mor(\gG) \,\, | \,\,[x]=[y]\in Y\}
\end{eqnarray*}

Assume that $X$ is a manifold, $X=Y_{1} \cup Y_{2}$, $Z:=Y_{1} \cap Y_{2}=\partial Y_{1}=\partial Y_{2}$ is a submanifold of $X$ of codimension one. Then $\gG|_{Y_{1}}$ and $\gG|_{Y_{2}}$ are called a \emph{cutting} of $\gG$ along $Z$.

Let $\gG_{1}$ and $\gG_{2}$ be two groupoids with boundaries. Namely for $i=1,2$, $Ob(\gG_{i})$ and $Mor(\gG_{i})$ are manifolds with boundaries. Moreover, $\partial Mor(\gG_{i})\twoarrows \partial Ob(\gG_{i})$ is a full subgroupoid of $\gG_{i}$. Let $\phi=(\phi_{0},\phi_{1})$ be a groupoid isomorphism from $\partial Mor(\gG_{1})\twoarrows \partial Ob(\gG_{1})$ to $\partial Mor(\gG_{2})\twoarrows \partial Ob(\gG_{2})$. Then the {\em glueing} of $\gG_{1}$ and $\gG_{2}$ with respect to $\phi$, denoted by $\gG_{1} \sqcup_{\phi} \gG_{2}$ is defined by:
 \begin{eqnarray*}
Ob(\gG_{1} \sqcup_{\phi} \gG_{2}) && \negthickspace\negthickspace\negthickspace\negthickspace\negthickspace\negthickspace := Ob(\gG_{1})\sqcup Ob(\gG_{2})/\<x_{1}\sim x_{2} \ \text{i.f.f.}\ \phi_{0}(x_{1})=x_{2} \ for\ x_{i}\in\partial Ob(\gG_{i}),\ i=1,2\>,\\
Mor(\gG_{1} \sqcup_{\phi} \gG_{2}) && \negthickspace\negthickspace\negthickspace\negthickspace\negthickspace\negthickspace:= Mor(\gG_{1})\sqcup Mor(\gG_{2})/\<r_{1}\sim r_{2} \ \text{i.f.f.}\ \phi_{0}(r_{1})=r_{2} \ for\ r_{i}\in\partial Mor(\gG_{i}), \ i=1,2\>.
\end{eqnarray*}

It is straightforward to verify that this defines a Lie groupoid.


\begin{thebibliography}{10}
\bibitem[ALR]{ALR} A. Adem, J. Leida, Y. Ruan. {\em Orbifolds and stringy topology}, Cambridge Tracts in. Mathematics, 171, Cambridge University Press, Cambridge, 2007.

\bibitem[AGV]{AGV}
D. Abramovich, T. Graber, A. Vistoli, {\em Algebraic orbifold quantum products}, in `Orbifolds in mathematics and physics (Madison, WI, 2001)', 1--24, Amer. Math. Soc., 2002.

\bibitem[BX]{BX} K. Behrend, P. Xu, {\em Differentiable stacks and gerbes}, J. Symplectic Geom. 9 (2011), no. 3, 285--341.

\bibitem[CLSZ]{CLSZ} B. Chen, A.-M. Li, S. Sun, G. Zhao, {\em Relative orbifold Gromov-Witten theory and degeneration formula}, arXiv:1110.6803.

\bibitem[CR1]{CR1} W. Chen, Y. Ruan, {\em A New Cohomology Theory for Orbifold}, Comm. Math. Phys. 248 (2004), no. 1, 1--31.

\bibitem[CR2]{CR2} W. Chen, Y. Ruan, {\em Orbifold Gromov-Witten theory}, in `Orbifolds in mathematics and physics (Madison, WI, 2001)', 25--85, Contemp. Math., 310, Amer. Math. Soc., 2002.

\bibitem[FOOO1]{FOOO00} K. Fukaya, Y.-G. Oh, H. Ohta, K. Ono, {\em Lagrangian intersection Floer theory-anomaly and obstruction,} Kyoto University preprint, 2000.


\bibitem[FOOO2]{FOOObookpre} K. Fukaya, Y.-G. Oh, H. Ohta, K. Ono, {\em Lagrangian intersection Floer theory-anomaly and obstruction,} expanded version of [FOOO1], 2006 \& 2007.

\bibitem[FOOO3]{FOOObook} K. Fukaya, Y.-G. Oh, H. Ohta,  K. Ono, {\em Lagrangian intersection Floer theory-anomaly and obstruction}, AMS/IP Studies in Advanced Mathematics, vol 46, AMS/International Press.

\bibitem[FOOO4]{FOOOtoric1} K. Fukaya, Y.-G. Oh, H. Ohta, K. Ono, {\em Lagrangian Floer theory on compact toric manifolds, I,} Duke Math. J. 151 (2010), no. 1, 23--174.

\bibitem[FOOO5]{FOOOtoric2} K. Fukaya, Y.-G. Oh, H. Ohta, K. Ono, {\em Lagrangian Floer theory on compact toric manifolds II : bulk deformations,} Selecta Math. (N.S.) 17 (2011), no. 3, 609--711, arXiv:0810.5654.

\bibitem[FOOO6]{FOOO-model08} K. Fukaya, Y.-G. Oh, H. Ohta, K. Ono, {\em Canonical models of filtered $A_\infty$-algebras and Morse complexes}, in `New Perspectives and Challenges in Symplectic Field Theory', 201--227, CRM Proc. Lecture Notes, 49, Amer. Math. Soc. (2009), arXiv:0812.1963.

\bibitem[FOOO7]{FOOO-integers} K. Fukaya, Y.-G. Oh, H. Ohta, K. Ono, {\em Lagrangian Floer theory over integers: spherically positive symplectic manifolds}, arXiv:1105.5124v1.

\bibitem[FOOO8]{FOOOtoric3} K. Fukaya, Y.-G. Oh, H. Ohta, K. Ono, {\em Lagrangian Floer theory and Mirror symmetry on compact toric manifolds}, arXiv:1009.1648v1.

\bibitem[FO]{FO} K. Fukaya, K. Ono, {\em Arnold conjecture and Gromov-Witten invariant}, Topology 38 (1999), no. 5, 933--1048.

\bibitem[F]{Fuk} K. Fukaya, {\em Cyclic symmetry and adic convergence in Lagrangian Floer theory},  Kyoto J. Math. 50 (2010), no. 3, 521--590.

\bibitem[G]{gor} R. M. Goresky, {\em Triangulation of stratified objects}, Proc. Amer. Math. Soc. 72 (1978), no. 1, 193--200. 

\bibitem[Hei]{Hein} J. Heinloth, {\em Notes on differentiable stacks}, Mathematisches Institut, Georg-August-Universit\"at G\"ottingen: Seminars Winter Term 2004/2005, {2005}, {1--32}.

\bibitem[Hep]{Hep} R. Hepworth, {\em Vector fields and flows on differentiable stacks}, Theory Appl. Categ. 22 (2009), 542--587. 

\bibitem[Jo]{Jo} D. Joyce, {\em Kuranishi bordism and Kuranishi homology}, arXiv:0707.3572.

\bibitem[L]{Le} E. Lerman, {\em Orbifolds as stacks?}, Enseign. Math. (2), 56(3-4):315--363, 2010, arXiv:0806.4160v2.
 
\bibitem[LM]{ML} E. Lerman, A. Malkin, {\em Hamiltonian group actions on symplectic Deligne-Mumford stacks and toric orbifolds}, arXiv:0907.2891.

\bibitem[LT]{LT} E. Lerman and S. Tolman, {\em Hamiltonian torus actions on symplectic orbifolds and toric varieties}, Trans. Amer. Math. Soc. 349 (1997), 4201--4230.

\bibitem[LiTi]{LiTi} J. Li,  G. Tian, {\em Virtual moduli cycles and Gromov-Witten invariants for general symplectic manifolds}, in `Topics in symplectic $4$-manifolds', 47--83, International Press, 1998.

\bibitem[Liu]{Liu} C.-C. Liu, {\em Moduli of J-holomorphic curves with Lagrangian boundary conditions and open Gromov-Witten invariants for an $S^{1}$-equivariant pair}, math.SG/0210257.
 
\bibitem[LU]{LU} E. Lupercio, B. Uribe, {\em Gerbes over Orbifolds and Twisted K-Theory}, Comm. Math. Phys. 245 (2004), no. 3, 449--489.

\bibitem[Ma]{Mack} Kirill C. H. Mackenzie, {\em General Theory of Lie Groupoids and Lie Algebroids},  London Mathematical Society Lecture Note Series, 213. Cambridge University Press, Cambridge, 2005. xxxviii+501 pp.

\bibitem[M]{M} D. McDuff, {\em  Quantum homology of fibrations over $S\sp 2$}, Internat. J. Math. 11 (2000), no. 5, 665--721.

\bibitem[MS]{MS} D. McDuff, D. Salamon, {\em $J$-holomorphic curves and Symplectic Topology},  American Mathematical Society Colloquium Publications, 52. American Mathematical Society, Providence, RI, 2004. xii+669 pp.

\bibitem[MT]{MT} D. McDuff, S. Tolman, {\em Topological properties of Hamiltonian circle actions},  IMRP Int. Math. Res. Pap. 2006, 72826, 1--77. 

\bibitem[MP1]{MP1} I. Moerdijk,  D. A. Pronk, {\em Orbifolds, Sheaves and Groupoids}, $K$-Theory 12 (1997), no. 1, 3--21.

\bibitem[MP2]{MP2} I. Moerdijk, D. A. Pronk, {\em Simplicial Cohomology of Orbifolds}, Indag. Math. (N.S.) 10 (1999), no. 2, 269--293.

\bibitem[P]{Pronkthesis} D. A. Pronk, {\em Groupoid Representations for Sheaves on Orbifolds}, Ph.D. thesis, Utrecht 1995.

\bibitem[Ro]{Ro} M. Romagny, {\em Group actions on stacks and applications}, Michigan Math. J. 53 (2005), no. 1, 209--236.

\bibitem[Sa]{Sa} H. Sakai, {\em The symplectic Deligne-Mumford stack associated to a stacky polytope}, to appear in Results in Mathematics, arXiv:1009.3547. 

\bibitem[Sa1]{sa56} I. Satake, {\em On a generalization of the notion of  manifold},  Proc. Nat. Acad. Sci. U.S.A. 42 (1956), 359--363. 

\bibitem[Sa2]{sa57} I. Satake, {\em The Gauss-Bonnet theorem for V-manifolds}, J. Math. Soc. Japan 9 1957 464--492.

\bibitem[Se]{Se} P. Seidel, {\em $\pi_1$  of symplectic automorphism groups and invertibles in quantum cohomology rings},  Geom. Funct. Anal. 7 (1997), no. 6, 1046--1095. 

\bibitem[TW]{TW} H.-H. Tseng, D. Wang, {\em Seidel representation and quantum cohomology of symplectic toric orbifolds}, in preparation. 
\end{thebibliography}
\end{document}